\documentclass[reqno,a4paper,12pt]{amsart} 
    
    \usepackage{amsmath,amscd,amsfonts,amssymb}
    \usepackage{mathrsfs,dsfont}
    \usepackage{color}
    \usepackage{mathtools}
    \usepackage{hyperref}
    \usepackage{tikz-cd}

    \usepackage{tikz}
    \usetikzlibrary{decorations.pathreplacing}
    
    \numberwithin{equation}{section}
    \numberwithin{figure}{section}
    
    \def\R{\mathbb{R}}

    \def\Z{\mathbb{Z}}

    \newcommand\Scal{\mathcal{S}}
    \newcommand\Wcal{\mathcal{W}}

    \def\Lam{\Lambda}

    \renewcommand\le{\leqslant}
    
    \renewcommand\leq{\leqslant}
    \renewcommand\geq{\geqslant}

    \newcommand\ord{\operatorname{ord}}
    \renewcommand{\ae}{{\mathrm{a.e.}}}

    \newcommand{\Tile}{\operatorname{Tile}}

    \newcommand{\Mod}[1]{\ (\mathrm{mod}\ #1)}

    \theoremstyle{plain}
    \newtheorem{thm}{Theorem}[section]
    \newtheorem{theorem}[thm]{Theorem}
    
    \newtheorem{lemma}[thm]{Lemma}
    \newtheorem{corollary}[thm]{Corollary}
    
    \newtheorem{proposition}[thm]{Proposition}
    \newtheorem{problem}[thm]{Problem}
    \newtheorem{question}[thm]{Question}
    
    \newtheorem{conjecture}[thm]{Conjecture}
    
    \newtheorem*{claim*}{Claim}
    
    \newcommand{\thmref}[1]{Theorem~\ref{#1}}
    \newcommand{\secref}[1]{Section~\ref{#1}}
    
    \newcommand{\lemref}[1]{Lemma~\ref{#1}}
    \newcommand{\defref}[1]{Definition~\ref{#1}}

    \newcommand{\corref}[1]{Corollary~\ref{#1}}
    \newcommand{\conjref}[1]{Conjecture~\ref{#1}}

    \theoremstyle{definition}
    \newtheorem{definition}[thm]{Definition}
    \newtheorem*{definition*}{Definition}
    \newtheorem*{remarks*}{Remarks}
    \newtheorem*{remark*}{Remark}
    \newtheorem{remark}[thm]{Remark}
    \newtheorem{example}[thm]{Example}

    \newenvironment{enumerate-math}
    {\begin{enumerate}
    		\addtolength{\itemsep}{5pt}
    		}
    	{\end{enumerate}}

    \newenvironment{enumerate-text}
    {\begin{enumerate}
    		\addtolength{\itemsep}{5pt}
    		}
    	{\end{enumerate}}

\begin{document}

    	\title{A counterexample to the periodic tiling conjecture}
    
    	\author{Rachel Greenfeld}
    	\address{School of Mathematics, Institute for Advanced Study, Princeton, NJ 08540.}
    	\email{greenfeld.math@gmail.com}
    	\author{Terence Tao}
    	\address{UCLA Department of Mathematics, Los Angeles, CA 90095-1555.}
    	\email{tao@math.ucla.edu}

    	\subjclass[05B45, 52C22, 52C23, 52C25]{05B45, 52C22, 52C23, 52C25}
    	\date{}
    	
    	
    \begin{abstract}   The periodic tiling conjecture asserts that any finite subset of a lattice $\Z^d$ that tiles that lattice by translations, in fact tiles periodically. In this work we disprove this conjecture for sufficiently large $d$, which also implies a disproof of the corresponding conjecture for Euclidean spaces $\R^d$.  In fact, we also obtain a counterexample in a group of the form $\Z^2 \times G_0$ for some finite abelian $2$-group $G_0$.  Our methods rely on encoding a ``Sudoku puzzle'' whose rows and other non-horizontal lines are constrained to lie in a certain class of ``$2$-adically structured functions'', in terms of certain functional equations that can be encoded in turn as a single tiling equation, and then demonstrating that solutions to this Sudoku puzzle exist, but are all non-periodic. 
    	\end{abstract}
    	
    	\maketitle

     \section{introduction}


     In 1960, Hao Wang \cite{wang60,wang} studied  the problem of tiling the plane by translated copies of finitely many squares a color attached to each side of each of them, also known as \emph{Wang squares},  where one square lies next to another only if the colors of common edges match. This is a variant of Hilbert's famous {\it Entscheidungsproblem}. Wang conjectured that if a set of such squares admits a tiling of the plane, then it also admits a periodic tiling. Wang's conjecture was disproved by  Berger \cite{Ber,Ber-thesis}, who constructed an  \emph{aperiodic} set of $20,426$  Wang squares, i.e., the set of squares admits tilings but none of these tilings is periodic. Over the years, many more constructions of aperiodic translational tilings (including several not based on Wang squares) were established, with smaller tile-sets (see, e.g., \cite[Table 1]{GT2}).  In this paper we construct an aperiodic translational tiling with a \emph{single tile} in $\Z^2\times G_0$, for a certain finite abelian group $G_0$. As a consequence, we disprove the celebrated  ``periodic tiling conjecture''.  Our methods are based on encoding a ``Sudoku puzzle'' rather than a Wang tiling problem.


        \subsection{The periodic tiling conjecture}
    
        Let $G = (G,+)$ be a discrete abelian group.  If $A, F$ are subsets of $G$, we write $A \oplus F = G$ if the translates $a+F \coloneqq \{a+f: f \in F \}$ of $F$ by elements $a$ of $A$ form a partition of $G$. If this occurs, we say that $F$ \emph{tiles $G$ (by translations)}, and that $A$ is a \emph{tiling set of $G$ by $F$}.  The tiling set $A$ is said to be \emph{periodic} if it is the finite union of cosets of a finite index subgroup of $G$.  
        We will refer to $A \oplus F = G$ as a \emph{tiling equation}, and think of $F,G$ as being given and $A \subset G$ as being an unknown.   We say that the tiling equation $A \oplus F = G$ is \emph{aperiodic} if there exist solutions $A \subset G$ to the tiling equation $A \oplus F = G$, but none of these solutions are periodic. 
        \begin{remark}
            We caution that in the aperiodic order literature the term ``periodic'' instead refers to sets that are unions of cosets of some non-trivial cyclic subgroup of $G$; in our notation, we would refer to such sets as being \emph{one-periodic}.  For instance, if $G = \Z^2$ and $A$ was an arbitrary subset of $\Z$, then $A \times \Z$ would be one-periodic, but not necessarily periodic in the sense adopted in this paper.  The notion of an aperiodic tiling  is similarly modified in the aperiodic order literature, and the notion of aperiodicity used here is sometimes referred to as ``weak aperiodicity''. For tilings in dimensions $d\leq 2$ the two notions of aperiodicity coincide \cite[Theorem 3.7.1]{grunbaum-shephard}.
        \end{remark}
        A well-known conjecture in the area is the periodic tiling conjecture:

        \begin{conjecture}[Discrete periodic tiling conjecture]\label{ptc}\cite{stein, grunbaum-shephard,LW}  Let $F$ be a finite non-empty subset of a finitely generated discrete abelian group $G$.  Then the tiling equation $A \oplus F = G$ is not aperiodic.
        \end{conjecture}

In other words, the conjecture asserts that if $F$ tiles $G$ by translations, then $F$ periodically tiles $G$ by translations.

   We also consider the following continuous analogue of this conjecture.  If $\Sigma$ is a bounded measurable subset of a Euclidean space $\R^d$ of positive measure, and $\Lam$ is a subset of $\R^d$, we write $\Lam \oplus \Sigma =_\ae \R^d$ if the translates $\lambda + \Sigma$, $\lambda \in \R^d$, partition $\R^d$ up to null sets; note from the Steinhaus lemma that this forces $\Lambda$ to be discrete.  If this occurs, we say that $\Sigma$ (measurably) \emph{tiles $\R^d$ by translations}, and that $\Lambda$ is a \emph{tiling set of $\R^d$ by $\Lambda$}.  The tiling set $\Lambda$ is said to be \emph{periodic} if it is the finite union of cosets of a lattice (a discrete cocompact subgroup) of $\R^d$.  As before, we view $\Lam \oplus \Sigma =_\ae \R^d$ as a tiling equation with $d$ and $\Sigma$ given, and $\Lam$ as the unknown.  We say that this tiling equation $\Lam \oplus \Sigma =_\ae \R^d$ is \emph{aperiodic} if there exist solutions $\Lam \subset \R^d$ to the tiling equation $\Lam \oplus \Sigma =_\ae \R^d$, but none of these solutions are periodic.
   
\begin{conjecture}[Continuous periodic tiling conjecture]\label{ptc-cts}\cite{grunbaum-shephard,LW}   Let $\Sigma$ be a bounded measurable subset of $\R^d$ of positive measure.  Then the tiling equation $\Lam \oplus \Sigma =_\ae \R^d$ is not aperiodic.
\end{conjecture}

A standard argument shows that \conjref{ptc-cts} implies \conjref{ptc}. This implication arises  from ``encoding'' a discrete subset $F$ of $\Z^d$ as a bounded measurable subset $F \oplus R_d$ in $\R^d$, where $R_d$ is a ``generic'' fundamental domain of $\R^d/\Z^d$; we provide the details in \secref{sec:cont}.

Conjectures \ref{ptc} and \ref{ptc-cts} have been extensively studied over the years. The following partial results towards these conjectures are known:

\begin{itemize}
    \item Conjecture \ref{ptc} is trivial when $G$ is a finite abelian group, since in this case all subsets of $G$ are periodic.
\item  Conjectures \ref{ptc} and \ref{ptc-cts} were established for $G=\Z$ and $G=\R$ respectively \cite{N,LM,LW}.  The argument in \cite{N} also extends to the case $G = \Z \times G_0$ for any finite abelian group $G_0$ \cite[Section 2]{GT2}. 
\item When $G=\Z^2$, \conjref{ptc}  was established by Bhattacharya \cite{BH} using ergodic theory methods.  In  \cite{GT} we gave an alternative proof of this result, and  furthermore showed that every  tiling in $\Z^2$ by a single tile is \textit{weakly periodic} (a disjoint union of finitely many one-periodic sets). 
\item When $G=\R^2$, \conjref{ptc-cts} is known to hold for any tile that is a topological disk  \cite{bn,gbn,ken,err}.
\item Conjecture \ref{ptc-cts} is known to be true for  convex tiles in all dimensions  \cite{V,M}.
\item For $d>2$, Conjecture \ref{ptc} is known to hold when the cardinality $|F|$ of $F$ is prime or equal to $4$ \cite{szegedy}, but remained open in general. 
    \item In \cite{bgu}, it was recently shown that the discrete periodic tiling conjecture in $\Z^d$ also implies the discrete periodic tiling conjecture in every quotient group $\Z^d/\Lambda$.
    \item The analogues of the above conjectures are known to fail when one has two or more translational tiles instead of just one; see \cite{GT2} (particularly Table 1) for a summary of results in this direction.  In particular, in \cite[Theorems 1.8, 1.9]{GT2} it was shown that the analogue of  \conjref{ptc} for two tiles fails\footnote{Strictly speaking, the counterexample in that paper involved tiling a periodic subset $E$ of the group $G$, rather than the full group $G$. See however Remark \ref{extend} below.} for $\Z^2 \times G_0$ for some finite group $G_0$, and also for $\Z^d$ for some $d$.
\end{itemize}

\subsection{Results}
In this work we construct counterexamples to Conjectures \ref{ptc} and \ref{ptc-cts}.  Our first main result is

\begin{theorem}[Counterexample to \conjref{ptc}, I]\label{main}  There exist a finite abelian group $G_0$ and a finite non-empty subset $F$ of $\Z^2 \times G_0$ such that the tiling equation $A \oplus F = \Z^2 \times G_0$ is aperiodic.  In other words, the discrete periodic tiling conjecture fails for $\Z^2 \times G_0$.
\end{theorem}

\begin{remark} Our construction will in fact make $G_0$ a (non-elementary) $2$-group, that is to say a finite group whose order is a power of two. 
\end{remark}

The abelian finitely generated group $\Z^2 \times G_0$ can be viewed as a quotient $\Z^d/\Lambda$ of a lattice $\Z^d$ for sufficiently large $d$, so by Theorem \ref{main} and the recent implication in\footnote{This is a generalization of the argument in \cite[Section 9]{GT2}.}  \cite[Corollary 1.2]{bgu}
 we derive

\begin{corollary}[Counterexample to \conjref{ptc}, II]\label{main-cor1}  For sufficiently large $d$, there exists a finite non-empty subset $F$ of $\Z^d$ such that the tiling equation $A \oplus F = \Z^d$ is aperiodic.  In other words, the discrete periodic tiling conjecture fails for $\Z^d$.
\end{corollary}

By a standard construction (going back to Golomb \cite{golomb}) relating discrete and continuous tiling problems, we then have a corresponding counterexample to the continuous periodic tiling conjecture:

\begin{corollary}[Counterexample to \conjref{ptc-cts}]\label{main-cor}  For sufficiently large $d$, there exists a bounded measurable subset $\Sigma$ of $\R^d$ of positive measure such that the tiling equation $\Lam \oplus \Sigma =_\ae \R^d$ is aperiodic.  In other words, the continuous periodic tiling conjecture fails for $\R^d$.
\end{corollary}

We give the (straightforward) derivation of Corollary \ref{main-cor} from Corollary \ref{main-cor1} in Section \ref{sec:cont}.

Our methods produce a finite group $G_0$, and hence a dimension $d$, that is in principle explicitly computable, but we have not attempted to optimize the size of these objects.  In particular the dimension $d$ produced by our construction will be extremely large.


\subsection{Previous works and constructions}

Aperiodic tilings have been extensively studied and  have found famous applications to many areas of mathematics and physics \cite{quasicrystals}.  	The study of the periodicity of tilings has attracted many researchers, who have introduced methods from various fields, such as geometry and topology \cite{gbn,ken,err}, Fourier analysis \cite{LW,kl,kol}, combinatorics \cite{GT,GT2}, ergodic theory and probability \cite{SM,Levin,BH}, commutative algebra \cite{szegedy,BH,GT},  model theory \cite{BJ,GT2}, and computability theory \cite{Ber,kari,Levin,JR,GT2}.

We do not attempt a comprehensive survey of aperiodic constructions here, but briefly summarize the current state of knowledge as follows.
\begin{itemize}
    \item Aperiodic tiling by multiple tiles have been long known to exist.   The online encyclopedia of tilings \cite{website} contains many explicit examples of such tilings.  In the plane, there are the famous    substitution tilings constructions of Penrose and Ammann \cite{penrose1,penrose2,DB,gardner,ags}. (See also \cite{GS1} and the references therein for the study of substitution tilings.)    Other aperiodic tiling construction methods include the  finite state machine approaches of Kari and Culik \cite{kari,culik}, and the approach of  encoding arbitrary Turing machines\footnote{This method in fact allows one to construct tiling problems which are not only aperiodic, but in fact \emph{logically undecidable}; see e.g., \cite{GT2} for further discussion.} into a tiling problem \cite{Ber,Ber-thesis,R,ollinger,GT2}.  
    \item In addition, if one allows for the tile to be rotated (and/or reflected) in addition to being translated, aperiodic non-translational tilings by a single tile (or ``monotile'') are known to exist; see, e.g., \cite{einstein,einstein2}. The question whether there are {\it planar aperiodic connected tiles} by translations, rotations and reflections remained open until very recently, when the ``hat'' monotile was discovered by Smith--Myers--Kaplan--Godman-Strauss \cite{hat}. Moreover, in a subsequent paper, the same authors constructed a connected planar domain which tiles the plane aperiodically by translations and rotations only (no reflections) \cite{chiral}.
    These results solve the celebrated ``einstein problem'' which is an extension of the second part of Hilbert's eighteenth  problem.
    \item Moreover, when one allows for the group to be \emph{non-abelian}, aperiodic (and undecidable) tilings by a single tile are known to exist.  For instance, in \cite[Theorem 11.2]{GT2} we give a construction in $\Z^2\times H$ for a certain finite non-abelian group $H$. See also \cite{th,SM2,GS2,SSU,abj,C17}    for further references to of aperiodic tilings (or subshifts of finite type) in various groups.
\end{itemize}

We were not able to adapt the previous aperiodic constructions to the setting of a \emph{single} translational tile.  Instead, our source of aperiodicity is more\footnote{Since the initial release of this preprint, we have learned (Emmanuel Jeandel, private communication) that a similar use of $p$-adic functions (with $p$ sufficiently large, but not necessarily a power of $2$) was employed by Aanderaa and Lewis \cite{al}, \cite{lewis} to establish the undecidability of an empty distance subshift problem, which in turn implied the undecidability  of the domino problem; see \cite[\S 4]{jv} for further discussion.}  novel, in that our tiling of $\Z^2 \times G_0$ is forced to exhibit a ``$q$-adic'' (or ``$2$-adic'') structure\footnote{Intriguingly, similar ``inverse limits of coset structure'' appears in other aperiodic tiling constructions, such as the dragon tiling \cite{dragon}, the Robinson tiling \cite{R}, or the trilobite and crab tilings \cite{GS3}, as well as some square-triangle tilings.}  for some large enough but fixed power of two $q=2^s$ (say $s=10$) in the sense that for each power $q^j$ of $q$, the tiling is periodic with period $q^j \Z^2 \times \{0\}$ outside of a small number of cosets of that subgroup $q^j \Z^2 \times \{0\}$, but is unable to be genuinely periodic with respect to any of these periods.  To achieve this we will set up a certain ``Sudoku puzzle'', which will be rigid enough to force all solutions of this problem to exhibit a certain ``self-similar'' (and therefore non-periodic) behavior, yet is not so rigid that there are no solutions whatsoever.  By modifying arguments from our previous paper \cite{GT2}, we are then able to encode this Sudoku-type puzzle as an instance of the original tiling problem $A \oplus F = \Z^2 \times G_0$.

Our encoding approach is similar in nature to previous ``encoding'' arguments in the tiling literature. Berger \cite{Ber,Ber-thesis} encoded any Turing machine as a Wang tiling problem. Since the halting problem is known to be undecidable, Berger's encoding implies the undecidability of the Wang domino problem.  Subsequently, Wang tilings were encoded  to obtain aperiodicity, strong aperiodicity, or even undecidability  of various other problems; see, e.g., \cite{einstein,golomb,SSU,SM,R,GS1,GS2}.  In particular, in  \cite{GT2} we used our tiling language approach to encode any Wang tiling problem as a tiling of $\Z^2\times G_0$ by two tiles, for a suitable finite abelian group $G_0$ (depending on the given problem). This implies the existence of an undecidable tiling problem with only two tiles. Unfortunately, in our encoding of the Wang domino puzzle, we were not able to reduce the number of the tiles from two to one. Thus, the main difficulty we address in our current work is finding another aperiodic puzzle (replacing the Wang domino puzzle) which is also \emph{expressible}  in our tiling language of a tiling by a single tile.

\subsection{Our argument and the organization of the paper}

\begin{figure}
\centering
    \includegraphics[width = .9\textwidth]{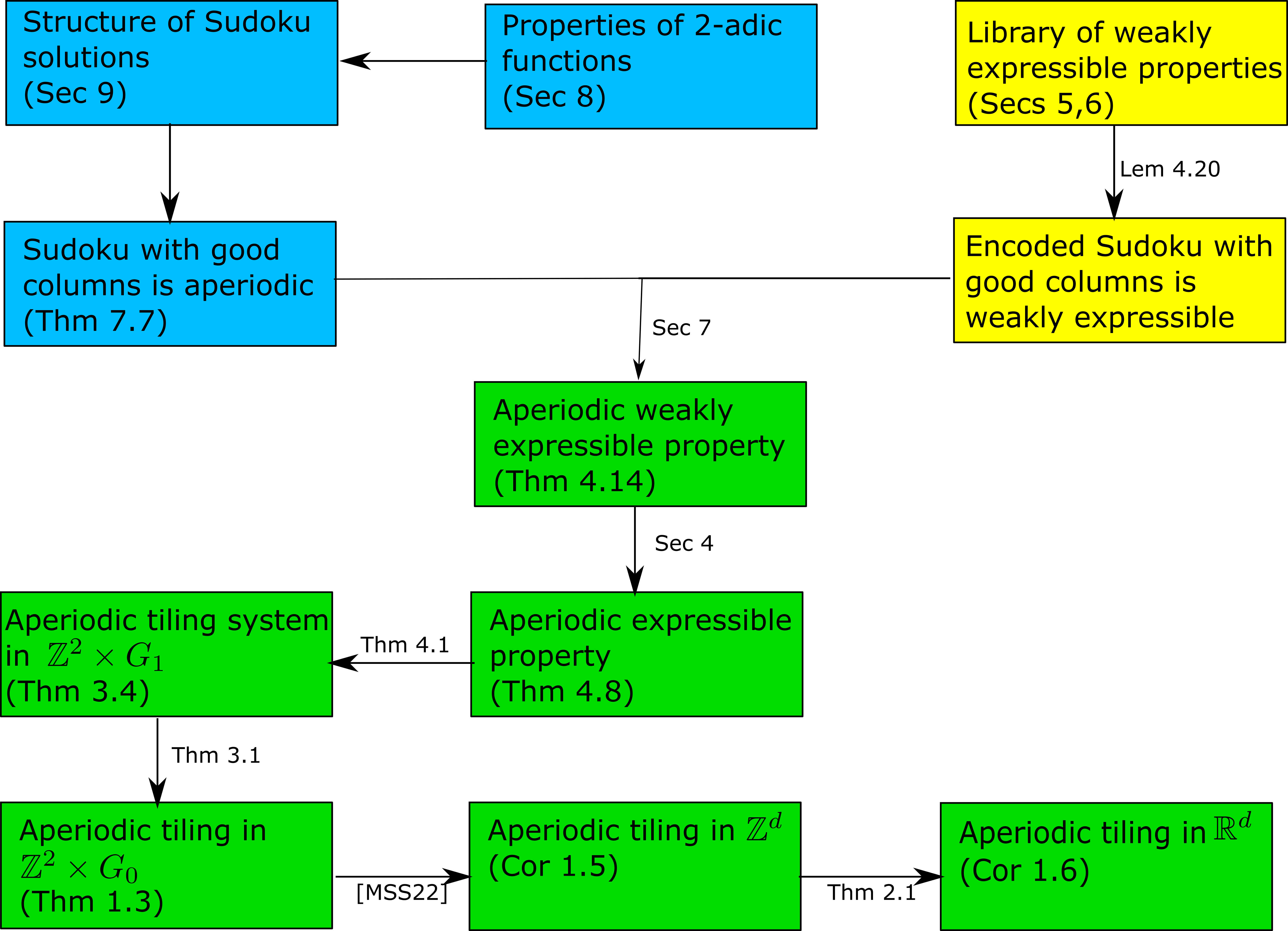}
    \caption{A high-level overview of the logical implications used in our proof. We introduce an aperiodic Sudoku puzzle (blue) and develop a tiling library to express this  puzzle inside a tiling by a single tile (yellow). This, in turn, eventually leads to   constructions of  aperiodic translational tilings by a single tile in $\Z^2\times G_0$, $\Z^d$ and $\R^d$ (green).  This diagram has been ``curled up'' into a compact bounding rectangle purely to save space, and the reader is welcome to mentally ``straighten'' it if desired.}
  \label{fig:logic}
\end{figure} 

Our argument is a variant of the construction used in our previous paper \cite{GT2} to produce aperiodic (and even undecidable) translational tilings with two tiles, and is summarized by the diagram in Figure \ref{fig:logic}. However, the fact that we are now tiling the whole group $G$ instead of a periodic subset of $G$, and that we are only allowed to use one tile instead of two, creates additional technical challenges.

As in \cite{GT2}, in \secref{sec:multi} we begin by replacing the single tiling equation $A \oplus F = G$ with a system $A \oplus F^{(m)} = G$, $m = 1,\dots,M$ of tiling equations for an arbitrary $M$, by an elementary ``stacking'' procedure that takes advantage of our freedom to enlarge the group $G$.  This creates a flexible ``tiling language'' of constraints on the tiling set $A$; the challenge is to use this language to obtain a system of constraints that is strict enough to force aperiodic behavior on this set $A$, while simultaneously being relaxed enough to admit at least one solution.

Next, in \secref{sec:functional}, we again follow \cite{GT2} and pass from this tiling language to a language of functional equations, basically by spending one of the equations $A \oplus F^{(m)} = G$ in the system to force the tiling set $A$ to be a graph of a function $f=(f_1,\dots,f_K)$, where  $f_i \colon \Z^2 \times G_0 \to \Z/q\Z$, $1\leq i\leq K$, and  $G_0$ is an additional small finite abelian group, which we retain for technical reasons.  

One can then use one or more tiling equations $A \oplus F^{(m)} = G$ in the tiling language to create a ``library'' of useful functional constraints on these functions $f_i$, this is done in \secref{sec:library}.  For instance, one can ensure that a given function $f_i$ exhibits periodicity in some direction $v_i \in \Z^2$, or that it encodes (the periodic extension of) a permutation of a cyclic group $\Z/q\Z$.  

In \secref{sec:boolean} we express via functional equations the assertion that a certain subcollection of the $f_i$ (after a routine normalization) take values in a two-element set $\{a,b\} \Mod{q}$, where $a,b$ have different parity, and can thus be viewed as  boolean functions.  By modifying our construction from \cite[Section 7]{GT2}, we can then use tiling equations to encode arbitrary pointwise constraints
\begin{equation}\label{fkx}
(f_1(x), \dots, f_K(x)) \in \Omega
\end{equation}
for all $x \in \Z^2 \times G_0$ and arbitrary subsets $\Omega$ of $\{a,b\}^K$.  This turns out to be a particularly powerful addition to our library of expressible properties.

In \secref{sec:sudoku}, by some further elementary transformations (including a change of variables that resembles the classical projective duality between lines and points), we are then able to reduce matters to demonstrating aperiodicity of a certain ``Sudoku puzzle''.  In this puzzle, we have an unknown function $F \colon \{1,\dots,N\} \times \Z \to \Z/q\Z\setminus\{0\}$ on a vertically infinite ``Sudoku board'' $\{1,\dots,N\} \times \Z$ which fills each cell $(n,m)$ of this board with an element $F(n,m)$ of  $ \Z/q\Z\setminus\{0\}$ for some fixed but large  $q=2^s$.  Along every row or diagonal (and more generally along any non-vertical line) of this board, the function $F$ is required\footnote{This is analogous to how, in the most popular form of a Sudoku puzzle, the rows, columns, and $3 \times 3$ blocks of cells on a board $\{1,\dots,9\} \times \{1,\dots,9\}$ are required to be permutations of the digit set $\{1,\dots,9\}$.} to exhibit ``$2$-adic behavior''; the precise description of this behavior will be given in \secref{sec:sudoku}, but roughly speaking we will require that on each such non-vertical line, $F$ behaves like a rescaled version of the function \begin{equation}\label{fpn}
f_q(n) \coloneqq \frac{n}{q^{\nu_q(n)}} \Mod{q}
\end{equation}
(where $\nu_q(n)$ is the number of times $q$ divides $n$), 
that assigns to each integer $n$ the final non-zero digit in its base $q$ expansion (with the convention $f_q(0) \coloneqq 1$).  We also impose a non-degeneracy condition that the Sudoku solution function $F$ is a periodized permutation along any of its columns.  

In \secref{conclusion}, for suitable choices of parameters $q,N$, we  ``solve'' this Sudoku problem and show that solutions to this problem exist, but necessarily exhibit self-similar behavior (in that certain rescalings of the solution obey similar properties to the original solution), and in particular are non-periodic.  By combining this aperiodicity result with the previous encodings and reductions, we are able to establish  \thmref{main} and hence \corref{main-cor}.

\begin{remark}\label{extend}
  Our current argument also provides a solution to \cite[Problem 12.3]{GT2}. Namely, using the more advanced library we develop here (Sections \ref{sec:cont}--\ref{sec:library}), we can strengthen our previous undecidability result with two tiles by now tiling all of the   group rather than just a periodic subset.  We leave the details of this modification of the construction to the interested reader.
\end{remark}

  \subsection{Notation}\label{notation-sec}

We define the disjoint union $\biguplus_{w \in \Wcal} E_w$ of sets $E_w$ indexed by some set $\Wcal$ to be the union $\bigcup_{w \in \Wcal} E_w$ if the $E_w$ are disjoint, and leave $\biguplus_{w \in \Wcal} E_w$ undefined otherwise.

All groups in this paper will be written additively and be assumed to be abelian unless otherwise specified.  If $A,B,C$ are subsets of $G$, we use $A \oplus B = C$ to denote the assertion that the translates $a+B$, $a \in A$ partition $C$; if the translates $a+B$ are not disjoint, we leave $A \oplus B$ undefined. Thus $A \oplus B = C$ is equivalent to $\biguplus_{a \in A} (a+B) = C$.  Similarly, if $\Lambda \subset \R^d$ and $\Sigma \subset \R^d$ are discrete and measurable respectively, and $E \subset \R^d$ is another measurable set, we write $\Lambda \oplus \Sigma =_\ae E$ if the translates $\lambda + \Sigma$, $\lambda \in \Lambda$ partition $E$ up to null sets; if the $\lambda + \Sigma$ are not disjoint up to null sets, we leave $\Lambda \oplus \Sigma$ undefined.

We use $1_E$ to denote the indicator of an event $E$, thus $1_E$ is $1$ when $E$ is true and $0$ otherwise.

By abuse of notation, we will sometimes identify an integer $a \in \Z$ with its representative $a \Mod{N} \in \Z/N\Z$ in a cyclic group $\Z/N\Z$ when there is no chance of confusion.  For instance, we may refer to the multiplicative identity of $\Z/N\Z$ (viewed as a ring) as $1$ rather than $1 \Mod{N}$.

If $v_1,\dots,v_k$ are elements of a group $G$, we use $\langle v_1,\dots,v_k\rangle$ to denote the group that they generate.  If $H$ is a subgroup of $G$, then a function $f \colon G \to X$ on $G$ is said to be \emph{$H$-periodic} if $f(x+h) = f(x)$ for all $x \in G$ and $h \in H$.  In particular, a function is $\langle v_1,\dots,v_k\rangle$-periodic if and only if $f(x+v_i)=f(x)$ for all $x \in G$ and $i=1,\dots,k$.

We use $X = O(Y)$, $X \ll Y$, or $Y \gg X$ to denote the estimate $|X| \leq CY$ for some absolute constant $C$ (which will not depend on other parameters such as $q$ or $N$).  We write $X \asymp Y$ for $X \ll Y \ll X$.

We use $|E|$ to denote the cardinality of a finite set $E$.  If $E \subset \Omega \subset \R^d$ with $\Omega$ non-empty, we define the \emph{upper density of $E$ in $\Omega$} to be the quantity
$$ \limsup_{M \to \infty} \frac{ |E \cap \{-M,\dots,M\}^d|}{|\Omega \cap \{-M,\dots,M\}^d|}.$$
Thus for instance if $q,N$ are natural numbers, the set $\{1,\dots,N\} \times q\Z$ has upper density $\frac{1}{q}$ in $\{1,\dots,N\} \times\Z$.

    	\subsection{Acknowledgments}
    	 RG was partially supported by the  AMIAS Membership  and NSF  grants  DMS-2242871 and   DMS-1926686. TT was partially supported by NSF grant DMS-1764034 and by a Simons Investigator Award. We thank Nishant Chandgotia, Asaf Katz,  S\'{e}bastien  Labb\'{e} and Misha Sodin   for drawing our attention to some relevant references and to Emmanuel Jeandel for helpful comments. We are also grateful to the anonymous referee for many helpful suggestions that improved the exposition of this paper.

        	
\section{Building a continuous aperiodic tiling equation from a discrete aperiodic tiling equation} \label{sec:cont}

    In this section we show that a counterexample to the discrete periodic tiling conjecture can be converted to a counterexample to the continuous periodic tiling conjecture.  More precisely, we show
    	
    \begin{theorem}[Lifting a discrete aperiodic tiling equation to a continuous aperiodic tiling equation]\label{RdvsZd}
    	Let $d\geq 1$. If there is an aperiodic tiling equation $A \oplus F = \Z^d$ for some finite non-empty subset $F$ of $\Z^d$, then there is an aperiodic tiling equation $\Lam \oplus \Sigma =_\ae \R^d$ for some bounded measurable subset $\Sigma$ of  $\R^d$ of positive measure.  In other words, if \conjref{ptc} fails in $\Z^d$, then  \conjref{ptc-cts} fails in $\R^d$.
    \end{theorem}

A basic connection between the discrete lattice $\Z^d$ and the continuous space $\R^d$ is given by the tiling relation
$$ \Z^d \oplus Q_d =_\ae \R^d,$$
where $Q_d \coloneqq [0,1]^d$ is the unit cube.  By translation invariance one also has
$$ (\Z^d+t) \oplus Q_d =_\ae \R^d$$
for any $t \in \R^d$.  However, due to the ability to ``slide'' cubes $Q_d$ in various directions, there are many more tilings of $\R^d$ by $Q_d$ than these; this is evidenced for instance by the failure of Keller's conjecture in high dimensions \cite{lagarias-shor}.  Because of this, the unit cube $Q_d$ is not a suitable tool for establishing Theorem \ref{RdvsZd}.  Instead, we need a ``rigid'' version $R_d$ of $Q_d$, or more precisely,

\begin{lemma}[Existence of a rigid tile]\label{rigid-exist}  For any $d \geq 1$, there exists a bounded measurable subset $R_d$ of $\R^d$ such that $\Z^d \oplus R_d =_\ae \R^d$, and conversely the only sets $\Lambda \subset \R^d$ with $\Lambda \oplus R_d =_\ae \R^d$ are translates $\Lambda = \Z^d + t$ of $\Z^d$ for some $t \in \R^d$.
\end{lemma}

The idea of using rigid tiles to pass back and forth between discrete and continuous tiling questions goes back to the work of Golomb \cite{golomb}; see also \cite[Lemma 9.3]{GT2} for a discretized version of this lemma.

\begin{proof}  The idea is to remove and add ``bumps'' at the facets  of $Q_d$ to make a rigid ``jigsaw puzzle piece''; see Figure \ref{fig:jpuzzle}.  There are many constructions available.  For instance, 
we can define $R_d$ to be the set
$$ R_d \coloneqq \left(Q_d \backslash \biguplus_{k=1}^d C_k\right) \uplus \biguplus_{k=1}^d (C_k+e_k)$$
where for $e_1,\dots,e_d$ is the standard basis for $\R^d$, and for each $k=1,\dots,d$, $C_k \subset Q_d$ is a $\epsilon$-subcube of $Q_d$, which one can for instance define as
$$C_k \coloneqq \left(\prod_{j=1}^{k-1}[x_j,x_j+\epsilon]\right)\times [0,\epsilon]\times \prod_{j=k+1}^d [x_j,x_j+\epsilon]$$
where $0<\epsilon<1/5$ and $2\epsilon<x_j<1-3\epsilon$, $j=1,\dots,d$ are arbitrary.  Because the piece $C_k$ removed for a given $k$ is a translate by an element of $\Z^d$ of the piece $C_k+e_k$ added for a given $k$, we still have
$$ \Z^d \oplus R_d =_\ae \Z^d \oplus Q_d =_\ae \R^d.$$
On the other hand, it is geometrically evident that if $\Lambda \oplus R_d =_\ae \R^d$ and $t \in \Lambda$, then $t \pm e_k$ must also lie in $\Lambda$ for all $k=1,\dots,d$, as there is no other way to fit translates of $R_d$ around the added and removed ``bumps'' $C_k+t$, $C_k+e_k+t$ of $R_d + t$.  Thus $\Lambda$ must contain a translated lattice $\Z^d + t$; since this lattice already is a tiling set of $\R^d$ by $R_d$, we therefore have $\Lambda = \Z^d + t$, as required.
\end{proof}

Using this rigid tile, it is now straightforward to establish Theorem \ref{RdvsZd}.
     
    	\begin{proof}[Proof of Theorem \ref{RdvsZd}]
    	Suppose  that there is a finite non-empty $F\subset \Z^d$ such that the tiling equation $A \oplus F = \Z^d$ is aperiodic.
     
With $R_d$ being the rigid tile provided by Lemma \ref{rigid-exist}, we introduce the bounded measurable subset $\Sigma$ of $\R^d$ by the formula
$$ \Sigma \coloneqq F \oplus R_d \subset \Z^d \oplus Q_d =_\ae \R^d.$$
Clearly $\Sigma$ has positive measure.
It will suffice to show that the tiling equation $\Lambda \oplus \Sigma =_\ae \R^d$ is aperiodic.  On the one hand we have
$$ A \oplus \Sigma =_\ae (A \oplus F) \oplus R_d =_\ae \Z^d \oplus R_d =_\ae \R^d$$
so there is at least one tiling of $\R^d$ by $\Sigma$.

Conversely, suppose that we have a tiling $\Lambda \oplus \Sigma =_\ae \R^d$ of $\R^d$.  Then we have
$$ (\Lambda \oplus F) \oplus R_d =_\ae \Lambda \oplus \Sigma =_\ae \R^d,$$
and hence by Lemma \ref{rigid-exist}, we have $\Lambda \oplus F = \Z^d + t$ for some $t \in \R^d$.  Then $\Lambda-t$ is a tiling set of $\Z^d$ by $F$ and is hence not periodic by hypothesis.  This implies that $\Lambda$ is not periodic, and so the tiling equation $\Lambda \oplus \Sigma =_\ae \R^d$ is aperiodic as claimed.
    	\end{proof}
    	
\begin{figure}
\centering
    \includegraphics[width = .5\textwidth]{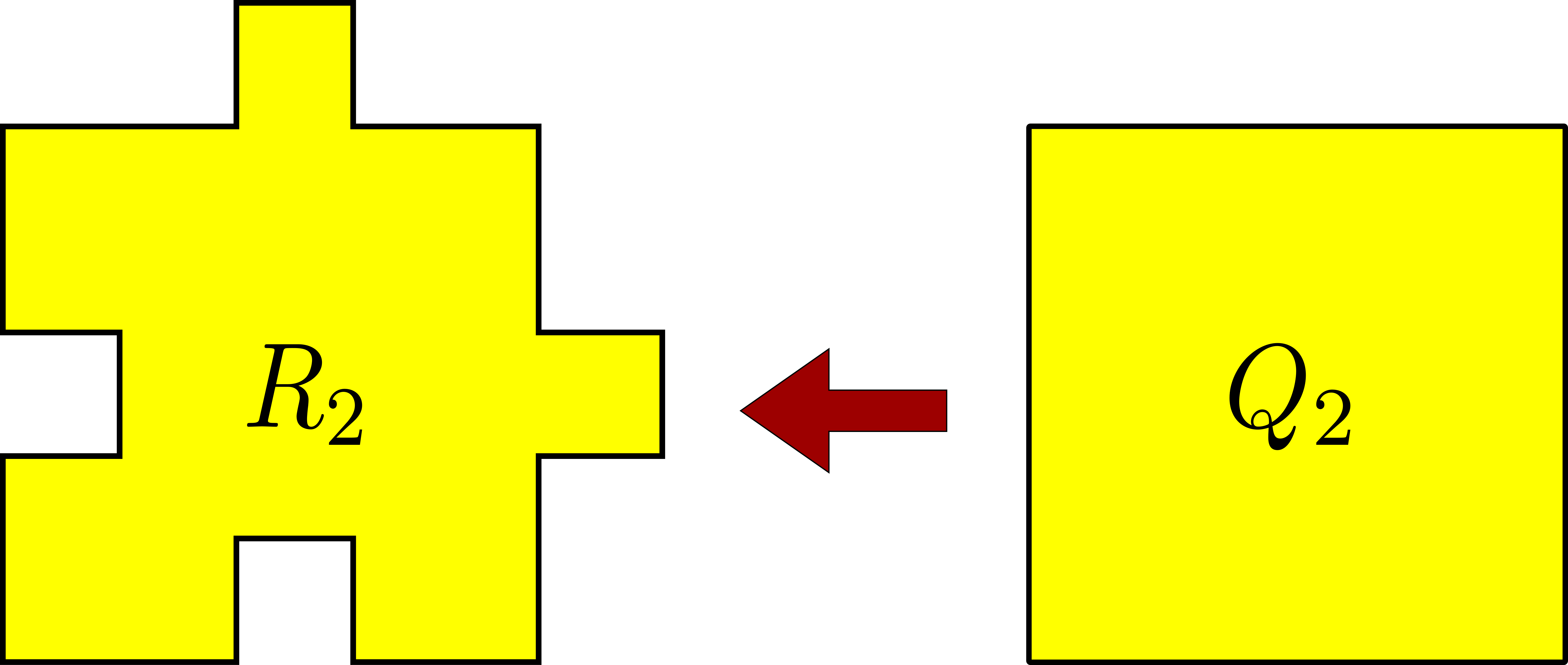}
    \caption{A ``rigid'' tile $R_2$ replacing the non rigid tile $Q_2=[0,1]^2$.  The only tilings $\Lambda \oplus R_2 =_\ae \R^2$ of the plane $\R^2$ by $\R_2$ are the translated lattice tilings $(\Z^2+t) \oplus R_2 =_\ae \R^2$ for $t \in \R^2$.}
  \label{fig:jpuzzle}
\end{figure} 

In view of Theorem \ref{RdvsZd}, we see that Corollary \ref{main-cor1} implies Corollary \ref{main-cor}.  In \cite{bgu} it was shown that any tiling of a quotient group $\Z^d/\Lambda$ can be identified with a tiling of $\Z^d$. This is done by a rigid pullback argument, generalizing  \cite[Section 9]{GT2}.
As a corollary, this gives that the discrete periodic tiling conjecture in $\Z^d$ also implies the discrete periodic tiling conjecture in every quotient group $\Z^d/\Lambda$ \cite[Corollary 1.2]{bgu}.
Thus, we have that Theorem \ref{main} implies Corollary \ref{main-cor}. It therefore remains to establish \thmref{main}.   This is the objective of the remaining sections of the paper.
      	
\section{Building an aperiodic tiling equation from an aperiodic system of tiling equations}\label{sec:multi}

Theorem \ref{main} asserts the construction of a single tiling equation $A \oplus F = G$ which is aperiodic.  As in our previous paper \cite{GT2}, it will be more convenient to consider the significantly more flexible problem of constructing a \emph{system}
\begin{equation}\label{a-system}
A \oplus F_m = G \hbox{ for all }  m=1,\dots,M
\end{equation}
of tiling equations which are (jointly) \emph{aperiodic} in the sense that solutions $A \subset G$  to the system \eqref{a-system} exist, but none of them are periodic.  The ability to pass to pass to this more flexible setup is provided by the following tool (compare with Theorem \ref{RdvsZd}):

\begin{theorem}[Concatenating an aperiodic system of tiling equations into a single aperiodic tiling equation]\label{main-red}  Let $G$ be a finitely generated abelian group.  Suppose that there exist finite non-empty sets $F_1,\dots,F_M \subset G$ for some $M \geq 1$ such that the system \eqref{a-system} of tiling equations is aperiodic.  Then there exist a $2$-group of the form $\Z/N\Z$, $N=2^r$, and a finite non-empty subset $\tilde F$ of $G \times \Z/N\Z$ such that the single tiling equation
$$ \tilde A \oplus \tilde F  = G \times \Z/N\Z$$
is aperiodic.
\end{theorem}

This theorem is a variant of our previous result \cite[Theorem 1.15]{GT2}, in which the $2$-group $\Z/N\Z$ was replaced by a proper \emph{subset} of the cyclic group $\Z/(M+1)\Z$.  In order to be able to tile the \emph{whole} group, we will utilize a  ``rigid'' partition of $\Z/N\Z$.  More precisely, we have the following analogue of Lemma \ref{rigid-exist}:

\begin{lemma}[Construction of a ``rigid'' partition]\label{good-part}  For every $M \geq 1$, there exist $N \geq 1$ and a partition $\Z/N\Z = E_1 \uplus \dots \uplus E_M$ of $\Z/N\Z$ into $M$ non-empty sets $E_1,\dots,E_M$, such that 
\begin{equation}\label{rigid-2}
E_i \cap (E_j+h) \neq \emptyset
\end{equation}
for any $1 \leq i,j \leq M$ and $h \in \Z/N\Z \backslash \{0\}$.  In particular, for any $1 \leq i,j \leq M$ and $h_i,h_j \in \Z/N\Z$, we have
\begin{equation}\label{rigid}
    (E_i + h_i) \cap (E_j + h_j) \neq \emptyset
\end{equation} 
unless $h_i = h_j$ and $i \neq j$.
\end{lemma}

\begin{proof}
To construct such $E_1,\dots,E_M$ we use the probabilistic method.  Let $N$ be a sufficiently large power of two (depending on $M$) to be chosen later.  Let $a \colon \Z/N\Z \to \{1,\dots,M\}$ be a function chosen uniformly at random, thus the $a(x) \in \{1,\dots,M\}$ for $x \in \Z/N\Z$ are independent uniform random variables.  We then set $E_i \coloneqq \{x \in \Z/N\Z: a(x) = i\}$ to be the level sets of $a$.  Clearly the $E_1,\dots,E_M$ partition $\Z/N\Z$.  The probability that a given $E_i$ is empty is $(1-1/M)^{N}$.  Now let $1 \leq i,j \leq M$ and $h \in \Z/N\Z \backslash \{0\}$.   Then the only way that \eqref{rigid-2} fails for this choice of parameters is if $(a(x), a(x-h)) \neq (i,j)$ for all $x \in \Z/N\Z$.  As $h \neq 0$, it has even order, so one can partition $\Z/N\Z$ into $N/2$ sets of the form $\{x,x-h\}$, so the probability that \eqref{rigid-2} fails for this choice of parameters is at most $(1-1/M^2)^{N/2}$.  As the total number of choices of $(i,j,h)$ is at most $M^2 N$, the probability that this construction fails to work is thus at most
$$ M (1-1/M)^N + M^2 N (1-1/M^2)^{N/2}.$$
For $N$ sufficiently large depending on $M$, this failure probability is less than $1$, and the claim follows.
\end{proof}

\begin{remark} An inspection of the bounds shows that one can take the $2$-group $\Z/N\Z$ to be of order $N = O(M^2 \log M)$.  A similar construction works with $\Z/N\Z$ replaced by other finite abelian groups of comparable order.  We were able to also find  deterministic constructions of the sets $E_1,\dots,E_M$ in various such groups, but for such constructions the verification of the key property \eqref{rigid} required a longer argument than the probabilistic arguments provided here.
\end{remark}

We are now ready to prove Theorem \ref{main-red}.  

\begin{proof}[Proof of Theorem \ref{main-red}]
Let $G,F_1,\dots,F_M$ be as in that theorem. We use the partition $\Z/N\Z = E_1 \uplus \dots \uplus E_M$ provided by the above lemma to form the combined tile
\begin{equation}\label{f-stack}
 \tilde F \coloneqq \biguplus_{m=1}^M (F_m \times E_m) \subset G \times \Z/N\Z.
 \end{equation}
To complete the proof of Theorem \ref{main-red}, it suffices to show that the single tiling equation 
\begin{equation}\label{ta2}
 \tilde A \oplus \tilde F =  G \times \Z/N\Z.
\end{equation} 
is aperiodic.

To verify this claim, we first observe that by hypothesis there exists $A \subset G$ such that $A \oplus F_m = G$ for all $m=1,\dots,M$.  If we set $\tilde A \coloneqq A \times \{0\} \subset G \times \Z/N\Z$, then we have from \eqref{f-stack} that
$$ \tilde A \oplus \tilde F = \biguplus_{m=1}^M ((A \oplus F_m) \times E_m) = G \times \biguplus_{m=1}^M E_m = G \times \Z/N\Z.$$
Thus the tiling equation \eqref{ta2} has at least one solution.

Conversely, suppose  $\tilde A \subset G \times \Z/N\Z$ solves the tiling equation \eqref{ta2}.
We first claim that any ``vertical line'' $\{a\} \times \Z/N\Z$, $a \in G$ intersects $\tilde A$ in at most one point.  Indeed, if $(a,h), (a,h') \in \tilde A$ for some $h \neq h'$, then by \eqref{f-stack} $\tilde A \oplus \tilde F$ will contain both $(a + F_1) \times (h+E_1)$ and $(a+F_1) \times (h'+E_1)$ as disjoint sets.  But by \eqref{rigid}, $h+E_1$, $h'+E_1$ intersect, a contradiction.

Because each vertical line $\{a\} \times \Z/N\Z$, $a \in G$ meets $A$ in at most one point, we can write $\tilde A$ as a graph
$$ \tilde A = \{ (a, f(a)): a \in A \}$$
for some $A \subset G$ and some function $f \colon A \to \Z/N\Z$.  From \eqref{ta2}, \eqref{f-stack} we see that the sets 
\begin{equation}\label{cart}
(a+F_m) \times (f(a) + E_m)
\end{equation}
for $a \in A$ and $m=1,\dots,M$ partition $G \times \Z/N\Z$.  

We now claim that for any $m=1,\dots,M$, the sets $a+F_m$, $a \in A$ are disjoint.  For if we had $a+f = a'+f'$ for some distinct $a,a' \in A$ and $f,f' \in F_m$, then $\{a+f\} \times (f(a) + E_m)$ and $\{a'+f'\} \times (f(a')+E_m)$ would have to be disjoint, but this again contradicts \eqref{rigid}.  

By restricting the partition \eqref{cart} of $G \times \Z/N\Z$ to a single vertical line $\{b\} \times \Z/N\Z$, we see that for any $b \in G$, we can partition $\Z/N\Z$ into $f(a_m)+E_m$, where $m=1,\dots,M$ and $a_m$ is the unique element of $A$ (if it exists) such that $b \in a_m+F_m$.  Since $f(a_m)+E_m$ has cardinality $|E_m| > 0$, and $|E_1|+\dots+|E_M|=N$, we conclude that $a_m$ must exist for every $m=1,\dots,M$.  In other words, $A \oplus F_m = G$ for every $m=1,\dots,M$.  By hypothesis, this implies that $A$ is non-periodic.  Since $A$ is the projection of $\tilde A$ to $G$, this implies that $\tilde A$ is also non-periodic. Thus
the tiling equation \eqref{ta2} is aperiodic, and Theorem \ref{main-red} follows.
\end{proof}

Let us say that the \emph{multiple periodic tiling conjecture} holds for some finitely generated abelian group $G$ if, whenever $F_1,\dots,F_M$ are finite non-empty subsets of $G$, the system \eqref{a-system} of tiling equations is not aperiodic.  Obviously, the multiple periodic tiling conjecture for a given group implies the periodic tiling conjecture for that group.  Applying Theorem \ref{main-red}, we conclude that the periodic tiling conjecture will hold for $\Z^2 \times G_0$ for all finite abelian groups $G_0$ if and only if the multiple periodic tiling conjecture holds for $\Z^2 \times G_1$ for all finite abelian groups $G_1$.  Thus, to establish Theorem \ref{main}, it now suffices to establish

\begin{theorem}[Counterexample to multiple periodic tiling conjecture]\label{main-multi}  There exist a finite abelian group $G_1$ and a finite non-empty subsets $F_1,\dots,F_M$ of $G = \Z^2 \times G_1$ such that the system \eqref{a-system} of tiling equations is aperiodic. In other words, the multiple periodic tiling conjecture fails for $\Z^2 \times G_1$.
\end{theorem}

Our remaining task is to establish Theorem \ref{main-multi}. This is the objective of the remaining sections of the paper.

\section{Building an aperiodic system of tiling equations from an aperiodic property expressible in functional equations}\label{sec:functional}

One can view the individual tiling equations $A \oplus F_m = G$ in \eqref{a-system} as sentences in a ``tiling language'' that assert various constraints on the tiling set $A$.  Theorem \ref{main-multi} can then be thought of as an assertion that this tiling language is expressive enough to describe a type of set $A \subset \Z^2 \times G_1$ that can exist, but is necessarily non-periodic.  

In this section we show that one can replace the language of tiling equations $A \oplus F = G$ by the language of \emph{functional equations}, in which the unknown object is now a function $\alpha \colon G \to H$ from a finitely generated abelian group $G$ to a finite abelian group $H$, rather than a subset $A$ of $G$, and then develop the further theory of this ``functional equation language''.  A single functional equation in this language will take the form
\begin{equation}\label{function-ex}
\biguplus_{j=1}^{J} (\alpha(x + h_j) + E_{j}) = H \hbox{ for every } x\in G
\end{equation}
for some given shifts $h_1,\dots,h_J \in G$ and some sets $E_1,\dots,E_J \subset H$, which we may take to be non-empty.  For instance, in Example \ref{clock} below we will consider the functional equation
$$ (\alpha(x) + \{1\}) \uplus (\alpha(x+1) + (\Z/N\Z \backslash \{0\})) = \Z/N\Z$$
that may or may not be satisfied by a given function $\alpha \colon \Z \to \Z/N\Z$.  A system
\begin{equation}\label{function-system}
 \biguplus_{j=1}^{J_i} (\alpha(x + h_{i,j}) + E_{i,j}) = H \hbox{ for all } i=1,\dots,M,\; x\in G
\end{equation}
of such functional equations will be said to be \emph{aperiodic} if solutions $\alpha \colon G \to H$ to this system exist, but that they are all non-periodic, by which we mean that there is no finite index subgroup $\Lambda$ of $G$ such that $\alpha(x+h) = \alpha(x)$ for all $x \in G$ and $h \in \Lambda$.

We then have the following tool to convert aperiodic systems of functional equations to aperiodic systems of tiling equations, in the spirit of Theorems \ref{RdvsZd}, \ref{main-red}.

\begin{theorem}[Converting an aperiodic system of functional equations to an aperiodic system of tiling equations]\label{functional-to-tile}  Let $G$ be a finitely generated abelian group, and let $H$ be a finite abelian group.  Suppose that there exists $M \geq 1$, and for each $i=1,\dots,M$ there exists $J_i \geq 1$, and for each $1 \leq j \leq J_i$ there exist shifts $h_{i,j} \in G$ and sets $E_{i,j} \subset H$, such that the system \eqref{function-system} of functional equations is aperiodic.  Then there exists a system \eqref{a-system} of tiling equations in $G \times H$ which is aperiodic.
\end{theorem}

\begin{proof}  We will consider the system of tiling equations in $G \times H$ consisting of the ``vertical line test'' equation
\begin{equation}\label{sys-2}
A \oplus (\{0\} \times H) = G \times H
\end{equation}
as well as the tiling equations
\begin{equation}\label{sys-1}
 A \oplus \biguplus_{j=1}^{J_i} \{-h_{i,j}\} \times E_{i,j} =   G \times H
 \end{equation}
for $i=1,\dots,M$.  It will suffice to show that this system of tiling equations is aperiodic.

On the one hand, by hypothesis there is a solution $\alpha \colon G \to H$ to the system \eqref{function-system}.  If we then form the graph 
\begin{equation}\label{agraph}
A \coloneqq \{ (x,\alpha(x)): x \in G \} = \biguplus_{x \in G} (\{x\} \times \{\alpha(x)\}) \subset G \times H
\end{equation}
then one has
$$ A \oplus (\{0\} \times H) = \biguplus_{x \in G} \{x\} \times H = G \times H$$
and
\begin{align*} 
A \oplus \biguplus_{j=1}^{J_i} \{-h_{i,j}\} \times E_{i,j} &= \biguplus_{j=1}^{J_i} \biguplus_{x \in G} \{x-h_j\} \times (\alpha(x) + E_{i,j}) \\
&= \biguplus_{j=1}^{J_i} \biguplus_{y \in G} \{y\} \times (\alpha(y+h_j) + E_{i,j}) \\
&= G \times H
\end{align*}
and so $A$ solves the system of tiling equations \eqref{sys-2}, \eqref{sys-1}.

Conversely, suppose that $A \subset G \times H$ solves the system of tiling equations \eqref{sys-2}, \eqref{sys-1}.  From \eqref{sys-2} we see that each vertical line $\{x\} \times H$, $x \in G$ meets $A$ in exactly one point; in other words, $A$ is a graph \eqref{agraph} of some function $\alpha \colon G \to H$.  By the above calculations, we then see that each tiling equation \eqref{sys-1} is equivalent to its functional counterpart \eqref{function-system}, so that $\alpha$ is a solution to the system \eqref{function-system}.  By hypothesis, $\alpha$ is non-periodic, and hence $A$ is non-periodic also.  This establishes the theorem.
\end{proof}

In order to use the above theorem, it is convenient to introduce some notation.

\begin{definition}  Let $G,H$ be abelian groups.  A \emph{$(G,H)$-property} is a property $P$ that may or may not be satisfied by any given function $\alpha \colon G \to H$.  (If one wishes, one can identify such a property with a subset of $H^G$, namely with the set of all $\alpha \in H^G$ that obey property $P$ (after viewing $\alpha$ as a function from $G$ to $H$.)
\end{definition}

\begin{example}
Every functional equation \eqref{function-ex} associated to a given set of parameters $h_1,\dots,h_J \in G$ and $E_1,\dots,E_J \subset H$ can be viewed as an example of a $(G,H)$-property.  The conjunction of any number of $(G,H)$-properties is obviously also  a $(G,H)$-property, so each functional system \eqref{function-system} also describes a $(G,H)$-property.
\end{example}

\begin{definition}[Expressible property]\label{exp}
    We say that a $(G,H)$-property $P$ is \emph{expressible in the language of functional equations}, or \emph{expressible}\footnote{This notion is somewhat analogous to the notion of an \emph{algebraic set} in algebraic geometry, or of a \emph{variety} in universal algebra.  For instance, the claim in Example \ref{conjunction} is analogous to the claim that the intersection of finitely many algebraic sets is again algebraic. On the other hand, unlike algebraic sets which are closed under unions thanks to the integral domain property $ab=0 \iff a=0 \vee b=0$, it is not the case that the disjunction of expressible properties is again expressible, as there is no analogue of the integral domain property in our setting.} for short, if there exists a system \eqref{function-system} of functional equations for some $M \geq 0$ which is obeyed by a function $\alpha \colon G \to H$ if and only if $\alpha$ obeys property $P$.  
\end{definition}

One can think of expressible properties as describing certain types of subshifts of finite type; see also Remark \ref{trans-remark} below.

\begin{definition}[Aperiodic property]\label{aper}
    We say that a $(G,H)$-property $P$ is \emph{aperiodic} if it is satisfiable, but only by non-periodic functions. 
\end{definition}

 The following examples may help illustrate these concepts.

\begin{example}[Empty and full property] The empty property (satisfied by no function $\alpha \colon G \to H$) is expressible, for instance using an empty functional equation \eqref{function-ex} with $J=0$. Similarly, the complete property (satisfied by every function $\alpha \colon G \to H$) is expressible, using the empty system with $M=0$ (or alternatively by using the functional equation $\alpha(x)+H=H$).  Neither property is aperiodic (the former has no solutions, and the latter includes periodic solutions).
\end{example}

\begin{example}[Closure under conjunction]\label{conjunction}  If $P_1,\dots,P_M$ are a finite collection of expressible $(G,H)$-properties, then their conjunction $P_1 \wedge \dots \wedge P_M$ is clearly also an expressible $(G,H)$-property.
\end{example}

\begin{example}[Expressing a clock]\label{clock} Let $\Z/N\Z$ be a cyclic group.  Let us call a function $\alpha \colon \Z \to \Z/N\Z$ a \emph{clock} if it obeys the property
$$ \alpha(x+1) = \alpha(x)+1$$
for all $x \in \Z$, or equivalently if it takes the form $\alpha(x) = x + a \Mod{N}$ for some $a \in \Z/N\Z$.  Then the property of being a clock is expressible by using the single functional equation
$$ (\alpha(x) + \{1\}) \uplus (\alpha(x+1) + (\Z/N\Z \backslash \{0\})) = \Z/N\Z.$$
On the other hand, the property of being a clock is clearly not aperiodic.
\end{example}

For technical reasons we will not actually employ the clock property in our main argument, but instead rely on the following variant.

\begin{example}[Expressing a periodized permutation]\label{perm-expr}  Let $\Z/N\Z$ be a cyclic group.  Let us call a function $\alpha \colon \Z \to \Z/N\Z$ a \emph{periodized permutation} if it is of the form $\alpha(x) = \sigma(x \Mod{N})$ for some permutation $\sigma \colon \Z/N\Z \to \Z/N\Z$.  For instance, every clock is a periodized permutation, but the converse is not true for $N>2$.  We claim that the property of being a periodized permutation is expressible by the single functional equation
$$ (\alpha(x) + \{0\}) \uplus (\alpha(x+1) + \{0\}) \uplus \dots \uplus (\alpha(x+N-1) + \{0\}) = \Z/N\Z$$
for all $x\in \Z$.  Indeed, this equation asserts that the $N$ points $\alpha(x),\dots,\alpha(x+N-1)$ in $\Z/N\Z$ are all distinct, which when applied to both $x$ and $x+1$ implies that $\alpha(x)=\alpha(x+N)$, and also that $\alpha$ is a permutation on any interval $\{x,\dots,x+N-1\}$, which gives the claim.  Obviously, this property is not aperiodic either.
\end{example}

Theorem \ref{functional-to-tile} tells us that if there is an expressible $(G,H)$-property $P$ that is aperiodic, then one can use this to build an aperiodic system of tiling equations.  (Note that the empty system $M=0$ is not aperiodic, so we must have $M \geq 1$.) As a consequence, Theorem \ref{main-multi} is implied by the following statement.

\begin{theorem}[Expressing aperiodicity]\label{main-aperiod}  There exist finite abelian groups $G_1, H$ and an $(\Z^2 \times G_1,H)$-property $P$ that is both expressible and aperiodic.
\end{theorem}

\begin{remark}[Translation invariance]\label{trans-remark}  An expressible property $P$ must necessarily be translation invariant in both the horizontal direction $G$ and the vertical direction $H$.  More precisely, if $\alpha \colon G \to H$ obeys $P$, then so do all the horizontal translates $x \mapsto \alpha(x+h)$ for $h \in G$, and vertical translates $x \mapsto \alpha(x)+u$ for $u \in H$.  This is because each equation in \eqref{function-system} is invariant with respect to these translations. The horizontal invariance (together with the ``local'' nature of the equations \eqref{function-system}) also means that such properties can be interpreted as subshifts of finite type.  The ``dilation lemma'' (see e.g., \cite[Theorem 1.2]{ggrt}) also can force some dilation invariances of expressible properties (at least if the shifts $h_{i,j}$ in \eqref{function-system} are of finite order), although we will not formalize this assertion here.  These invariances are a technical complication for our applications, as they provide some limitations on what types of properties one can hope to express in the language of functional equations.  For instance, one cannot remove the constant $a$ from the clock property in Example \ref{clock} and still retain expressibility.
\end{remark}

It will be convenient to ``coordinatize'' the function $\alpha \colon G \to H$ by replacing it with a \emph{tuple} $(\alpha_w)_{w \in \Wcal}$  of functions $\alpha_w \colon G \to H_w$ into various finite abelian groups $H_w$ indexed by a finite set $\Wcal$.  Note that any such tuple $(\alpha_w)_{w \in \Wcal}$ can be identified with a single function $\alpha \colon G \to \prod_{w \in \Wcal} H_w$, defined by the formula
$$ \alpha(x) \coloneqq (\alpha_w(x))_{w \in \Wcal}$$
for $x \in G$.
Define a \emph{$(G, (H_w)_{w \in \Wcal})$-function} to be a tuple $(\alpha_w)_{w \in \Wcal}$ of functions $\alpha_w \colon G \to H_w$, and define a \emph{$(G, (H_w)_{w \in \Wcal})$-property} to be a property $P$ of a $(G, (H_w)_{w \in \Wcal})$-function $(\alpha_w)_{w \in \Wcal}$.  We will say such a property $P$ is \emph{expressible in the language of functional equations}, or \emph{expressible} for short, if the corresponding $(G,\prod_{w \in \Wcal} H_w)$-property $\tilde P$ of the combined function $\alpha \colon G\to \prod_{w \in \Wcal} H_w$ is expressible, that is to say that there is a system of functional equations
\begin{equation}\label{function-system-tuple}
 \biguplus_{j=1}^{J_i} ((\alpha_w(x + h_{i,j}))_{w \in \Wcal} + E_{i,j}) = \prod_{w \in \Wcal} H_w \hbox{ for all } i=1,\dots,M
\end{equation}
for some $M$, some $J_1,\dots,J_M$, and some $h_{i,j} \in G$ and $E_{i,j} \subset \prod_{w \in \Wcal} H_w$ for $1 \leq i \leq M$ and $1 \leq j \leq J_i$, which is satisfied by the tuple $(\alpha_w)_{w \in \Wcal}$ if and only if the property $P$ holds.  We say that $P$ is \emph{aperiodic} if $\tilde P$ is, or equivalently if there are tuples $(\alpha_w)_{w \in \Wcal}$ obeying property $P$, but any such tuple has at least one of the $\alpha_w$ non-periodic.

\begin{example}[Differing by a constant is expressible]\label{const}  Let $H$ be a finite abelian group.  The property of two functions $\alpha_1, \alpha_2 \colon \Z^2 \to H$ differing by a constant (thus $\alpha_1(x) = \alpha_2(x)+c$ for all $x \in \Z$ and some $c \in H$) can be seen to be an expressible $(\Z^2, (H,H))$-property by using the system of functional equations
\begin{equation}\label{alpha-alpha}
 ((\alpha_1(x),\alpha_2(x)) + \Delta) \uplus (\alpha_1(x+e_i),\alpha_2(x+e_i)) + (H^2 \backslash \Delta)) = H^2
 \end{equation}
for $x \in \Z^2$ and $i=1,2$, where $e_1 = (1,0)$, $e_2 = (0,1)$ is the standard basis of $\Z^2$, and $\Delta$ is the diagonal group
$$ \Delta \coloneqq \{ (a,a): a \in H \}.$$ 
Indeed, the equation \eqref{alpha-alpha} can easily be seen to be equivalent to the equation
$$ \alpha_1(x) - \alpha_2(x) = \alpha_1(x+e_i) - \alpha_2(x+e_i),$$
which is in turn equivalent to the constancy of $\alpha_1-\alpha_2$ since the $e_1,e_2$ generate $\Z^2$.  This property is of course not aperiodic, since one can easily find a pair $(\alpha_1,\alpha_2)$ of periodic functions that differ by a constant.
\end{example}

Recall that Theorem \ref{main-multi} is implied by Theorem \ref{main-aperiod} which can now be reduced to establishing the following claim.

\begin{theorem}[Expressing aperiodicity for a tuple]\label{main-aperiod-tuple}  There exist a finite abelian group $G_1$, a tuple $(H_w)_{w \in \Wcal}$ of finite abelian groups indexed by a finite set $\Wcal$, and a $(\Z^2 \times G_1, (H_w)_{w \in \Wcal})$-property that is both expressible and aperiodic.
\end{theorem}

\begin{remark} By Remark \ref{trans-remark}, an expressible $(G, (H_w)_{w \in \Wcal})$-property must be invariant with joint horizontal translation of a $(G, (H_w)_{w \in \Wcal})$-function $(\alpha_w)_{w \in \Wcal}$ to $(\alpha_w(\cdot+h))_{w \in \Wcal}$ by a shift $h \in G$, and also by independent vertical translations $(\alpha_w+u_w)_{w \in \Wcal}$ of such functions by arbitrary shifts $u_w \in H_W$, and in some cases there are also dilation invariances.  Again, these invariances present some limitations on what properties one can hope to be expressible.
\end{remark}

To add even more flexibility to our framework, it will be convenient to relax the notion of expressibility in which we ``allow existential quantifiers''.  

\begin{definition}[Weak expressibility]\label{weak-express} Let $G$ be a finite abelian group, and let $(H_w)_{w \in \Wcal \uplus \Wcal_0}$ be a tuple of finite abelian groups indexed by the disjoint union of two finite sets $\Wcal, \Wcal_0$.  
\begin{itemize}
    \item[(i)] Given a $(G, (H_w)_{w \in \Wcal \uplus \Wcal_0})$-property $P^*$, we define the \emph{existential quantification} (or \emph{projection}) $P$ of $P^*$ to $(G, (H_w)_{w \in \Wcal})$ to be the $(G, (H_w)_{w \in \Wcal})$-property defined by requiring a $(G, (H_w)_{w \in \Wcal})$-function $(\alpha_w)_{w \in \Wcal}$ to obey $P$ if and only if there exists a $(G, (H_w)_{w \in \Wcal \uplus \Wcal_0})$-function $(\alpha_w)_{w \in \Wcal \uplus \Wcal_0}$ extending the original tuple $(\alpha_w)_{w \in \Wcal}$ that obeys $P^*$.
    \item[(ii)] A $(G, (H_w)_{w \in \Wcal})$-property $P$ is said to be \emph{weakly expressible} if it is the existential quantification of some expressible $(G, (H_w)_{w \in \Wcal \uplus \Wcal_0})$-property $P^*$ for some $\Wcal_0$ disjoint from $\Wcal$.
\end{itemize}  
\end{definition}

Expressible and weakly expressible properties (or more precisely, the sets of tuples obeying such properties) can be viewed as analogous\footnote{They are also somewhat analogous to the notions of an algebraic set and semi-algebraic set respectively in real algebraic geometry, though as before this analogy should not be taken too literally.} to $\Pi^0_0$ and $\Sigma^0_1$ sets respectively in the arithmetic hierarchy; we will not need any analogues of higher order sets in this hierarchy.

Obviously every expressible property is weakly expressible (take $\Wcal_0=\emptyset$).  It is somewhat more challenging to locate a weakly expressible property that is not obviously expressible, but we will do so in later sections.  Observe that if $P$ is an aperiodic weakly expressible $(G, (H_w)_{w \in \Wcal})$-property, then the associated expressible $(G, (H_w)_{w \in \Wcal \uplus \Wcal_0})$-property $P^*$ is necessarily also aperiodic, since it is satisfied by at least one tuple $(\alpha_w)_{w \in \Wcal \uplus \Wcal_0}$ (formed by extending a tuple obeying $P$), and any such tuple must contain a non-periodic function: $\alpha_{w_0}\colon G\to H_{w_0}$ for at least one $w_0\in \Wcal\uplus \Wcal_0$  (because the restriction $(\alpha_w)_{w \in \Wcal}$ does).  Hence, to prove Theorem \ref{main-aperiod-tuple}, it suffices to show:

\begin{theorem}[Weakly expressing aperiodicity for a tuple]\label{main-aperiod-tuple-weak}  There exist a finite abelian group $G_1$, a tuple $(H_w)_{w \in \Wcal}$ of finite abelian groups indexed by a finite set $\Wcal$, and an $(\Z^2 \times G_1, (H_w)_{w \in \Wcal})$-property that is both weakly expressible and aperiodic.
\end{theorem}

To prove this theorem, it will be useful to observe that the class of weakly expressible properties is closed under a number of natural operations, which we now introduce.

\begin{definition}[Lift]\label{lift}
If $G$ is a finitely generated abelian group, $(H_w)_{w \in \Wcal}$ is a tuple of finite abelian groups indexed by a finite set $\Wcal$, $\Wcal_1$ is a subset of $\Wcal$, and $P_1$ is a $(G, (H_w)_{w \in \Wcal_1})$-property, we define the \emph{lift} of $P_1$ to $(G, (H_w)_{w \in \Wcal})$ to be the $(G, (H_w)_{w \in \Wcal})$-property $P$, defined by requiring a $(G, (H_w)_{w \in \Wcal})$-function $(\alpha_w)_{w \in \Wcal}$ to obey $P$ if and only if the $(G, (H_w)_{w \in \Wcal_1})$-function $(\alpha_w)_{w \in \Wcal_1}$ obeys $P_1$.  
\end{definition}

One can think of this operation as that of adding ``dummy functions''  $\alpha_w \colon \Z^2 \times G_1 \to H_w$ for $w \in \Wcal \backslash \Wcal_1$ that play no actual role in the lifted property $P$.

\begin{example}\label{equal-ex}  The $(\Z^2,(H,H,H))$-property of a triple $(\alpha_1,\alpha_2,\alpha_3)$ of functions $\alpha_1,\alpha_2,\alpha_3 \colon \Z^2 \to H$ such that $\alpha_2, \alpha_3$ both differ from $\alpha_1$ by a constant  (i.e., $\alpha_2 = \alpha_1 + c$ and $\alpha_3 = \alpha_1 + c'$ for some $c,c' \in H$) can be viewed as the conjunction of two lifts of (relabelings of) the $(\Z,(H,H))$-property described in Example \ref{const}; one of these lifts will capture the property of $\alpha_1$ and $\alpha_2$ differing by a constant, and another will capture the property of $\alpha_1$ and $\alpha_3$ differing by a constant.  If we take an existential quantification to eliminate the role of $\alpha_1$, we conclude (from Lemma \ref{closure} below) that the $(\Z^2, (H,H))$-property of a pair $\alpha_2,\alpha_3 \colon \Z^2 \to H$ differing by a constant is then weakly expressible (since this occurs if and only if we can locate $\alpha_1 \colon \Z^2 \to H$ such that $\alpha_2,\alpha_3$ both differ from $\alpha_1$ by a constant).  Of course, from Example \ref{const} we already knew that this property was in fact expressible, so this does not give an example of a weakly expressible property that is not expressible.  However, in the next section we shall see several examples in which existential quantification can be used to produce weakly expressible properties that are not obviously expressible.
\end{example}

\begin{example} If one lifts a $(G, (H_w)_{w \in \Wcal_1})$-property $P_1$ to a $(G, (H_w)_{w \in \Wcal})$-property and then takes an existential quantification back to $(G, (H_w)_{w \in \Wcal_1})$, one recovers the original property $P_1$ (since one could simply set all the dummy functions equal to zero).
\end{example}

\begin{definition}[Pullback]   Let $G$ be a finitely generated abelian group, let $G'$ be a subgroup of $G$, and let $(H_w)_{w \in \Wcal}$ be a tuple of finite abelian groups indexed by a finite set $\Wcal$.  If $P'$ is a $(G', (H_w)_{w \in \Wcal})$-property, we define the \emph{pullback} of $P'$ to $(G, (H_w)_{w \in \Wcal})$ to be the $(G, (H_w)_{w \in \Wcal})$-property $P$ defined by requiring a $(G, (H_w)_{w \in \Wcal})$-function $(\alpha_w)_{w \in \Wcal}$ to obey $P$ if and only if the $(G', (H_w)_{w \in \Wcal})$-function $(\alpha_{w,x_0})_{w \in \Wcal}$ defined by $\alpha_{w,x_0}(x') \coloneqq \alpha_w(x_0+x')$ for $x' \in G'$ and $w \in \Wcal$ obeys $P'$ for every choice of base point $x_0 \in G$.
\end{definition}

\begin{example}[Pulling back the clock]\label{pull-clock}  Let $v$ be a non-zero vector in $\Z^2$. Then we can identify $\Z$ with the subgroup $\Z v = \{ n v: n \in \Z\}$ of $\Z^2$.  If we view the clock property from Example \ref{clock} as a $(\Z v, \Z/N\Z)$-property, its pullback to $(\Z^2, \Z/N\Z)$ is the $(\Z^2, \Z/N\Z)$-property of a function $\alpha \colon \Z^2 \to \Z/N\Z$ being a clock along the direction $v$, that is to say for every $x_0 \in \Z^2$ there exists $a_{x_0} \in \Z/N\Z$ such that $\alpha(x_0 + nv) = a_{x_0} + n \Mod{N}$ for every $n \in \Z$.
\end{example}

We now record the closure properties of (weak) expressibility that we will need.

\begin{lemma}[Closure properties of (weak) expressibility]\label{closure}\ 
\begin{itemize}
    \item[(i)]  Any lift of an expressible (resp. weakly expressible) property is also expressible (resp. weakly expressible).
    \item[(ii)]  Any pullback of an expressible (resp. weakly expressible) property is also expressible (resp. weakly expressible).
    \item[(iii)]  The conjunction $P \wedge P'$ of two expressible (resp. weakly expressible) $(G, (H_w)_{w \in \Wcal})$-properties is also expressible (resp. weakly expressible).
    \item[(iv)] Any existential quantification of a weakly expressible property is weakly expressible.
\end{itemize}
\end{lemma}

\begin{proof}  We begin with the expressible case of (i).  Suppose that $P$ is a $(G, (H_w)_{w \in \Wcal})$-property formed by lifting an expressible $(G, (H_w)_{w \in \Wcal_1})$-property $P_1$.
By definition, we can find $M, J_1,\dots,J_M$, and $h_{i,j} \in G$ and $E_{i,j,1} \subset \prod_{w \in \Wcal_1} H_w$ for $1 \leq i \leq M$ and $1 \leq j \leq J_i$ such that a $(G, (H_w)_{w \in \Wcal_1})$-function $(\alpha_w)_{w \in \Wcal_1}$ obeys $P_1$ if and only if it solves the system
$$ \biguplus_{j=1}^{J_i} \left((\alpha_w(x + h_{i,j}))_{w \in \Wcal_1} + E_{i,j,1}\right) = \prod_{w \in \Wcal_1} H_w$$
for all $i=1,\dots,M$ and $x \in G$.  If we then define the lifted sets
$$ E_{i,j} \coloneqq E_{i,j,1} \times \prod_{w \in \Wcal \backslash \Wcal_1} H_w \subset \prod_{w \in \Wcal} H_w$$
for $i=1,\dots,M$ and $j=1,\dots,J_i$, we see from the definitions \ref{lift} and \ref{exp} that a $(G, (H_w)_{w \in \Wcal})$-function $(\alpha_w)_{w \in \Wcal}$ obeys $P$ if and only if
$$ \biguplus_{j=1}^{J_i} \left((\alpha_w(x + h_{i,j}))_{w \in \Wcal} + E_{i,j}\right) = \prod_{w \in \Wcal} H_w$$
for all $i=1,\dots,M$.  The claim follows.

For the weakly expressible case of (i), suppose that $P$ is a $(G, (H_w)_{w \in \Wcal})$-property formed by lifting a weakly expressible $(G, (H_w)_{w \in \Wcal_1})$-property $P_1$.  By Definition \ref{weak-express}, the weakly expressible $(G, (H_w)_{w \in \Wcal_1})$-property $P_1$
is associated to an expressible $(G, (H_w)_{w \in \Wcal_1 \uplus \Wcal_0})$-property $P^*_1$.  By relabeling, we may assume that $\Wcal_0$ is disjoint from $\Wcal$.  The lift $P^*$ of $P^*_1$ to $(G, (H_w)_{w \in \Wcal \uplus \Wcal_0})$ is then an expressible $(G, (H_w)_{w \in \Wcal \uplus \Wcal_0})$-property by the expressible case of (i), and can be seen to be associated to $P$ in the sense of   \defref{weak-express} by expanding out the definitions. Thus $P$ is weakly expressible as desired.

Now we establish the expressible case of (ii).   Suppose that $P$ is a $(G, (H_w)_{w \in \Wcal})$-property formed by pulling back an expressible $(G', (H_w)_{w \in \Wcal})$-property $P'$. 
By definition, we can find $M, J_1,\dots,J_M$, and $h_{i,j} \in G'$ and $E_{i,j} \subset \prod_{w \in \Wcal} H_w$ for $1 \leq i \leq M$ and $1 \leq j \leq J_i$ such that a $(G', (H_w)_{w \in \Wcal})$-function $(\alpha_w)_{w \in \Wcal}$ obeys $P'$ if and only if it solves the system
\begin{equation}\label{xm}
 \biguplus_{j=1}^{J_i} \left((\alpha_w(x + h_{i,j}))_{w \in \Wcal} + E_{i,j}\right) = \prod_{w \in \Wcal} H_w
 \end{equation}
for all $i=1,\dots,M$ and $x \in G'$.  By expanding out the definitions, we then see that a $(G, (H_w)_{w \in \Wcal})$-function $(\alpha_w)_{w \in \Wcal}$ obeys $P$ if and only if it obeys the same system of equations \eqref{xm} for $i=1,\dots,M$, but now with $x$ ranging over $G$ instead of $G'$.  Thus $P$ is also expressible as required.

For the weakly expressible case of (ii), suppose that $P$ is a $(G, (H_w)_{w \in \Wcal})$-property formed by pulling back a weakly expressible $(G', (H_w)_{w \in \Wcal})$-property $P'$.  By Definition \ref{weak-express}, $P'$ is associated to some expressible $(G', (H_w)_{w \in \Wcal \uplus \Wcal_0})$-property $(P')^*$.  If we let $P^*$ be the pullback of $(P')^*$ to $(G, (H_w)_{w \in \Wcal \uplus \Wcal_0})$, then $P^*$ is expressible by the expressible case of (ii), and can be seen to be associated to $P$ in the sense of Definition \ref{weak-express} by expanding out the definitions. Thus $P$ is weakly expressible as desired.

The expressible case of (iii) is trivial (and was already noted in Remark \ref{conjunction}.  Now suppose that $P, P'$ are weakly expressible $(G, (H_w)_{w \in \Wcal})$-properties.  By Definition \ref{weak-express}, $P$ is associated with an expressible $(G, (H_w)_{w \in \Wcal \uplus \Wcal_0})$-property $P^*$, and $P'$ is similarly associated with an expressible $(G, (H_w)_{w \in \Wcal \uplus \Wcal'_0})$-property $(P')^*$.  By relabeling, we can assume that $\Wcal_0$ and $\Wcal'_0$ are disjoint.  Let $Q^*$ be the $(G, (H_w)_{w \in \Wcal \uplus \Wcal_0 \uplus W'_0})$-property formed by lifting both $P^*$ and $(P')^*$ to $(G, (H_w)_{w \in \Wcal \uplus \Wcal_0 \uplus W'_0})$ and then taking their conjunction.  By the previously established parts of this lemma, $Q^*$ is expressible, and  can be seen to be associated to $P \wedge P'$ in the sense of Definition \ref{weak-express} by expanding out the definitions. Thus $P \wedge P'$ is weakly expressible as desired.

Finally, (iv) is immediate from Definition \ref{weak-express}, after observing that an existential quantification of an existential quantification is again an existential quantification.
\end{proof}

\section{A library of (weakly) expressible properties}\label{sec:library}

In view of Lemma \ref{closure}, a natural strategy to establish Theorem \ref{main-aperiod-tuple-weak} is to first build up a useful ``library'' of (weakly) expressible $(G, (H_w)_{w \in \Wcal})$-properties for various choices of $G$ and $(H_w)_{w \in \Wcal}$, with the aim of combining them via various applications of Lemma \ref{closure} to create more interesting (and ultimately, aperiodic) examples of weakly expressible properties (analogously to how one can create a complex computer program by combining more fundamental library routines together in various ways).  For instance, the clock in Example \ref{clock} can be regarded as one entry in this library, as can the property of being a periodized permutation as discussed in Example \ref{perm-expr}, or the property of differing by a constant as discussed in Example \ref{const}.  The final objective is to then ``program'' such a combination of properties in the library that necessarily generates an non-periodic function. In fact we will achieve this by ``programming'' a certain type of ``Sudoku puzzle'' that can be solved, but only in a non-periodic fashion.

\begin{example}  Consider the $(\Z^2, \Z/N\Z)$-property $P$ of a function $\alpha \colon \Z^2 \to \Z/N\Z$ being of the form $\alpha(x,y) = x+y+c$ for all $(x,y) \in \Z^2$ and some $c \in \Z/N\Z$.  This is equivalent to $\alpha$ being a clock along the direction $e_1=(1,0)$ and simultaneously being a clock along the direction $e_2=(0,1)$, in the sense of Example \ref{pull-clock}.  Thus this property $P$ is the conjunction of two pullbacks of the clock property; since we know from Example \ref{clock} that the clock property is expressible, we conclude from several applications of Lemma \ref{closure} that this property $P$ is also expressible.  However, this property is not aperiodic, and so does not complete the proof of Theorem \ref{main-aperiod-tuple-weak}.
\end{example}

\begin{example}  The $(\Z, (\Z/N\Z)_{w =1,2})$-property of two functions $\alpha_1,\alpha_2 \colon \Z \to \Z/N\Z$ being periodized permutations that differ by a constant is expressible, as can be seen from Lemma \ref{closure} after lifting Example \ref{perm-expr} twice and taking conjunctions of those lifts with Example \ref{const}.  Again, this property is not aperiodic, and so does not complete the proof of Theorem \ref{main-aperiod-tuple-weak}.
\end{example}

\subsection{Expressing linear constraints}

One basic property that we will add to our library is the ability to express linear constraints (up to constants) between different functions $\alpha_w$, which significantly generalizes Example \ref{const}.  The basic relation is

\begin{proposition}[Expressing constancy modulo a subgroup]\label{const-mod-subgrp}  Let $G$ be a finitely generated abelian group, let $H$ be a finite abelian group, and let $H'$ be a subgroup of $H$.  Then the $(G,H)$-property of a $(G,H)$-function $\alpha$ taking values in a single coset $c+H'$ of $H'$ (i.e., there exists $c \in H$ such that $\alpha(x) \in c+H'$ for all $x \in G$) is expressible.
\end{proposition}

\begin{proof}  Let $e_1,\dots,e_d$ be a set of generators for $G$. Similarly to Example \ref{const}, we consider the functional equation
$$ (\alpha(x) + H') \uplus \left(\alpha(x+e_i) + (H \backslash H')\right) = H$$
for all $i=1,\dots,d$ and $x \in G$, and some unknown function $\alpha \colon G \to H$.  This equation can be equivalently expressed as
$$ \alpha(x) = \alpha(x+e_i) \Mod{H'},$$
that is to say $\alpha(x)$ and $\alpha(x+e_i)$ lie in the same coset of $H'$.  Since the $e_i$ generate $G$, this is equivalent to $\alpha$ lying in a single coset of $H'$, as claimed.
\end{proof}

We isolate two useful corollaries of this proposition:

\begin{corollary}[Expressing periodicity]\label{express-period}  Let $G$ be a finitely generated abelian group, let $H$ be a finite abelian group, and let $G'$ be a subgroup of $G$.  Then the $(G,H)$-property that a $(G,H)$-function $\alpha$ is $G'$-periodic in the sense that $\alpha(x+h) = \alpha(x)$ for all $x \in G$ and $h \in G'$, is expressible.
\end{corollary}

\begin{proof}  From Proposition \ref{const-mod-subgrp} with $G$ replaced by $G'$ and $H'$ replaced by $\{0\}$, we see that the $(G',H)$ property of being a constant $(G',H)$-function is expressible.  Pulling back from $(G',H)$ to $(G,H)$ using Lemma \ref{closure}(ii), we obtain the claim.
\end{proof}

\begin{corollary}[Expressing linear constraints]\label{express-linear} Let $G$ be a finitely generated abelian group, let $\Z/N\Z$ be a cyclic group, and let $c_1,\dots,c_W \in \Z/N\Z$ be coefficients.  Then the $(G,(\Z/N\Z)_{w=1,\dots,W})$-property of a tuple $\alpha_1,\dots,\alpha_W \colon G \to \Z/N\Z$ of functions obeying the linear relation
\begin{equation}\label{ca}
c_1 \alpha_1(x) + \dots + c_W \alpha_W(x) = c
\end{equation}
for all $x \in G$ and some constant $c \in \Z/N\Z$, is expressible.
\end{corollary}

\begin{proof}
    We can view $(\alpha_w)_{w=1,\dots,W}$ as a single $(G, (\Z/N\Z)^W)$-function.  The linear relation \eqref{ca} is then equivalent to this function lying in a single coset of the group
    $$ H' \coloneqq \{ (a_1,\dots,a_W) \in (\Z/N\Z)^W: c_1 a_1 + \dots + c_w a_w = 0 \}.$$
    The claim now follows from Proposition \ref{const-mod-subgrp}.
\end{proof}

\begin{remark} Note that Example \ref{const} is essentially the special case of Corollary \ref{express-linear} with $W=2$, $c_1=1$, and $c_2=-1$.  The presence of the constant $c$ in \eqref{ca} is unfortunately necessary due to the translation invariance mentioned in Remark \ref{trans-remark}.  We remark that a variant of Corollary \ref{express-linear} (in which one did not tile the whole group, and was thus able to set $c$ to zero) was implicitly used in our previous work \cite[\S 6]{GT2}.
\end{remark}

\begin{example}  Let $\Z/N\Z$ be a cyclic group, and consider the $(\Z^2, (\Z/N\Z)_{w=1,2,3})$-property of a triple of functions $\alpha_1,\alpha_2,\alpha_3 \colon \Z^2 \to \Z/N\Z$ obeying the properties
\begin{align*}
\alpha_1(x,y)  &= \alpha_1(x+1,y) \\
\alpha_2(x,y)  &= \alpha_2(x,y+1) \\
\alpha_3(x,y)  &= \alpha_3(x+1,y-1) \\
\alpha_1(x,y) + \alpha_2(x,y) &= \alpha_3(x,y)
\end{align*}
for all $(x,y) \in \Z^2$, namely, $\alpha_1,\alpha_2,\alpha_3$ are periodic along the directions $(1,0), (0,1), (1,-1)$ respectively, and that $\alpha_1+\alpha_2=\alpha_3$. Thus, this property is expressible.  It is not difficult to show that the solutions to this system of equations are given by $\alpha_1(x,y) = \phi(y) + c_1$, $\alpha_2(x,y) = \phi(x) + c_2$, $\alpha_3(x,y) = \phi(x+y) + c_3$ for some homomorphism $\phi \colon \Z \to \Z/N\Z$ and some constants $c_1,c_2,c_3 \in \Z/N\Z$.  Thus the property of $\alpha_1,\alpha_2,\alpha_3$ taking this form is expressible.  Applying existential quantification to eliminate the role of $\alpha_2, \alpha_3$, we conclude that the $(\Z^2, \Z/N\Z)$-property of a function $\alpha \colon \Z^2 \to \Z/N\Z$ taking the form $\alpha(x,y) = f(y)$ for some affine function $f \colon \Z \to \Z/N\Z$ (i.e., the sum of a homomorphism and a constant), is weakly expressible. This is our first example of a property which is weakly expressible, but which is not obviously expressible.
\end{example}

\section{Expressing boolean functions}\label{sec:boolean}

Thus far we have been considering properties of functions $\alpha_w \colon G \to H_w$ which can range over the entirety of a finite abelian group $H_w$.  In order to be able to express boolean operations (as in \cite[\S 5]{GT2}), we will need to start expressing properties of functions that take on only two values $\{a,b\}$ in a larger ambient group $H_w$ (which we will take to be a cyclic $2$-group $\Z/2^M\Z$).  To do this, we introduce the following definition.

\begin{definition}[Boolean function]  Let $G$ be a finitely generated abelian group, let $e$ be an element of $G$ of order $2$, let $\Z/2^M\Z$ be a cyclic $2$-group for some $M \geq 1$, and let $a,b$ be distinct elements of $\Z/2^M\Z$ of opposite parity (thus one of the $a,b$ is even and the other is odd).  A function $\alpha \colon G \to \Z/2^M\Z$ is \emph{$(e,\{a,b\})$-boolean} if it takes values in $\{a,b\}$, and furthermore obeys the alternating property
\begin{equation}\label{alternating}
\alpha(x + e) = a+b - \alpha(x)
\end{equation}
for all $x \in G$; i.e., for each $x \in G$, $\alpha(x)$ takes one of the values $a,b$, and $\alpha(x+e)$ takes the other value.  In particular, $\{a,b\}$ is equal to the image $\alpha(G)$ of $\alpha$.

A $(e,\{a,b\})$-boolean function $\alpha$ is said to be \emph{compatible} with a $(e,\{a',b'\})$-boolean function $\alpha'$ if $\{a',b'\}$ is a translate of $\{a,b\}$, or equivalently if the image $\alpha(G)$ of $\alpha$ is a translate of the image $\alpha'(G)$ of $\alpha'$.
\end{definition}

We restrict to $2$-groups $\Z/2^M\Z$ here because in later arguments it will be important to exploit the fact that all odd elements of such groups are invertible (with respect to the usual ring structure on cyclic groups), and in particular have order $2^M$ equal to the order of the group.  This will also be the main reason why we will work with ``$2$-adic Sudoku puzzles'' in later sections, as opposed to Sudoku puzzles in odd characteristic which are slightly easier to analyze.

\begin{proposition}[Expressing a single boolean function]\label{bool-expr}  Let $G$ be a finitely generated abelian group, let $e$ be an element of $G$ of order two, and let $\Z/2^M\Z$ be a cyclic $2$-group for some $M \geq 1$.  Then the $(G, \Z/2^M\Z)$-property of being $(e,\{a,b\})$-boolean for some distinct $a,b \in \Z/2^M\Z$ of opposite parity, is expressible.
\end{proposition}

\begin{proof}  Let $e_1,\dots,e_r$ be a set of generators for $G$, and consider the $(G, \Z/2^M\Z)$-property of a function $\alpha \colon G \to \Z/2^M\Z$ obeying the functional equation
\begin{equation}\label{funct-1}
\left(\alpha(x) + 2\Z/2^M\Z\right) \uplus \left(\alpha(x+e) + 2\Z/2^M\Z\right) = \Z/2^M\Z
\end{equation}
for all $x \in G$, as well as the equations
\begin{equation}\label{funct-2}
\biguplus_{y = x, x+e} \left((\alpha(y+e_i) + \{0\}) \uplus (\alpha(y) + (2\Z/2^M\Z \backslash \{0\}))\right) = \Z/2^M\Z
\end{equation}
for all $x \in G$ and $i=1,\dots,r$.

Suppose that $\alpha$ obeys this system \eqref{funct-1}, \eqref{funct-2}.  Since $2\Z/2^M\Z$ is an index two subgroup of $\Z/2^M\Z$, we see that for each $x$, the pair $(\alpha(x), \alpha(x+e))$ must consist of an even element $a(x)$ and an odd element $b(x)$ of $\Z/2q\Z$.  On the other hand, from comparing \eqref{funct-2} with \eqref{funct-1} we have for each $x \in G$ and $i=1,\dots,r$ that
$$ (\alpha(x)+\{0\}) \uplus (\alpha(x+e)+\{0\}) = (\alpha(x+e_i) + \{0\}) \uplus (\alpha(x+e+e_i) + \{0\})$$
or equivalently that $a(x) = a(x+e_i)$ and $b(x)=b(x+e_i)$.  Since the $e_1,\dots,e_r$ generate $G$, we conclude that $a(x)=a$, $b(x)=b$ are constant in $x$, and $\alpha$ is $(e,\{a,b\})$-boolean.  Conversely, if $\alpha$ is $(e,\{a,b\})$-boolean, we can reverse the above arguments and conclude the functional equations \eqref{funct-1}, \eqref{funct-2}.  The claim follows.
\end{proof}

Let $G$ be a finitely generated abelian group, let $e$ be an element $G$ of order two, let $\Z/2q\Z$ be a cyclic group of even order, and let $W \geq 1$.  
By the above proposition and Lemma \ref{closure}, one can express the $(G, (\Z/2q\Z)_{w=1,\dots,W})$-property of a tuple $\alpha_1,\dots,\alpha_W \colon G \to \Z/2q\Z$ of functions being such that each $\alpha_i$ is $(e,\{a_i,b_i\})$-periodic for some $a_i, b_i \in \Z/2q\Z$ of different parity.  However, this property does not force the boolean functions to be compatible; in other words, it does not require that the $\{a_i,b_i\}$ are translates of each other. This is a new difficulty that was not present in our previous work \cite{GT2}, where we could enforce this compatibility by only tiling a subset of $H$ rather than the full group $H$.  In our context, the desired compatibility will be achieved with the assistance of the following elementary lemma.

\begin{lemma}[An equation to force boolean compatibility]\label{force-compat}  Let $\Z/2^M\Z$ be a cyclic $2$-group for some $M \geq 2$, and let $\{a,b\}, \{c,d\}, \{f,g\}, \{h',k'\}, \{h'',k''\}$ be pairs of elements of $\Z/2^M\Z$ of different parity.  Let $z \in \Z/2^M\Z$ be such that
\begin{equation}\label{z-sum}
(a+b) + (h'+k') + (h''+k'') = 2(c+d) + (f+g) + 2z.
\end{equation}
Suppose also that for any triple $(\alpha,\tau,\tau') \in \{a,b\} \times \{h',k'\} \times \{h'',k''\}$ there exists $(\beta,\gamma) \in \{c,d\}\times\{f,g\}$ solving the equation
\begin{equation}\label{alphatau-eq}
\alpha + \tau + \tau' = 2\beta + \gamma + z.
\end{equation}
Then the sets $\{a,b\}, \{h',k'\}, \{h'',k''\}$ are translates of each other.
\end{lemma}

\begin{proof}  Observe that we can translate any of the pairs $\{a,b\}, \{h',k'\}, \{h'',k''\}$ by some shift in $\Z/2^M\Z$, so long as we also shift $z$ by the same shift.  So we may normalize $a=h'=h''=0$.  By \eqref{alphatau-eq} we may then find $(\beta,\gamma) \in \{c,d\}\times\{f,g\}$ such that $0 = 2\beta+\gamma+z$.  By shifting $\{c,d\}$ by $-\beta$, $\{f,g\}$ by $-\gamma$, and replacing $z$ with $0$, we may thus also normalize $\beta=\gamma=z=0$, so without loss of generality $c=f=z=0$.  Thus we now have
\begin{equation}\label{bkk-eq}
b + k' + k'' = 2d + g
\end{equation}
and for any triple $(\alpha,\tau,\tau') \in \{0,b\} \times \{0,k'\} \times \{0,k''\}$ there exists $(\beta,\gamma) \in \{0,d\}\times\{0,g\}$ such that
$$ \alpha+\tau+\tau' = 2 \beta + \gamma.$$
In particular
$$ b, k', k'' \in \{ 0, 2d, g, g+2d\}.$$
By the hypothesis of distinct parities, $b,k',k'',d,g$ are all odd.  Thus in fact we must have
$$ b, k', k'' \in \{ g, g+2d\}.$$
If $b=k'=k''$ then $\{0,b\}, \{0,k'\}, \{0,k''\}$ are translates of each other as desired.  There are only two remaining cases:
\begin{enumerate}
    \item If two of the $b,k',k''$ are equal to $g$ and the third is equal to $g+2d$, then from \eqref{bkk-eq} we have $3g+2d=2d+g$, which is absurd since $g$ is odd and $M \geq 2$.
    \item If one of the $b,k',k''$ is equal to $g$ and the other two are equal to $g+2d$, then from \eqref{bkk-eq} we have $3g+4d=2d+g$, so in particular $g+2d = -g$ and so $\{0,g\}$ and $\{0,g+2d\}$ are translates of each other.  Thus $\{0,b\}, \{0,k'\}, \{0,k''\}$ are translates of each other as desired.     
\end{enumerate}
\end{proof}

\begin{definition}[Compatible boolean property]\label{compbool}
Let $G$ be a finitely generated abelian group, let $e, e', e''$ be elements of $G$ that generate a copy of $(\Z/2\Z)^3$ (so that $e,e',e''$ are of order $2$ and linearly independent over $\Z/2\Z$), let $\Z/2^M\Z$ be a cyclic $2$-group for some $M \geq 2$, and let $W \geq 1$. 
     We say that a tuple $(\alpha_w)_{w=1,\dots,W}$ of functions $\alpha_1,\dots,\alpha_W \colon G \to \Z/2^M\Z$ obeys the \emph{compatible boolean property} $P$ (with parameters $e,e',e''$) if  each $\alpha_i$ is $\langle e', e''\rangle$-periodic (thus $\alpha_i(x+e')=\alpha_i(x+e'')=\alpha_i(x)$ for all $x \in G$) and $(e,\{a_i,b_i\})$-boolean for some $a_i, b_i \in \Z/2^M\Z$ of different parity, and additionally that the $\alpha_i$ are compatible (i.e., the $\{a_i,b_i\}$ are translates of each other).
\end{definition}

We can exploit \lemref{force-compat} as follows.

\begin{proposition}[Expressing multiple compatible boolean functions]\label{multi-express}   The compatible boolean property $P$ is a weakly expressible $(G, (\Z/2^M\Z)_{w=1,\dots,W})$-property.
\end{proposition}

\begin{proof}  For sake of notation we just demonstrate this for $W=2$; the general case is similar, and in any event follows from the $W=2$ case by applying Lemma \ref{closure} in a similar spirit to Example \ref{equal-ex}.

Let $\alpha_1, \alpha_2 \colon G \to \Z/2^M\Z$ be functions.  
We introduce some auxiliary functions
$$ \beta_1, \beta_2, \gamma_1, \gamma_2, \tau', \tau'' \colon G \to \Z/2^M\Z.$$
Consider the $(G, (\Z/2^M\Z)_{w=1,\dots,8})$-property $P^*$ that the tuple $(\alpha_1,\alpha_2,\beta_1,\beta_2,\gamma_1,\gamma_2,\tau)$ obeys the following properties for $i=1,2$:
\begin{itemize}
    \item[(i)] $\alpha_i$ is $\langle e',e''\rangle$-periodic and $(e,\{a_i,b_i\})$-boolean for some $a_i,b_i \in \Z/2^M\Z$ of different parity.
    \item[(ii)] $\beta_i$ is $(e,\{c_i,d_i\})$-boolean for some $c_i,d_i \in \Z/2^M\Z$ of different parity.
    \item[(iii)] $\gamma_i$ is $(e,\{f_i,g_i\})$-boolean for some $f_i,g_i \in \Z/2^M\Z$ of different parity.
    \item[(iv)] $\tau'$ is $\langle e+e', e''\rangle$-periodic and $(e,\{h',k'\})$-boolean for some $h',k' \in \Z/2^M\Z$ of different parity.
    \item[(v)] $\tau''$ is $\langle e', e+e''\rangle$-periodic and $(e,\{h'',k''\})$-boolean for some $h'',k'' \in \Z/2^M\Z$ of different parity.
    \item[(vi)] There is a constant $z_i \in \Z/2^M\Z$ such that $\alpha_i(x) + \tau'(x) + \tau''(x) = 2 \beta_i(x) + \gamma_i(x) + z_i$ for all $x \in G$.
\end{itemize}
From several applications of Corollary \ref{express-period}, Corollary \ref{express-linear},  Proposition \ref{bool-expr}, and Lemma \ref{closure} we already know that $P^*$ is expressible.  To conclude the proposition, it suffices to show that the compatible boolean property $P$ is the existential quantification of $P^*$.

We first show that any pair $(\alpha_1,\alpha_2)$ obeying the compatible boolean property $P$ can be lifted to a octuplet $(\alpha_1,\alpha_2,\beta_1,\beta_2,\gamma_1,\gamma_2,\tau,\tau')$ obeying $P^*$.  By hypothesis and \defref{compbool}, each $\alpha_i$ is already $\langle e', e'' \rangle$-periodic and $(e,\{a_i,b_i\})$-boolean, where $a_i,b_i \in \Z/2^M\Z$ are of different parity and with $\alpha_1,\alpha_2$ compatible.  By applying independent translations to the compatible boolean functions $\alpha_1,\alpha_2$ (which we can do by Remark \ref{trans-remark}), we may normalize $\{a_1,b_1\}=\{a_2,b_2\} = \{0,b\}$ for some odd $b \in \Z/2^M\Z$.  Next, let $\tau'$ be an arbitrary $\langle e+e', e''\rangle$-periodic and $(e,\{0,b\})$-boolean function; such a function can be constructed by arbitrarily partitioning $G$ into cosets $x + \langle e,e',e''\rangle$ with a marked point $x$ and then setting 
$$\tau(x + re + se' + te'') = b 1_{r=s}$$
on each such coset for $r,s,t \in \Z/2\Z$.  Similarly we can let $\tau''$ be an arbitrary $\langle e', e+e''\rangle$-periodic and $(e,\{0,b\})$-boolean function.
For each $i=1,2$, the function $\alpha_i+\tau'+\tau''$ then takes values in $\{0, b, 2b, 3b\}$, and obeys the alternating property
$$ (\alpha_i+\tau'+\tau'')(x+e) = 3b - (\alpha_i+\tau'+\tau'')(x)$$
for all $x \in G$.  Note also that the quantities $0,b,2b,3b$ are all distinct since $b$ is odd and $M \geq 2$.  By binary expansion, we may thus decompose
$$ (\alpha_i+\tau'+\tau'')(x) = 2 \beta_i(x) + \gamma_i(x)$$
for some unique functions $\beta_i, \gamma_i \colon G \to \{0,b\}$.  It is easy to verify that these functions are $(e,\{0,b\})$-boolean, and so the octuplet $(\alpha_1,\alpha_2,\beta_1,\beta_2,\gamma_1,\gamma_2,\tau,\tau')$ obeys $P^*$ as required (with $z_i=0$ in (vi)).

Conversely, suppose that we have an octuplet $(\alpha_1,\alpha_2,\beta_1,\beta_2,\gamma_1,\gamma_2,\tau,\tau')$ obeying property $P^*$.  Applying (vi) to $x$ and $x+e$ and summing, we have
$$ \sum_{y = x, x+e} (\alpha_i(y) + \tau'(y) + \tau''(y)) = \sum_{y = x,x+e} (2\beta_i(y) + \gamma_i(y)) + 2z_i$$
for any $x \in G$ and $i=1,2$.  Using the boolean nature of the functions $\alpha_1,\alpha_2,\beta_1,\beta_2,\gamma_1,\gamma_2,\tau,\tau'$ in the direction $e$, we conclude that
$$ (a_i+b_i) + (h'+k') + (h''+k'') = 2(c_i+d_i) + (f_i+g_i) + 2z_i$$
for $i=1,2$.

Let $i=1,2$.  By (i), (iv), (v), we see that for any $x \in G$, the triple $(\alpha_i(x), \tau'(x), \tau''(x))$ takes values in the eight-element set $\{a_i,b_i\} \times \{h',k'\} \times \{h'',k''\}$.  Furthermore, shifting $x$ by $e'$ changes the value of $\tau'(x)$ but not $\alpha(x),\tau''(x)$; shifting $x$ by $e''$ changes the value of $\tau''(x)$ but not $\alpha(x), \tau'(x)$; and shifting $x$ by $e+e'+e''$ changes the value of $\alpha_i(x)$ but not $\tau'(x),\tau''(x)$.  We conclude that all eight of the elements of $\{a_i,b_i\} \times \{h',k'\} \times \{h'',k''\}$ are actually representable in the form $(\alpha_i(x), \tau'(x), \tau''(x))$ for some $x \in G$.  By (vi), we conclude that every element $(\alpha_i, \tau, \tau')$ of $\{a_i,b_i\} \times \{h',k'\} \times \{h'',k''\}$ has a representation $\alpha_i + \tau + \tau' = 2\beta_i + \gamma_i + z_i$ for some $(\beta_i,\gamma_i) \in \{c_i,d_i\} \times \{f_i,g_i\}$.  We can now apply Lemma \ref{force-compat} to conclude that $\{a_i,b_i\}, \{h',k'\}, \{h'',k''\}$ must be translates of each other. Thus, both $\alpha_1,\alpha_2$ are compatible with the $\tau',\tau''$.  By transitivity, this implies that $\alpha_1$ is compatible with $\alpha_2$, and hence the compatible boolean property $P$ holds as required.
\end{proof}

Let $G, M, e, e', e''$ be as in the above proposition.  If $\alpha_1,\dots,\alpha_W \colon G \to \Z/2^M\Z$ obey the compatible boolean property, then (after permuting $a_i,b_i$ as necessary) each $\alpha_i$ is $\langle e',e''\rangle$-invariant and is $(e, \{a_i,a_i+b\})$-boolean for some $a_i \in \Z/2^M\Z$ and some odd $b$ independent of $i$.  Thus we have representations
\begin{equation}\label{alphai} 
\alpha_i(x) = a_i + b \tilde \alpha_i(x)
\end{equation}
for all $x \in G$ and $i=1,\dots,W$, where the normalized boolean functions $\tilde \alpha_i \colon G \to \Z/2^M\Z$ are $\langle e',e''\rangle$-invariant and $(e,\{0,1\})$-boolean.  Note that the $a_i,b$ are only unique up to the reflection symmetry
$$ (a_1,\dots,a_W,b) \mapsto (a_1+b,\dots,a_W+b, -b)$$
that effectively replaces the normalized boolean functions $\tilde \alpha_i$ with their reflections $1-\tilde \alpha_i$.

\begin{definition}[Property $P_\Omega$]
    Let $\Omega$ be a subset of $\{0,1\}^W$ which is symmetric with respect to the reflection 
$$ (y_1,\dots,y_W) \mapsto (1-y_1,\dots,1-y_W).$$
We say that a $(G, (\Z/2^M\Z)_{w=1,\dots,W})$-function $(\alpha_1,\dots,\alpha_W)$ obeys property $P_\Omega$ if it obeys the compatible boolean property $P$, and furthermore that the normalized functions $\tilde \alpha_1,\dots,\tilde \alpha_W$ obey the boolean constraint
$$ (\tilde \alpha_1(x), \dots, \tilde \alpha_W(x)) \in \Omega$$
for all $x \in G$.
\end{definition}  
Note that from the symmetry hypothesis,  it does not matter which of the two available normalizations $\tilde \alpha_i$ of the $\alpha_i$ are used here. 
 Importantly, such relations are weakly expressible when $M$ is large enough:

\begin{proposition}[Expressing symmetric boolean constraints]\label{sym-bool}  Let $G$ be a finitely generated abelian group, let $e, e', e''$ be elements of $G$ that generate a copy of $(\Z/2\Z)^3$, let $\Z/2^M\Z$ be a cyclic $2$-group for some $M \geq 2$, and let $W \geq 1$.  Let $\Omega$ be a symmetric subset of $\{0,1\}^W$.  If $2^M > 2W+4$, then the $(G, (\Z/2^M\Z)_{w=1,\dots,W})$-property $P_\Omega$ is weakly expressible.
\end{proposition}

\begin{proof}  This will be a variant of the arguments in \cite[\S 6]{GT2}.  By increasing $W$ by one or two if necessary (and relaxing $2^M > 2W+4$ to $2^M > 2W$) using Lemma \ref{closure}(iv), we may assume without loss of generality that $W$ is odd with $W \geq 3$.  The symmetric set $\Omega$ can be expressed as the intersection of a finite\footnote{This number, while finite, could be very large (exponentially large in $W$).  This exponential growth will cause the dimension $d$ in Theorem \ref{main} to be enormous, as one has to perform the conjunction of exponentially many expressible properties.  A substantially more efficient approach will be needed here if one wishes to obtain a more reasonable value for the dimension $d$.} number of symmetric sets of the form
\begin{equation}\label{single}
 \{0,1\}^W \backslash \{ (\epsilon_1,\dots,\epsilon_W), (1-\epsilon_1,\dots,1-\epsilon_W) \}
\end{equation}
for some $\epsilon_1,\dots,\epsilon_W \in \{0,1\}$.  By Lemma \ref{closure}(iii), it thus suffices to verify the claim for $\Omega$ of the form \eqref{single}.  

We introduce some auxiliary functions $\beta_1,\dots,\beta_{W-2} \colon G \to \Z/2^M\Z$, and let $P^*_\Omega$ be the $(G, (\Z/2^M\Z)_{w=1,\dots,2W-2})$-property that a tuple $(\alpha_1,\dots,\alpha_W,\beta_1,\dots,\beta_{W-2})$ of functions from $G$ to $\Z/2^M\Z$ obey the following properties:
\begin{itemize}
    \item[(i)]  $(\alpha_1,\dots,\alpha_W,\beta_1,\dots,\beta_{W-2})$ obeys the compatible boolean property $P$ (with $W$ replaced by $2W-2$);
    \item[(ii)] There is a constant $z \in \Z/2^M\Z$ such that
    $$ (-1)^{\epsilon_1} \alpha_1(x)+\dots+ (-1)^{\epsilon_W} \alpha_W(x) = \beta_1(x) + \dots + \beta_{W-2}(x) + z$$
    for all $x \in G$.
\end{itemize}
From Proposition \ref{multi-express}, Corollary \ref{express-linear}, and Lemma \ref{closure}, the property $P^*_\Omega$ is weakly expressible.  By Lemma \ref{closure}(iv), it thus suffices to show that $P_\Omega$ is the existential quantification of $P^*_\Omega$.

We first show that any tuple $(\alpha_1,\dots,\alpha_W)$ obeying $P_\Omega$ can be extended to a tuple $(\alpha_1,\dots,\alpha_W,\beta_1,\dots,\beta_{W-2})$ obeying $P^*_\Omega$.  By hypothesis, we can write the $\alpha_i$ in the form \eqref{alphai} for some $a_i, b \in \Z/2^M\Z$ with $b$ odd, and some $(e,\{0,1\})$-boolean and $\langle e', e'' \rangle$-periodic functions $\tilde \alpha_1,\dots,\tilde \alpha_W \colon G \to \{0,1\}$.  In particular
$$ (-1)^{\epsilon_1} \alpha_1+\dots+ (-1)^{\epsilon_W} \alpha_W = b (\tilde \alpha_{1, \epsilon_1} + \dots + \tilde \alpha_{W, \epsilon_W}) + z_0$$
for some constant $z_0 \in \Z/2^M\Z$, where $\tilde \alpha_{i,\epsilon_i} \colon G \to \{0,1\}$ is the $(e,\{0,1\})$-boolean and $\langle e', e'' \rangle$-periodic function $R_{\epsilon_i}(\tilde \alpha_i)$, where for $a=0,1$
\begin{equation}\label{reflection}
R_{a}(x)\coloneqq a+ (-1)^{a}x,\quad x\in G.
\end{equation}
By the choice \eqref{single} of $\Omega$, we see that for every $x$, the tuple $(\tilde \alpha_{1,\epsilon_1}(x),\dots,\tilde \alpha_{W,\epsilon_W}(x))$ is an element of the cube $\{0,1\}^W$ that avoids both $(0,\dots,0)$ and $(1,\dots,1)$.  In particular, we have
$$ \tilde \alpha_{1, \epsilon_1}(x) + \dots + \tilde \alpha_{W, \epsilon_W}(x) = b f(x)$$
for some $f(x) \in \{1,\dots,W-1\}$, which is well-defined since the odd element $b$ of $\Z/2^M\Z$ has order $2^M > 2W$; note that $f$ is $\langle e',e''\rangle$-periodic and obeys the alternating property $f(x+e) = W-f(x)$ for all $x \in G$.  We can therefore write
$$ b(\tilde \alpha_{1, \epsilon_1}(x) + \dots + \tilde \alpha_{W, \epsilon_W}(x)) = \beta_1(x) + \dots + \beta_{W-2}(x) + b$$
for all $x \in G$, where $\beta_i(x)$ for $i=1,\dots,W-2$ is defined by
$$ \beta_i(x) \coloneqq b 1_{i < f(x)}$$ 
if $f(x) < W/2$ and
$$ \beta_i(x) \coloneqq b 1_{i \geq W-f(x)}$$
if $f(x) > W/2$.  The reason for this rather complicated choice of $\beta_i$ is so that $\beta_i$ becomes a $(e, \{0,b\})$-boolean and $\langle e',e'' \rangle$-invariant function (in particular one has $\beta_i(x+e) = b-\beta_i(x)$ for all $x \in G$).  It is then a routine matter to verify that $(\alpha_1,\dots,\alpha_W,\beta_1,\dots,\beta_{W-2})$ obeys property $P_\Omega$ as required.

Conversely, suppose that $(\alpha_1,\dots,\alpha_W,\beta_1,\dots,\beta_{W-2})$ obeys the property $P^*_\Omega$.  By property (i), we may write $\alpha_i = a_i + b \tilde \alpha_i$ and $\beta_i = c_i + b \tilde \beta_i$ for some $a_i,c_i,b \in \Z/2^M\Z$ with $b$ odd and some $\langle e',e''\rangle$-invariant and $(e,\{0,1\})$-boolean functions $\tilde \alpha_1,\dots,\tilde \alpha_W,\tilde \beta_1,\dots,\tilde \beta_{W-2} \colon G \to \{0,1\}$.  Inserting these representations into (ii), we see that there exists a constant $z_0 \in \Z/2^M\Z$ such that
$$ \tilde \alpha_{1, \epsilon_1}(x) + \dots + \tilde \alpha_{W, \epsilon_W}(x) = \tilde \beta_1(x) + \dots + \tilde \beta_{W-2}(x) + z_0$$
for all $x \in G$, where $\tilde \alpha_{i,\epsilon_i}$ is defined by $R_{\epsilon_i}(\tilde \alpha_i)$ and \eqref{reflection}.  Summing over $x$ and $x+e$ and using the $(e,\{0,1\})$-boolean nature of the $\tilde \alpha_{i,\epsilon_i}$ and $\tilde \beta_i$, we conclude that
$$ W = W-2 + 2z_0$$
and hence $z_0$ is equal to $1$ or $2^{M-1}+1$.  On the other hand, since $\tilde \alpha_{i,\epsilon_i}$, $\tilde \beta_i$ take values in $\{0,1\}$, $z_0$ must take values in $\{-W+2,\dots,W\}$ modulo $2^M$.  Since $2^{M-1} > W$, we must therefore have $z_0=1$.  In particular, $\tilde \alpha_{1, \epsilon_1} + \dots + \tilde \alpha_{W, \epsilon_W}$ takes values in $\{1,\dots,W-1\}$, and hence $(\tilde \alpha_{1, \epsilon_1}, \dots, \tilde \alpha_{W, \epsilon_W})$ cannot be $(0,\dots,0)$ or $(1,\dots,1)$.  This implies that $(\alpha_1,\dots,\alpha_W)$ obeys the property $P_\Omega$, and we are done.
\end{proof}

We need a variant of the above proposition which involves a modification of Example \ref{perm-expr} that is compatible\footnote{The simpler clock property from Example \ref{clock} is unsuitable for this purpose due to its incompatiblity with \eqref{alternating}, but the reader is encouraged to think of the periodized permutation property as technical substitute for the clock property.} with the alternating relation \eqref{alternating}.  We again let $G, M, e, e', e''$ be as in Proposition \ref{multi-express}, and suppose that $\alpha_1,\dots,\alpha_W \colon G \to \Z/2^M\Z$ obey the compatible boolean property $P$, so as before we have a representation \eqref{alphai}, unique up to reflection symmetry.  

\begin{definition}[Boolean periodized permutation]\label{boolper}
If $v \in G$, we say that a tuple $(\alpha_1,\dots,\alpha_W)$ is a \emph{boolean periodized permutation} along the direction $v$ if it obeys the compatible boolean property $P$, and for each $x \in G$, the map
$$ j \mapsto (\tilde \alpha_1(x+jv), \dots, \tilde \alpha_W(x+jv))$$
is a bijection from $\{0,\dots,2^W-1\}$ to $\{0,1\}^W$.
\end{definition}
Note that the boolean periodized permutation property is preserved under reflection symmetry and is therefore well-defined.  Comparing this claim with the corresponding claim with $x$ replaced by $x+v$, we see that the boolean periodized permutation property implies in particular that
$$ (\tilde \alpha_1(x+2^W v), \dots, \tilde \alpha_W(x+2^W v)) = (\tilde \alpha_1(x), \dots, \tilde \alpha_W(x))$$
and hence each of the $\tilde \alpha_i$ (or $\alpha_i$) are $\langle 2^W v\rangle$-periodic.  

\begin{proposition}[Expressing a boolean periodized permutation]\label{perm-express}  Let $G$ be a finitely generated abelian group, let $e, e', e''$ be elements of $G$ that generate a copy of $(\Z/2\Z)^3$, let $\Z/2^M\Z$ be a cyclic $2$-group for some $M \geq 2$, let $W \geq 1$, and let $v \in G$.  Then the property of being a boolean periodized permutation along the direction $v$ is a weakly expressible $(G, (\Z/2^M\Z)_{w=1,\dots,W})$-property.
\end{proposition}

\begin{proof}  We can assume that $v$ has order at least $W$, otherwise the property is impossible to satisfy.
We claim that a tuple $(\alpha_1,\dots,\alpha_W)$ obeys the boolean periodized permutation property along $v$ if and only if it obeys the compatible boolean property  $P$ and additionally solves the functional equation
\begin{equation}\label{eqn}
 \biguplus_{j=0}^{2^W-1} (\alpha_1(x+jv),\dots,\alpha_W(x+jv)) + (2\Z/2^M\Z)^W = (\Z/2^M\Z)^W
 \end{equation}
for all $x$.  The equation \eqref{eqn} defines an expressible property by definition (the $jv$ are all distinct as $v$ has order at least $W$), and so the proposition will follow from ths claim, Proposition \ref{multi-express}, and Lemma \ref{closure}.

It remains to verify the claim.  If $(\alpha_1,\dots,\alpha_W)$ is a boolean periodized permutation along $v$, then it obeys $P$, and the $2^W$ tuples $$(\tilde \alpha_1(x+jv), \dots, \tilde \alpha_W(x+jv)), \quad j=0,\dots,2^W-1$$ occupy distinct points in $\{0,1\}^W$. Since $\alpha_w(x+jv) = a_i + b \tilde \alpha_w(x+jv)$ and $b$ is odd, we conclude that the $2^W$ tuples
$(\alpha_1(x+jv),\dots,\alpha_W(x+jv))$ occupy distinct cosets of $(2\Z/2^M\Z)^W$.  Since there are only $2^W$ such cosets, this gives \eqref{eqn}.  The converse implication follows by reversing these steps.
\end{proof}

\section{Programming a Sudoku puzzle}\label{sec:sudoku}

We now combine the various weakly expressible statements described in the previous section to reduce matters to demonstrating aperiodicity of a certain type of ``Sudoku puzzle''.  To define this puzzle we need some notation.

\subsection{A $2$-adically structured function}  

We begin the construction of the ``Sudoku puzzle''.
We henceforth fix a base $q=2^{s_0}$, which will be a sufficiently large but constant power of two (for instance $s_0=10$, $q = 2^{10}$ would suffice).  In particular, the reader should interpret any exceptional set of (upper) density $O(1/q)$ as being negligible in size. We define the \emph{digit set} $\Sigma$ to be the finite set
$$ \Sigma =\Sigma_q \coloneqq (\Z/q\Z) \backslash \{0\}.$$
We need a large width $N$ depending on $q$; one convenient choice to take will be
$$ N \coloneqq q^2,$$
although our arguments would also work \emph{mutatis mutandis} for larger choices of $N$.
We then define the \emph{Sudoku board}
$$ \mathbb{B} \coloneqq \{1,\dots,N\} \times \Z.$$
Elements $(n,m)$ of this board will be referred to as \emph{cells}.  We isolate some collections of cells of relevance to our arguments:
\begin{itemize}
    \item A \emph{column} is a set of cells of the form $\{n\} \times \Z$ for some $1 \leq n \leq N$.
    \item A \emph{non-vertical line} $\ell = \ell_{i,j}$ is a set of cells of the form
$$ \ell_{i,j} \coloneqq \{ (n,jn+i): 1 \leq n \leq N \}$$
    for some \emph{slope} $j \in \Z$ and \emph{intercept} $i \in \Z$.
    \item A \emph{row} is a non-vertical line of slope $0$, that is to say a set of cells of the form $\{1,\dots,N\} \times \{m\}$ for some $m \in \Z$.
    \item A \emph{diagonal} is a non-vertical line of slope $1$, that is to say a set of cells of the form $\{ (n, n+i): 1 \leq n \leq N\}$ for some $i \in \Z$.
    \item An \emph{anti-diagonal} is a non-vertical line of slope $-1$, that is to say a set of cells of the form $\{ (n, i-n): 1 \leq n \leq N \}$ for some $i \in \Z$.
    \item A \emph{square} $Q_{n_0,m_0}$ is a set of cells of the form 
    \begin{equation}\label{square}
    Q_{n_0,m_0} \coloneqq \{n_0,\dots,n_0+7\} \times \{m_0,\dots,m_0+7\}
    \end{equation}
    for some $1 \leq n_0 \leq N-7$ and $m_0 \in \Z$.
\end{itemize}

See Figure \ref{fig:board}.

\begin{figure}
    \centering
    \includegraphics[width = .9\textwidth]{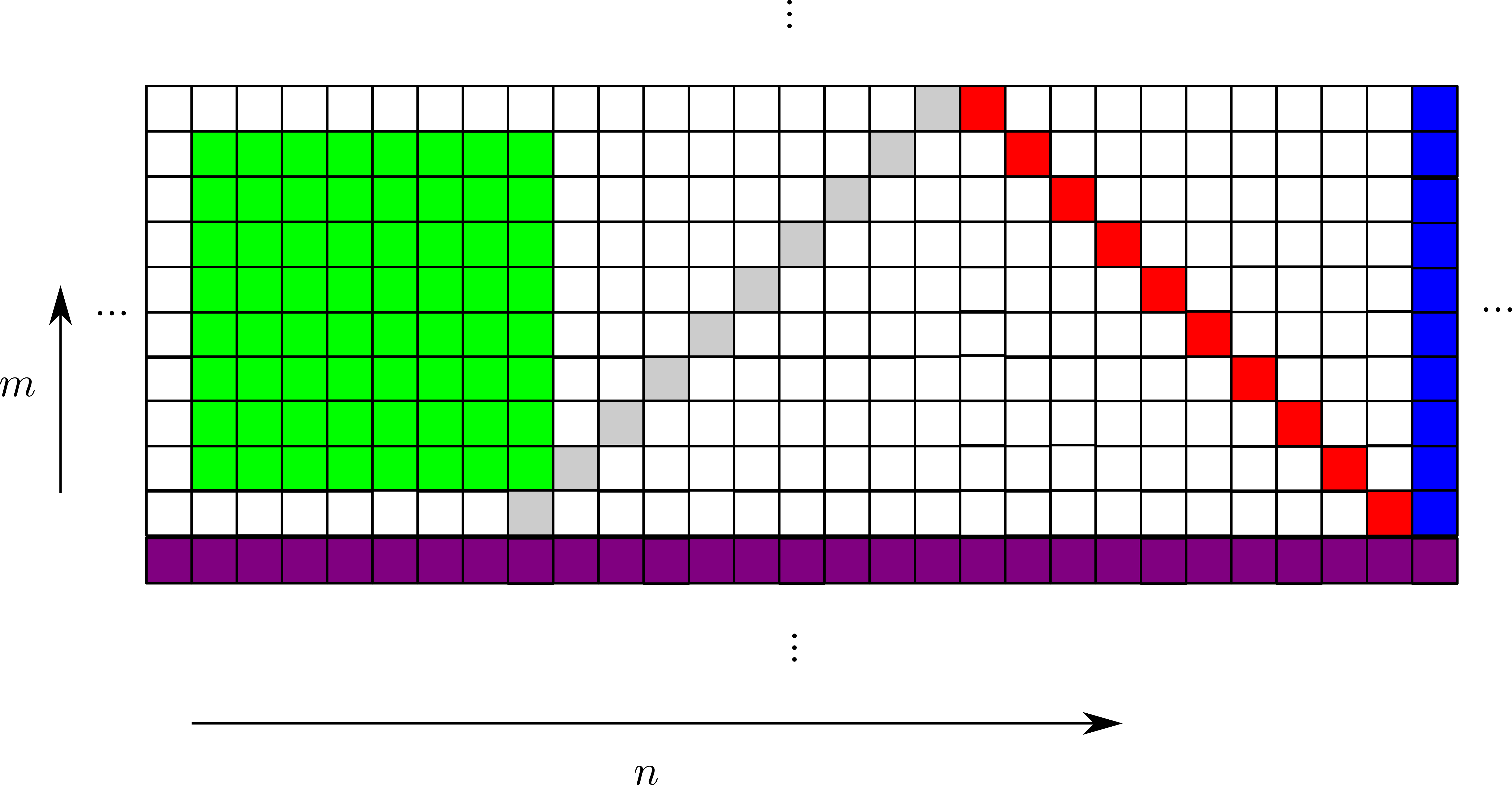}
    \caption{A portion of the Sudoku board $\Omega$, with some selected (overlapping) objects: a column (in blue), a row (in purple), a diagonal (in gray), an antidiagonal (in red), and a square (in green).}
    \label{fig:board}
\end{figure}

The Sudoku puzzle that we will introduce later will be solved by filling in the cells $(n,m)$ of the Sudoku board $\mathbb{B}$ with digits $F(n,m)$ from $\Sigma$ that obey certain permutation-like constraints along the lines of this board.  This may be compared with a traditional Sudoku puzzle, in which the digit set is $\{1,\dots,9\}$, the board is $\{1,\dots,9\}^2$, and the constraints are that the digit assignment is a permutation on every row and column of the puzzle, as well as certain $3 \times 3$ squares in the board, and also agrees with some prescribed initial data on certain cells.  We note, however, that while traditional Sudoku puzzles are designed to have a \emph{unique} solution, the Sudoku puzzle that we will study will have a number of solutions, though all have a similar \emph{$2$-adic structure} as described below.

We now introduce a ``basic $2$-adically structured function'' $f_q \colon \Z \to \Sigma$, defined by the formula
$$ f_q( q^k m ) \coloneqq m \Mod{q}$$
whenever $k \geq 0$ and $m$ is an integer not divisible by $q$, with the (somewhat arbitrary) convention that $f_{q}(0)=1$.  In other words, $f_{q}(n)$ is the last non-zero digit in the base $q$ expansion of $n$, or $1$ if no such digit exists (see Figure \ref{fig:fp}). 
\begin{figure}
    \centering
    \includegraphics[width = 1.1\textwidth]{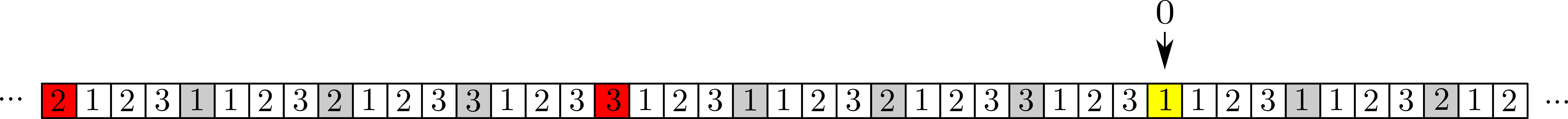}
     \caption{The function $f_q$ for $q=2^2$. The white cells correspond to $n\in  \Z\setminus q\Z$, the gray cells are those with $n\in q\Z\setminus q^2\Z$ and, the red ones have $n\in q^2\Z\setminus q^3\Z$ and  yellow indicates $n=0$. Compare with the example of a limit-periodic pattern in \cite[Figure 3]{einstein}.} 
    \label{fig:fp}
\end{figure}

We observe the functional equations
\begin{equation}\label{mult-1}
f_{q}(q n) = f_{q}(n)
\end{equation}
for all $n \in \Z$
\begin{equation}\label{init}
f_{q}(n) = n \Mod{q}
\end{equation}
when $n$ is not divisible by $q$; indeed, these equations specify $f_{q}$ uniquely except for the value at zero.  We also observe the multiplicativity property 
\begin{equation}\label{mult-2}
f_{q}(a n) = a f_{q}(n)
\end{equation}
whenever $a,n \in \Z$ with $a$ odd and $n$ non-zero.

\begin{remark}
 The function $f_{q} \colon \Z \to \Sigma_q$ is an example of a \emph{limit-periodic function} \cite{God,einstein} (so, in particular, is an \emph{almost periodic function} in the sense of Besicovitch \cite{besicovitch}): for any natural number $r$, $f_{q}$ agrees with a $q^r$-periodic function outside of a single coset $0 + q^r \Z$ of $q^r \Z$, so in particular it agrees with a periodic function outside of a set of arbitrarily small upper density in $\Z$.  For $s_0$ large, this function is also ``approximately affine'' in the sense that it agrees with the affine map $n \mapsto n \Mod{q}$ outside of a single coset $0 + q\Z$ of $q\Z$, which one should view as being a relatively small (though still positive  density) subset of the integers $\Z$.  
 \end{remark}

\begin{remark}\label{2-adics}  One could extend $f_{q}$ to a function $f_{q} \colon \Z_2 \to \Sigma_q$ on the $2$-adics $\Z_2 \coloneqq  \varprojlim_{r \to \infty} \Z/2^r\Z$ (or equivalently, the $q$-adics $\Z_q \coloneqq \varprojlim_{r \to \infty} \Z/q^r \Z$) which is continuous away from the origin (and has a ``piecewise affine'' structure).  As such, we will informally think of the function $f_q$ (as well as various rescaled versions of this function) as having ``$2$-adic structure''.  However, we will not explicitly use the $2$-adic numbers $\Z_2$ in our arguments below, as we did not find that the use of this number system gave any significant simplifications to the argument.
\end{remark}

Our \emph{Sudoku puzzle} is to fill the board $\mathbb{B}$ in such a way that every non-vertical line (but not necessarily every column) is a rescaled version of $f_{q}$.  To make this precise we introduce the following class of finite sequences.
    
    \begin{definition}[A class of $2$-adically structured functions]\label{Sp}  Let $\Scal[N]=\Scal_q[N]$ denote the set of all functions $g \colon \{1,\dots,N\} \to \Sigma$ which take the form
     $$ g(n) = cf_{q}(an+b)$$
     for all $m=1,\dots,N$ and some integers $a,b,c \in \Z$ with $c$ odd.
     \end{definition}

See Figures  \ref{fig:badcoset},  \ref{fig:badcosetorder1}, \ref{fig:nobadcoset13}, \ref{fig:nobadcoset} for some examples of elements of $\Scal[N]$, where we set $q = 2^2$ (and hence $N=16$) in order to make the figures small.  The scaling factor $c$ is of little significance and will often be normalized to $1$ in our arguments.

\begin{figure}
    \centering
    \includegraphics[width = .9\textwidth]{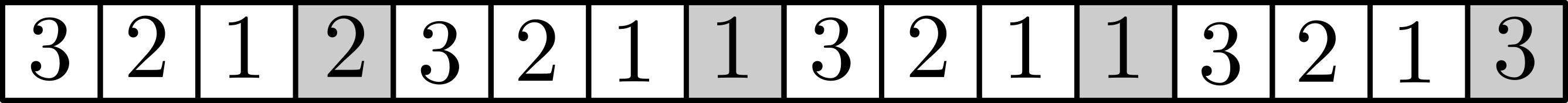}
    \caption{An element $g(n) = f_q(12-n)$ of $\Scal_q[N]$ with $q=4$ and $N=q^2$, depicted as a row of $N$ boxes filled with digits in $\Sigma = \{1,2,3\}$.  In the language of Lemma \ref{stats} below, the step is $s_g = 3 \Mod{q}$, the order $\ord_g$ is zero, the bad coset $\Gamma_g = 4\Z$ is the set of shaded boxes (which in this case has upper density $1/4$), and the associated affine function $\alpha_g(n) = 12 - n \Mod{4}$ vanishes on the bad coset $\Gamma_g$ and agrees with $g$ outside of that coset.}
    \label{fig:badcoset}
\end{figure}

\begin{figure}
    \centering
    \includegraphics[width = .9\textwidth]{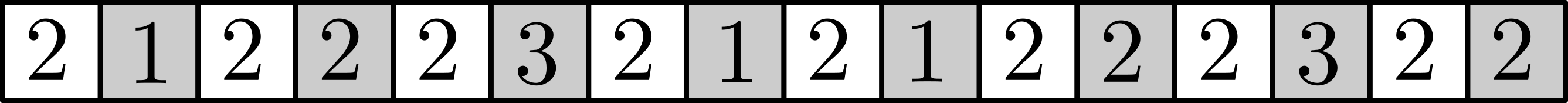}
    \caption{Another element $g(n) = f_4(2(n-8))$ of $\Scal_q[N]$, again with $q=4$ and $N=q^2$.  In the language of Lemma \ref{stats} below, the step is $s_g = 2 \Mod{q}$, the order $\ord_g$ is one, the bad coset $\Gamma_g = 2\Z$ is the set of shaded boxes (which in this case has upper density $1/2$), and the associated affine function $\alpha_g(n) = 2n \Mod{4}$ vanishes on the bad coset $\Gamma_g$ and agrees with $g$ outside of that coset.}
    \label{fig:badcosetorder1}
\end{figure} 

\begin{figure}
    \centering
    \includegraphics[width = .9\textwidth]{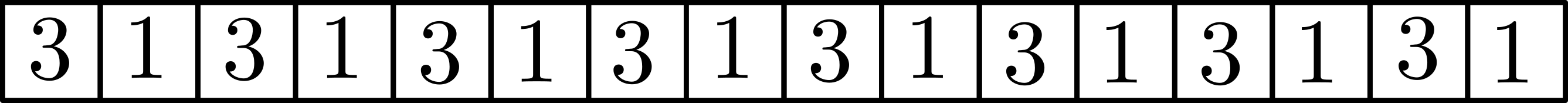}
    \caption{A third element $g(n) = f_4(2n+1)$ of $\Scal_4[16]$.  In the language of Lemma \ref{stats} below, the step is $s_g = 2 \Mod{q}$, the order $\ord_g$ is $-\infty$, the bad coset $\Gamma_g$ is empty (so has upper density $0$), and the associated affine function $\alpha_g(n) = 2n+1 \Mod{4}$ agrees with $g$ everywhere.}
    \label{fig:nobadcoset13}
\end{figure}

\begin{figure}
    \centering
    \includegraphics[width = .9\textwidth]{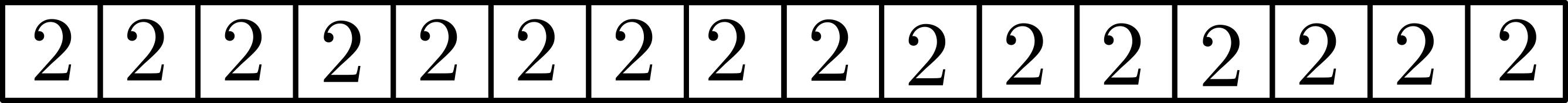}
    \caption{A constant element $g(n)=2 \Mod{4}$ of $\Scal_4[16]$. In the language of Lemma \ref{stats} below, the step is $0$, the order is $-\infty$, the bad coset is empty, and the associated affine function $\alpha_g(n)=2 \Mod{4}$ agrees with $g$ everywhere.}
    \label{fig:nobadcoset}
\end{figure}

We will explore the properties of this class $\Scal[N]$ further in later sections.  For now, we use this class to define our Sudoku puzzle.

\begin{definition}[Sudoku puzzle]\label{sudoku-solution}  Define a \emph{Sudoku solution} to be a function $F \colon \mathbb{B} \to \Sigma$ with the property that for every slope $j \in \Z$ and intercept $i \in \Z$, the function $F_{i,j} \colon \{1,\dots,N\} \to \Sigma$ defined by $F_{i,j}(n) \coloneqq F(n, jn+i)$ lies in the class $\Scal[N]$. (See Figure \ref{fig:sudoku}.)  Informally, $F$ is a Sudoku solution if it is a rescaled copy of $f_q$ along every non-vertical line $\ell_{i,j} = \{ (n,jn+i): 1 \leq n \leq N \}$.

A Sudoku solution is said to have \emph{good columns} if, for every $n=1,\dots,N$, there exists a permutation $\sigma_n \colon \Z/q\Z \to \Z/q\Z$ such that $F(n,m) = \sigma_n(m \Mod{q})$ whenever $\sigma_n(m \Mod{q})$ is non-zero.

A Sudoku solution is \emph{periodic} if the columns $m \mapsto F(n,m)$ is periodic for all $n=1,\dots,N$, and \emph{non-periodic} if at least one of the columns is non-periodic.
\end{definition}

\begin{example}[Standard Sudoku solution]\label{Standard-solution}  The function $F(n,m) \coloneqq f_{q}(m)$ is a Sudoku solution with good columns (in this case, the permutations $\sigma_1,\dots,\sigma_N$ are all equal to the identity permutation.  It is non-periodic. (See Figure \ref{fig:standard}.)
\end{example}

\begin{figure}
    \centering
    \includegraphics[width = .9\textwidth]{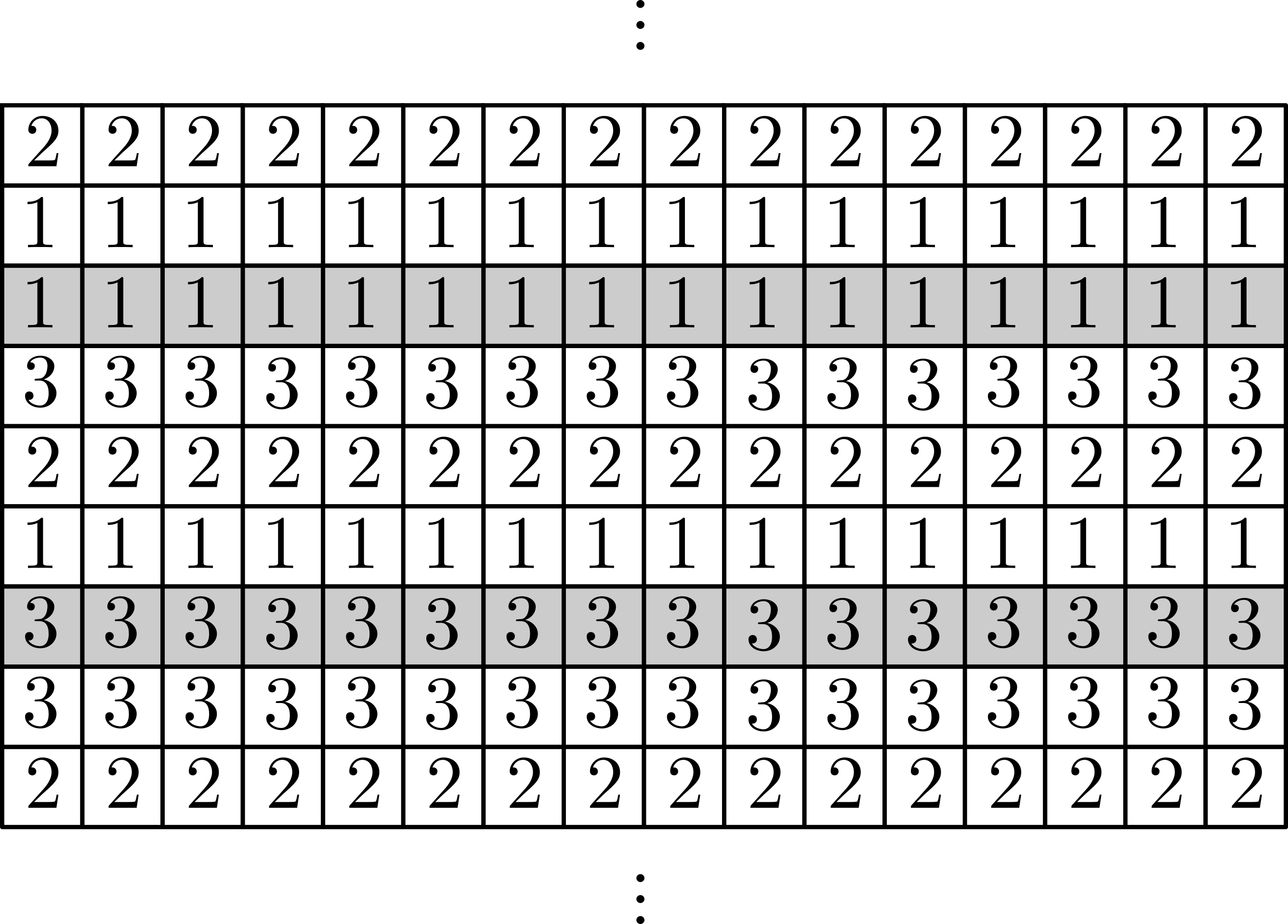}
    \caption{A portion of a standard Sudoku solution (with $q=4$).  Observe that it is affine outside of the shaded cells. This solution is  non-periodic.}
    \label{fig:standard}
\end{figure}

\begin{example}[Constant Sudoku solutions]  If $c \in \Sigma$, then the constant function $F(n,m) \coloneqq c$ is a Sudoku solution which is periodic, but which does not have good columns.
\end{example}

\begin{figure}
    \centering
    \includegraphics[width = .9\textwidth]{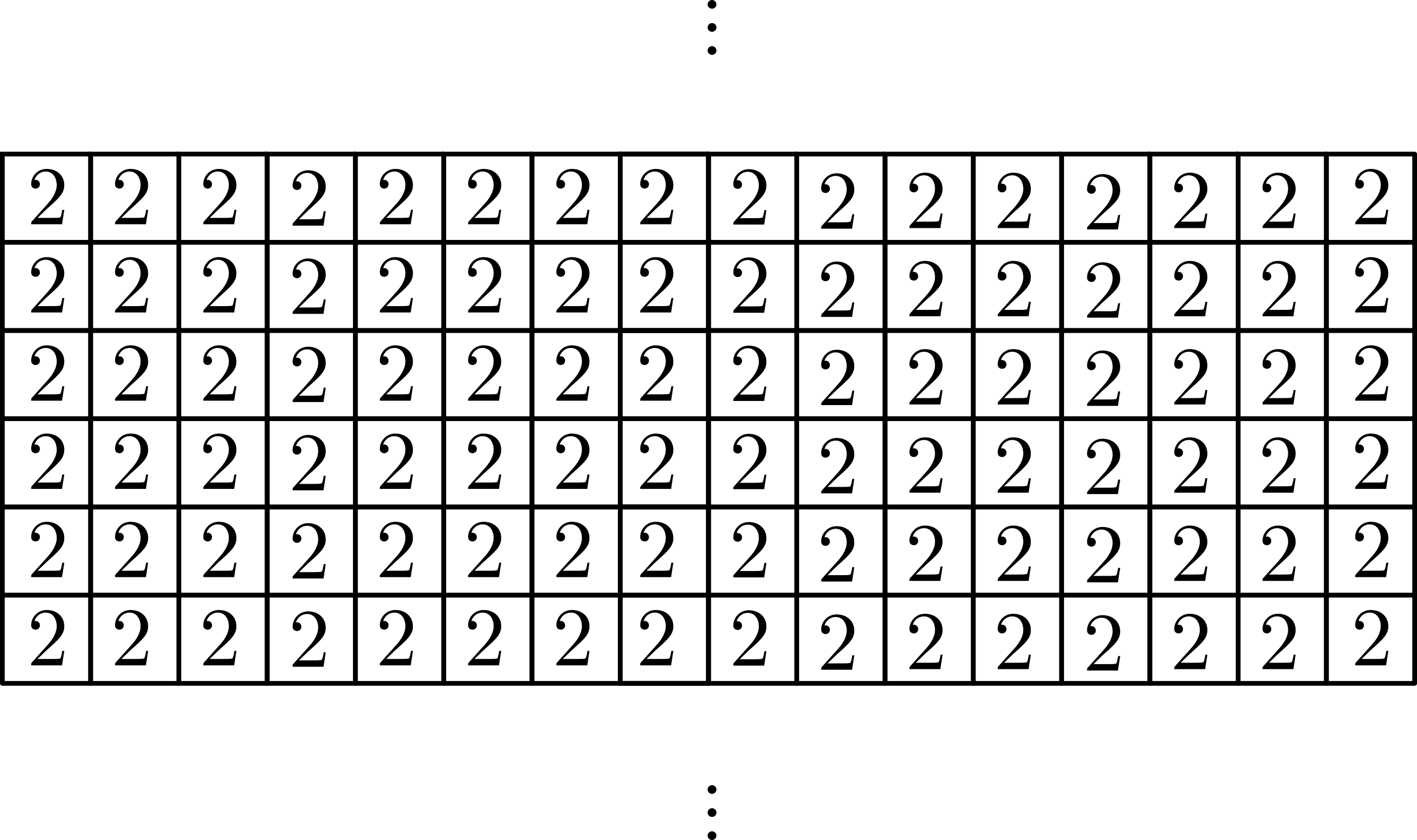}
    \caption{A portion of a constant Sudoku solution (with $c=2$).  Observe that it is affine  and  also periodic.}
    \label{fig:badsudoku}
\end{figure}

\begin{figure}
    \centering
    \includegraphics[width = .9\textwidth]{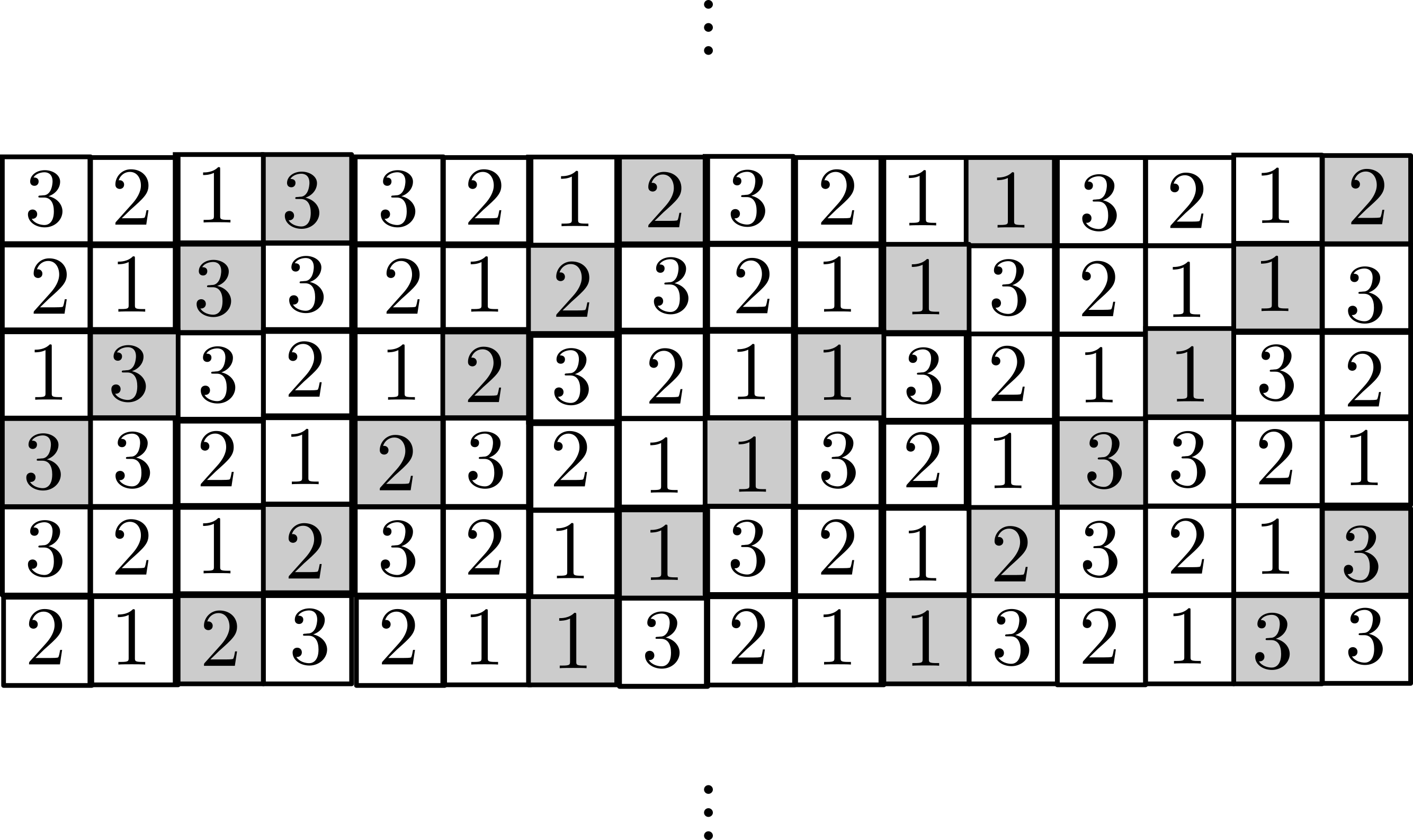}
    \caption{A portion of a Sudoku solution with good columns.  Observe that it is affine outside of the shaded cells, and is also non-periodic.}
    \label{fig:sudoku}
\end{figure}

For future reference we record some simple invariances of Sudoku solutions.

\begin{proposition}[Sudoku invariances]\label{inv}\  
\begin{itemize}
    \item[(i)] (Affine invariance) If $F \colon \mathbb{B} \to \Sigma$ is a Sudoku solution, then for any integers $a,b,c$, the function $(n,m) \mapsto F(n, am+bn+c)$ is also a Sudoku solution.
    \item[(ii)] (Reflection symmetry) If $F \colon \mathbb{B} \to \Sigma$ is a Sudoku solution, then so is the reflection $(n,m) \mapsto F(N+1-n, m)$.
    \item[(iii)] (Homogeneity) If $F \colon \mathbb{B} \to \Sigma$ is a Sudoku solution, then so is $cF$ for any odd $c \in \Z/q\Z$.
\end{itemize}
\end{proposition}

\begin{proof} Claims (i), (ii) are immediate from Definition \ref{sudoku-solution}, while claim (iii) follows from this definition together with Definition \ref{Sp} and \eqref{mult-2}.
\end{proof}

In the remaining sections of the paper, we will prove the non-periodicity of Sudoku solutions with good columns:

\begin{theorem}[Non-periodicity of Sudoku solutions with good columns]\label{main-sudoku}  Let $q = 2^{s_0}$ be sufficiently large.  Then every Sudoku solution with good columns is non-periodic.
\end{theorem}

\begin{remark}  A remarkable feature of this result is that while the property of being a Sudoku solution with good columns is ``local'' in the sense that it can be verified by considering a bounded number of cells of the solution at a time, the conclusion is ``global'' in that it genuinely involves an infinite number of cells, and is not obviously verifiable in a bounded complexity fashion.  Namely, Sudoku puzzles have enough rigidity in them to achieve non-trivial constraints on the solutions, but are not so rigid that they cannot be solved. An analogous claim can be proven for odd primes $q$ as well, and is in fact slightly simpler (for instance, the pseudo-affine functions appearing in Section \ref{conclusion} can be replaced by genuinely affine functions), but we will not be able to use that variant of the above theorem for our purposes, and so leave the details of this variant to the interested reader.
\end{remark}

We assume this theorem for now and show how it implies Theorem \ref{main-aperiod-tuple-weak} (and thus Theorem \ref{main} and Corollary \ref{main-cor}).

\begin{proof}[Proof of Theorem \ref{main-aperiod-tuple-weak} assuming Theorem \ref{main-sudoku}]  We will show that Sudoku solutions with good columns can be encoded as a weakly expressible property. Assuming \thmref{main-sudoku}, this will give an aperiodic weakly expressible property, proving \thmref{main-aperiod-tuple-weak}.   Let $s_0, N$ be such that Theorem \ref{main-sudoku} holds.  We introduce the binary encoding map $B \colon \{0,1\}^{s_0} \to \Z/q\Z$ defined by
$$ B( \epsilon_0, \dots, \epsilon_{s_0-1} ) \coloneqq \epsilon_0 + 2 \epsilon_1 + \dots + 2^{s_0-1} \epsilon_{s_0-1};$$
this is of course a bijection.  

In order to motivate the construction below, we begin with some preliminary calculations.  Suppose one is given a Sudoku solution $F: \mathbb{B} \to \Sigma$ with good columns, thus there exist permutations $\sigma_1,\dots,\sigma_N \colon \Z/q\Z \to \Z/q\Z$ obeying the following properties:
\begin{itemize}
    \item For every $j,i \in \Z$, the function $n \mapsto F(n,jn+i)$ for $n=1,\dots,N$ lies in  $\mathcal{S}[N]$.
    \item One has $F(n,m) = \sigma_n(m \Mod{q})$ whenever $\sigma_n(m \Mod{q})$ is non-zero.
\end{itemize}
We encode this data as a collection of boolean functions $\beta_{a,b,n} \colon \Z^2 \to \{0,1\}$ for $a=0,1$, $b=0,\dots,s-1$, and $n=1,\dots,N$ by enforcing the equations
\begin{equation}\label{binary-0}
B( \beta_{0,0,n}(i,j), \dots, \beta_{0,s-1,n}(i,j)) = \sigma_n(jn+i \Mod{q} )
\end{equation}
and
\begin{equation}\label{binary}
B( \beta_{1,0,n}(i,j), \dots, \beta_{1,s-1,n}(i,j)) = F(n,jn+i)
\end{equation}
for $i,j \in \Z$ and $n=1,\dots,N$.  One then observes the following claims:
\begin{itemize}
    \item[I.] For each $(i,j) \in \Z^2$, the sequence $n \mapsto B( \beta_{1,0,n}(i,j), \dots, \beta_{1,s-1,n}(i,j))$ lies in $\mathcal{S}[N]$.
    \item[II.] If $n=1,\dots,N$ and $(i,j) \in \Z^2$ is such that $B( \beta_{0,0,n}(i,j), \dots, \beta_{0,s-1,n}(i,j))$ is non-zero, then $B( \beta_{1,0,n}(i,j), \dots, \beta_{1,s-1,n}(i,j)) = B( \beta_{0,0,n}(i,j), \dots, \beta_{0,s-1,n}(i,j))$.
    \item[III.]  For each $a=0,1$, $b=1,\dots,s$ and $n=1,\dots,N$, the function $\beta_{a,b,n}$ is $\langle (-n,1)\rangle$-periodic.
    \item[IV.] For each $n=1,\dots,N$, the tuple $(\beta_{0,0,n}, \dots,\beta_{0,s-1,n})$ is a boolean periodized permutation in the direction $(1,0)$. 
\end{itemize}
As it turns out, the converse also holds: if $\beta_{a,b,n} \colon \Z^2 \to \{0,1\}$ are boolean functions for $a=0,1$, $b=0,\dots,s-1$, and $n=1,\dots,N$ obeying the properties I-IV, then they will arise from a Sudoku solution $F$ with good columns via the relations \eqref{binary-0}, \eqref{binary}.  From the machinery established in previous sections, the properties I-IV are essentially weakly expressible (after some technical modifications), and will form the basis of our encoding.

We now turn to the details.  In the space $\{0,1\}^{2s_0N}$ of tuples 
$$(\omega_{a,b,n})_{(a,b,n) \in \Wcal}$$
of boolean variables $\omega_{a,b,n} \in \{0,1\}$ indexed by the $2s_0N$-element set
$$ \Wcal \coloneqq \{0,1\} \times \{0,\dots,s_0-1\} \times \{1,\dots,N\},$$
we define the subset $\Omega$ of those tuples in $\{0,1\}^{2s_0N}$ obeying the following axioms:
\begin{itemize}
    \item[(i)]  (Encoded Sudoku solution)  The sequence $n \mapsto B(\omega_{1,0,n},\dots,\omega_{1,s_0-1,n})$ lies in $\Scal[N]$.
    \item[(ii)]  (Encoded good columns)  If $n=1,\dots,N$ is such that $B(\omega_{0,0,n},\dots,\omega_{0,s_0-1,n}) \neq (0,\dots,0)$, then 
    $B(\omega_{1,0,n},\dots,\omega_{1,s_0-1,n}) = B(\omega_{0,0,n},\dots,\omega_{0,s_0-1,n})$ (or equivalently that $\omega_{0,b,n} = \omega_{1,b,n}$ for $b=0,\dots,s_0-1$).
\end{itemize}
(Compare with the properties I, II discused above.)  The set $\Omega$ is not symmetric, sowe also introduce the symmetrized counterpart $\tilde \Omega \subset \{0,1\}^{1+2s_0N}$ in $\{0,1\}^{1+2s_0N}$ consisting of those tuples $(\omega_*, (\omega_{a,b,n})_{(a,b,n) \in \Wcal})$ such that
$$ (R_{\omega_*}(\omega_{a,b,n}))_{(a,b,n) \in \Wcal} \in \Omega$$ where $R_a$, $a=0,1$ is as in \eqref{reflection}.

We consider the group $G \coloneqq \Z^2 \times (\Z/2\Z)^3$, which contains in particular the three elements $e \coloneqq ((0,0), (1,0,0))$, $e' \coloneqq ((0,0),(0,1,0))$, $e'' \coloneqq ((0,0),(0,0,1))$ which generate a copy of $(\Z/2\Z)^3$.  We now introduce a property $S$, which aims to encode Sudoku solutions with good columns. Let $M$ be a natural number that is sufficiently large depending on $s_0,N$, and let $S$ denote the $(G, (\Z/2^M\Z)^{1+2s_0N})$-property that a tuple  $(\alpha_*, (\alpha_{a,b,n})_{(a,b,n) \in \Wcal})$ of functions $\alpha_*, \alpha_{a,b,n} \colon G \to \Z/2^M\Z$ obeys the following axioms.
\begin{itemize}
    \item[(a)]  $(\alpha_*, (\alpha_{a,b,n})_{(a,b,n) \in \Wcal})$ obeys property $P_{\tilde \Omega}$.
    \item[(b)]  For each $a=0,1$; $b=0,\dots,s_0-1$; $n=1,\dots,N$, the function $\alpha_{a,b,n}$ is $\langle ((-n,1),(0,0,0)) \rangle$-periodic.
    \item[(c)]  For each $n=1,\dots,N$, the tuple $(\alpha_{0,0,n},\dots,\alpha_{0,s_0-1,n})$ is a boolean periodized permutation in the direction $((1,0),(0,0,0))$.
\end{itemize}
The axioms (b), (c) should be compared with the properties III, IV discussed previously.  By Proposition \ref{sym-bool}, Corollary \ref{express-period}, Proposition \ref{perm-express}, and Lemma \ref{closure}, $S$ is a weakly expressible property.  It will thus suffice to show that $S$ is aperiodic.

We first show that there is at least one tuple $(\alpha_*, (\alpha_{a,b,n})_{(a,b,n) \in \Wcal})$ obeying $S$.  Let $F(n,m)$ be a Sudoku solution with good columns (for instance, one can take the standard Sudoku solution from Example \ref{Standard-solution}), and let $\sigma_1,\dots,\sigma_N \colon \Z/q\Z \to \Z/q\Z$ be the associated permutations.  We represent this data via the boolean functions $\beta_{a,b,n} \colon \Z^2 \to \{0,1\}$ for $(a,b,n) \in \Wcal$, defined by the binary encodings \eqref{binary-0}, \eqref{binary}.  By properties III, IV, the $\beta_{a,b,n}$ are $\langle (-n,1)\rangle$-periodic, and the tuple $(\beta_{0,0,n},\dots,\beta_{0,s_0-1,n})$ obeys the boolean permutation property in the direction $(1,0)$ for each $n=1,\dots,N$.  By properties I, II, we see that the tuple $(\beta_{a,b,n}(i,j))_{(a,b,n) \in \Wcal}$ lies in $\Omega$ for each $(i,j) \in \Z^2$.  If we now define the tuple 
$(\alpha_*, (\alpha_{a,b,n})_{(a,b,n) \in \Wcal})$ of functions $\alpha_*, \alpha_{a,b,n} \colon G \to \Z/2^M\Z$ by the formulae
$$ \alpha_*( x, (\epsilon,\epsilon',\epsilon'')) \coloneqq \epsilon$$
and
$$ \alpha_{a,b,n}( x, (\epsilon, \epsilon', \epsilon'')) \coloneqq R_\epsilon(\beta_{a,b,n}(x))$$
for all $x \in \Z^2$, $a=0,1$, $b=0,\dots,s_0-1$, $n=1,\dots,N$,$\epsilon, \epsilon, \epsilon'' \in \{0,1\}$ (where $R_\epsilon$ is as in \eqref{reflection}), it is a routine matter to verify that this tuple obeys property $S$.

Conversely, suppose that $(\alpha_*, (\alpha_{a,b,n})_{(a,b,n) \in \Wcal})$ obeys $S$.  By (a), we can write $\alpha_* = a'_* + b' \tilde \alpha_*$ and $\alpha_{a,b,n} = a'_{a,b,n} + b' \tilde \alpha_{a,b,n}$ for some $a'_*, a'_{a,b,n}, b' \in \Z/2^M\Z$ with $b'$ odd, and some $\langle e',e''\rangle$-periodic $(e,\{0,1\})$-boolean functions $\tilde \alpha_*,\tilde \alpha_{a,b,n} \colon G \to \{0,1\}$ such that
\begin{equation}\label{tax}
(\tilde \alpha_*(\tilde x),(\tilde \alpha_{a,b,n}(\tilde x))_{(a,b,n) \in \Wcal}) \in \tilde \Omega
\end{equation}
for all $\tilde x \in G$.  If we define the functions $\beta_{a,b,n} \colon \Z^2 \to \{0,1\}$ by the formula
$$ \beta_{a,b,n}(x) \coloneqq R_{\tilde \alpha_*(x,(0,0,0))}(\tilde \alpha_{a,b,n}(x,(0,0,0)))$$
for all $(a,b,n) \in \Wcal$, $x \in \Z^2$, we have
\begin{equation}\label{bad}
(\beta_{a,b,n}(x))_{(a,b,n) \in \Wcal} \in \Omega
\end{equation}
for all $x \in \Z^2$.  From axiom (b) we see that each $\beta_{a,b,n}$ is $(-n,1)$-periodic, and from axiom (c) we see that for each $n=1,\dots,N$, the tuple $(\beta_{0,0,n},\dots,\beta_{0,s_0-1,n})$ is a boolean periodized permutation in the direction $(1,0)$.  From the $(-n,1)$-periodicity of the $\beta_{a,b,n}$, we may define functions $F_a \colon \mathbb{B} \to \Z/2^M\Z$ for $a=0,1$ by requiring that
$$ B( \beta_{a,0,n}(i,j),\dots, \beta_{a,s_0-1,n}(i,j) ) = F_a(n, jn+i)$$
for all $n=1,\dots,N$ and $(i,j) \in \Z^2$.  From \eqref{bad} we see that $F_1$ is a Sudoku solution (in particular, it avoids zero and takes values in $\Sigma$), and also that $F_1(n,m) = F_0(n,m)$ whenever $F_0(n,m)$ is non-zero.  Since $(\beta_{0,0,n},\dots,\beta_{0,s_0-1,n})$ is a boolean periodized permutation in the direction $(1,0)$, we see that for all $(n,m) \in \mathbb{B}$, the $q$ points $F_0(n,m),\dots,F_0(n,m+q-1)$ take on distinct values of $\Z/q\Z$, and thus we must have $F_0(n,m) = \sigma_n(m \Mod{q})$ for some permutation $\sigma_n \colon \Z/q\Z \to \Z/q\Z$ (cf. Example \ref{perm-expr}).  Thus $F_1$ has good columns, and is thus non-periodic thanks to Theorem \ref{main-sudoku}.    If $(\alpha_*, (\alpha_{a,b,n})_{(a,b,n) \in \Wcal})$ were periodic, $F_1$ would be periodic. We then conclude that $(\alpha_*, (\alpha_{a,b,n})_{(a,b,n) \in \Wcal})$ is non-periodic.  Thus, property $S$ is aperiodic as required.
\end{proof}

\begin{remark}\label{duality}  The encoding \eqref{binary} resembles the classical projective duality between points and lines in the plane.  Indeed, a non-vertical line $\ell_{i,j} = \{ (n,jn+i): n=1,\dots,N\}$ in the Sudoku board $\mathbb{B}$ corresponds to a point $(i,j)$ in the lattice $\Z^2$; the various boolean expressions $\tilde \alpha_{a,b,n}(i,j,(\epsilon,\epsilon',\epsilon''))$ are encoding different ``bits'' of the a Sudoku puzzle (and its attendant permutations) on this line $\ell_{i,j}$.
\end{remark}

It remains to prove Theorem \ref{main-sudoku}.  This is the objective of the remaining sections of the paper.

\section{Basic properties of $2$-adic structured functions and Sudoku solutions}

We begin by analyzing the class $\Scal[N]$ defined in Definition \ref{Sp}.   We can largely describe the behavior of an element $g$ of $\Scal[N]$ by some statistics which we call the ``order'', ``step'', ``bad coset'', and ``associated affine function'' of $g$. 

\begin{lemma}[Statistics of a $2$-adic function]\label{stats}  To every $g \in \Scal[N]$ one can find an \emph{order} $\ord_g \in \{-\infty,0,\dots,s_0-1\}$, a \emph{step} $s_g \in \Z/q\Z$, a \emph{bad coset} $\Gamma_g \subset \Z$, and an \emph{associated affine function} $\alpha_g \colon \Z \to \Z/q\Z$, obeying the following axioms:
\begin{itemize}
    \item [(i)]  $\alpha_g \colon \Z \to \Z/q\Z$ is not identically zero, and is a function of the form $\alpha_g(n) = s_g n + c_g$ for some $c_g \in \Z/q\Z$, for all $n \in \Z$, thus the step $s_g$ is the slope of the affine function $\alpha_g$.
    \item[(ii)]  The bad coset $\Gamma_g \subsetneq \Z$ is the zero set $\{n \in \Z: \alpha_g(n) = 0\}$ of $\alpha_g$; it is empty if $\ord_g = -\infty$, and is a coset of $2^{-\ord_g} q \Z$ otherwise.  (In particular, the upper density of $\Gamma_g$ is equal to $2^{\ord_g}/q$, and if $\ord_g \geq 0$, then $s_g$ is an odd multiple of $2^{\ord_g}$.)
    \item[(iii)]  One has $g(n) = \alpha_g(n)$ whenever $\alpha_g(n) \neq 0$; in other words, $g$ agrees with the affine function $\alpha_g$ outside of the bad coset $\Gamma_g$.
    \item[(iv)] One can find integers $a, b \in \Z$   such that $\alpha_g(n) = a n + b \Mod{q}$ and $g(n) = f_q(a n + b)$ for all $n \in \Z$.  (In particular, this implies that $a = s_g \Mod{q}$, so if $\ord_g \geq 0$, $a$ is an odd multiple of $2^{\ord_g}$ and $b$ is divisible by $2^{\ord_g}$.)
\end{itemize}
\end{lemma}

See Figures \ref{fig:badcoset}, \ref{fig:badcosetorder1}, \ref{fig:nobadcoset13}, \ref{fig:nobadcoset} for some illustrations of these statistics.
We remark that the elements of $\Scal[N]$ of very high order (close to $s_0$) will be problematic for our analysis, because the bad coset has large upper density in those cases; fortunately, we will be able to show that this case occurs quite rarely for our applications.

\begin{proof} If $g \in \Scal[N]$, then by definition there exist integers $a,b,c$ with $c$ odd such that $g(n) = cf_q(an+b)$.  Since $f_q(0n+0) = f_q(0n+1)=1$, we may assume without loss of generality that $a,b$ do not both vanish.  Noting that $f_q(an+b) = f_q(a(n+q^r)+b)$ for all $n=1,\dots,N$ if $r$ is large enough, we may assume without loss of generality that $an+b$ is non-vanishing on $\{1,\dots,N\}$.  By \eqref{mult-2}, we may replace $a,b,c$ by $ca,cb,1$ and assume without loss of generality that $c=1$.  By \eqref{mult-1} we may assume that $a,b$ are not simultaneously divisible by $q$.

We then set $\alpha_g(n) \coloneqq an+b \Mod{q}$, $s_g \coloneqq a \Mod{q}$, $\Gamma_g \coloneqq \{n \in \Z: \alpha_g(n)=0\}$.  If $\Gamma_g$ is empty we set $\ord_g = -\infty$, otherwise we set $\ord_g$ equal to the largest number of powers of two that divide $a$.  (This order cannot exceed $s_0-1$, otherwise $\alpha_g$ would be constant and non-zero, and so $\Gamma_g$ would be empty).  The verification of the axioms (i)-(iv) is then routine.
\end{proof}

In principle, it is possible that the order $\ord_g$, step $s_g$, bad coset $\Gamma_g$, or associated affine function $\alpha_g$ produced by this lemma are not unique, because there are multiple ways to express $g$ in the form $f_q(an+b)$.  For instance, for $n=1,\dots,N$, the function $f_q(n)$ can also be written as $f_q(n+q^m)$ for any $m$ with $q^m>N$, or as $f_q(q^r n)$ for any $r \geq 1$.  We will be able to exclude this scenario, thanks to (a modification of) the following useful proposition.

\begin{proposition}[Rigidity outside of a bad coset]\label{rigid-out}
Let $\{n_0,\dots,n_0+7\}$ be an interval of length $8$ in $\{1,\dots,N\}$, and let $\alpha \colon \Z \to \Z/q\Z$ be an affine function.  Suppose that $g(n) = f_q(an+b)$ is an element of $\Scal[N]$ such that $g(n) = \alpha(n)$ whenever $n \in \{n_0,\dots,n_0+7\}$ is such that $\alpha(n) \neq 0$.  Then in fact we have $g(n) = \alpha(n)$ whenever $n \in \{1,\dots,N\}$ and $\alpha(n) \neq 0$.
\end{proposition}

We caution that while the conclusion of this proposition strongly suggests that $\alpha_g = \alpha$, and the proof below will support this claim in most cases, there are a few cases in which this is not actually true.  For instance, if $g(n) = f_q(0n + \frac{q}{2}) = \frac{q}{2}$, then $g$ agrees with the affine function $\alpha(n) \coloneqq \frac{q}{2} n \Mod{q}$ whenever $\alpha(n) \neq 0$, but $\alpha_g(n) = \frac{q}{2}$ is not the same function as $\alpha$.  

\begin{proof} By definition, we can write (possibly non-uniquely)
$$ g(n) = f_q (an+b)$$
 and $\alpha_g(n) = a n +  b \Mod{q}$ for some $a,b\in\Z$, for all $n \in \Z$.  

We may assume that $\alpha$ does not vanish identically, as the claim is vacuously true otherwise.  We set $\Gamma \coloneqq \{n \in \Z: \alpha(n)=0\}$ to be the zero set of $\alpha$; this is either empty, or a coset of $2^j\Z$ for some $1 \leq j \leq s_0$.

First suppose that we can find elements $n,m \in \{n_0,\dots,n_0+7\}$ of different parity that lie outside $\Gamma \cup \Gamma_g$.  Then the affine functions $\alpha$ and $\alpha_g$ both agree at $n,m$; since the odd number $m-n$ is invertible in $\Z/q\Z$, this implies that $\alpha = \alpha_g$, and hence $g(n) = \alpha_g(n) = \alpha(n)$ whenever $n \in \{1,\dots,N\}$ lies outside $\Gamma_g = \Gamma$, giving the claim in this case.

Thus the only remaining cases are when $\Gamma \cup \Gamma_g$ occupies at least one full coset of $2\Z$.  There are three ways this can happen: either $\Gamma$ is a coset of $2\Z$, $\Gamma_g$ is a coset of $2\Z$, or $\Gamma_g = 4\Z + c$ and $\Gamma = 4\Z+c+2$ for some $c$.

We first exclude the latter case.  Without loss of generality we may place $c \in \{n_0,\dots,n_0+3\}$.  By hypothesis, we have $g(c) = \alpha(c)$ and $g(c+4) = \alpha(c+4)$; since $\alpha$ vanishes on $4\Z+c+2$, we conclude that $g(c) =  g(c+4) = \frac{q}{2} \Mod{q}$.  On the other hand, as $c \in \Gamma_g$, by \lemref{stats} (iv), we have  that $\ord_g=s_0-2$ (as $\Gamma_g$ to be a coset of $4\Z$), $a=\frac{q}{4}a'$ is an odd multiple of $\frac{q}{4}$
and we can write 
$ac+b = q m$ for some integer $m$ with $g(c) = f_q(m)$ and $g(c+4) = f_q(m+a')$.  As $a'$ is odd, by \eqref{init} we conclude that at least one of $g(c), g(c+4)$ is odd, a contradiction.  Hence this case cannot occur.

Now suppose that $\Gamma = 2\Z+c$ is a coset of $2\Z$.  We divide into two subcases:
\begin{itemize}
\item If $\Gamma_g \backslash \Gamma$ is empty or contained in a coset of $8\Z$, then $g(n) = \alpha(n) = \alpha_g(n)$ for at least three of the four points in $n \in \{n_0,\dots,n_0+7\} \cap (2\Z+c+1)$.  But $\alpha(n) = \frac{q}{2} \Mod{q}$ on these points, hence $\alpha_g$ equals $\frac{q}{2} \Mod{q}$ on these points also.  As $\alpha_g$ is affine, we conclude that $\alpha_g$ equals $\frac{q}{2} \Mod{q}$ on all of $2\Z+c+1$, and the claim follows.
\item If $\Gamma_g$ is a coset $2^i\Z + c'$ disjoint from $\Gamma = 2\Z+c$ for some $i=1,2$, then $\ord_g = s_0-i$, $a = 2^{\ord_g} a'$ is an odd multiple of $2^{\ord_g} = q/2^i$, and we may normalize $c' \in \{n_0,n_0+1,n_0+2,n_0+3\}$, hence by hypothesis
$$g(c')=\alpha(c');\quad g(c'+2^i)=\alpha(c'+2^i).$$  As $c' \in \Gamma_g$, we may write $$\frac{q}{2^i}a'c'+b=a c' + b = qm$$ for some integer $m$ and odd $a'$, with $g(c') = f_q(m) = \alpha(c')$ and $g(c'+2^i) = f_q(m+a') = \alpha(c'+2^i)$.  Since $a'$ is odd, by \eqref{init} at least one of $f_q(m), f_q(m+a')$ is odd; on the other hand, $\alpha$ is equal to $\frac{q}{2} \Mod{q}$ on $2\Z+c+1$, and hence at $c',c'+2^i$, giving a contradiction.
\end{itemize}

Finally, suppose that $\Gamma$ is not a coset of $2\Z$, but $\Gamma_g = 2\Z+c$ is.  We again divide into two subcases:
\begin{itemize}
    \item If $\Gamma \backslash \Gamma_g$ is empty or contained in a coset of $8\Z$, then by arguing as before we see that the affine function $\alpha$ equals $\frac{q}{2} \Mod{q}$ on at least three of the four points in $n \in \{n_0,\dots,n_0+7\} \cap (2\Z+c+1)$, and hence on all of $2\Z+c+1$.  Since $\Gamma$ is not a coset of $2\Z$, this forces $\alpha$ to be the constant function  $\frac{q}{2} \Mod{q}$, and hence $g$ is equal to $\frac{q}{2} \Mod{q}$ on all of $\{n_0,\dots,n_0+7\}$.  In particular, $g$ is even on $\{n_0,\dots,n_0+7\} \cap (2\Z+c)$ However, if we normalize $c \in \{n_0,n_0+1\}$, we observe that $\ord_g=s_0-1$, $a = \frac{q}{2} a'$ is an odd multiple of $\frac{q}{2}$, and $a c + b = qm$ for some integer $m$, with $g(c)=f_q(m)$ and $g(c+2)=f_q(m+a')$.  Thus at least one of $g(c), g(c+2)$ is odd, a contradiction.
    \item If $\Gamma$ is a coset $4\Z+c'$ disjoint from $\Gamma_g = 2\Z+c$, then $\alpha$ is always a multiple of $\frac{q}{4}$, so in particular $g$ is even on $\{n_0,\dots,n_0+7\} \cap (2\Z+c)$.  Now we can argue as in the previous case to obtain a contradiction.
\end{itemize}
\end{proof}

A variant of the argument gives

\begin{proposition}\label{n8} If $N \geq 8$, then an element $g$ of $\Scal[N]$ has a well-defined order, step, affine function, and bad coset.
\end{proposition}

\begin{proof}  Suppose $g \in \Scal[N]$ has two representations $g = f_q(a_1 n + b_1) = f_q(a_2 n + b_2)$ with associated orders $\ord_1, \ord_2$, steps $s_1, s_2$, affine functions $\alpha_1, \alpha_2$, and bad cosets $\Gamma_1, \Gamma_2$.   Our task is to show that $\ord_1=\ord_2$, $s_1=s_2$, $\alpha_1=\alpha_2$, and $\Gamma_1=\Gamma_2$.

First suppose that we can find elements $n,m \in \{1,\dots,N\}$ of different parity that lie outside $\Gamma_1 \cup \Gamma_2$.  The arguments in the proof of Proposition \ref{rigid-out} (with $\alpha_1, \alpha_2$ playing the roles of $\alpha, \alpha_g$ respectively, and similarly for $\Gamma_1,\Gamma_2$ and $\Gamma, \Gamma_g$)  imply that $\alpha_1=\alpha_2$, which implies that the steps $s_1=s_2$ agree, and that the bad cosets $\Gamma_1=\Gamma_2$ (which are the zero sets of $\alpha_1=\alpha_2$) agree.  Since the upper density of $\Gamma_i$ is $2^{\ord_i}/q$, we then conclude that $\ord_1=\ord_2$, as claimed.

On repeating the rest of the analysis in the proof of Proposition \ref{rigid-out}, we see that the only other case that does not lead to a contradiction is if $\Gamma_1$ is a coset $2\Z+c$ of $2\Z$ and $\alpha_2 = \frac{q}{2}\Mod{q}$ on the complementary coset $2\Z+c+1$.  Thus $\alpha_2$ is either equal to $\alpha_1$, or the constant $\frac{q}{2}$.  In the former case we are done as before.  In the latter case, $\Gamma_2$ is empty, and now we can obtain a contradiction by interchanging the roles of $\Gamma_1$ and $\Gamma_2$ and appealing again to the analysis in the proof of Proposition \ref{rigid-out}.
\end{proof}

We remark that the statistics $s_g, \ord_g, \Gamma_g, \alpha_g$ of an element $g$ of $\Scal[N]$ do
not uniquely determine $g$, because there is still some variability of $g$ within the bad coset $\Gamma_g$.  For instance, the elements $n \mapsto f_q(n)$ and $n \mapsto f_q(n+q)$ of $\Scal[N]$ are both of step $1$ and order $0$ with bad coset $q\Z$ and associated affine function $n \mapsto n \Mod{q}$, but disagree inside of the bad coset.  Nevertheless, the statistics $s_g, \ord_g, \Gamma_g, \alpha_g$ still give some partial constraints on the behavior of $g$ on the bad coset $\Gamma_g$; for instance, all elements $g \in \Scal[N]$ with step $1$, order $0$, bad coset $q\Z$ and associated affine function $n \mapsto n \Mod{q}$ must take the form $g(n) = \frac{n}{q} - c \Mod{q}$ for all $n$ in the bad coset $\{1,\dots,N\} \cap q\Z$ except at a single point $n=cq$ for some $c=1,\dots,q$.  It is this partial control inside the bad coset which will ultimately allow us to conclude the aperiodicity required in Theorem \ref{main-sudoku}.

\section{The structure of Sudoku solutions}\label{conclusion}

\begin{figure}
\centering
    \includegraphics[width = .9\textwidth]{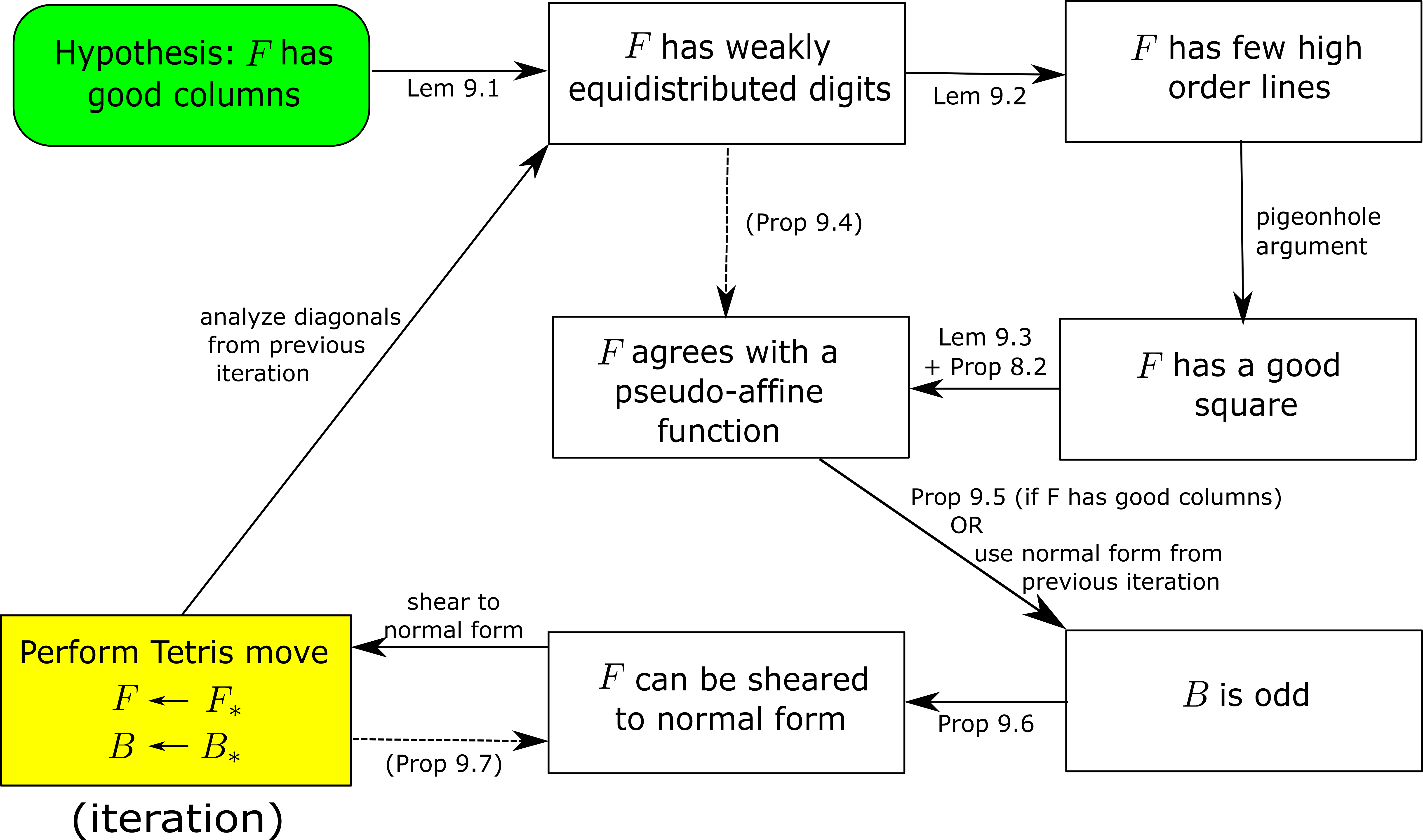}
    \caption{A schematic description of how enough structure is obtained on a Sudoku solution $F$ with good columns that one can eventually conclude that solutions are aperiodic and prove Theorem \ref{main-sudoku}.  Dashed arrows indicate implications that are essentially compositions of other arrows in the diagram.  At a key step in the argument (depicted by the yellow box) the analysis shifts from a Sudoku solution $F$ to its ``post-Tetris move'' version $F_*$ (after first shearing to normal form).  Among other things, this move reduces any putative period in the solution by a factor of $q$.}
  \label{fig:structure}
\end{figure} 

We are now ready to prove Theorem \ref{main-sudoku}.  The logical structure of the argument is summarized in Figure \ref{fig:structure}.

To use the property of having good columns, we begin with the following lemma.  By definition, if the digits of a Sudoku solution $F \colon \mathbb{B} \to \Sigma$ were perfectly equidistributed, then each digit would occur on a set of density $\frac{1}{q-1}$ (as defined in Section \ref{notation-sec}).  It will be convenient to work with a weaker variant of this property.  We say that $F$ has \emph{weak digit equidistribution} if each digit $\sigma \in \Sigma$ occurs in the solution with upper density at most $\frac{2}{q}$ in $\mathbb{B}$.

\begin{lemma}[Good columns implies weak digit equidistribution]\label{digitdensity} Every Sudoku solution with good columns has weak digit equidistribution.
\end{lemma}

\begin{proof}  Let $F \colon \mathbb{B} \to \Sigma$ be a Sudoku solution with good columns. By the triangle inequality, it suffices to verify the claim for each separate column $\{n\} \times \Z$, that is to say to show that
$$ \limsup_{M \to \infty} \frac{1}{2M+1} | \{ m \in \{-M,\dots,M\}: F(n,m) = \gamma \}| \leq \frac{2}{q}$$
for each $n=1,\dots,N$.  By the good column property, there is a permutation $\sigma_n \colon \Z/q\Z \to \Z/q\Z$ such that $F(n,m) = \sigma_n(m \Mod{q})$ whenever $\sigma_n(m \Mod{q}) \neq 0$.  Thus the property $F(n,m)=\gamma$ can only occur in two cosets of $\Z/q\Z$, the coset $\sigma_n^{-1}(\{\gamma\})$ and the coset $\sigma_n^{-1}(\{0\})$, and the claim follows.
\end{proof}

For each line $\ell_{i,j} = \{ (n,jn+i): 1 \le n \leq N \}$ in the Sudoku board, we have the associated element $F_{i,j}$ of $\Scal[N]$ defined by
$$ F_{i,j}(n) \coloneqq F(n,jn+i).$$
In particular we have an associated order $\ord_{F_{i,j}} \in \{-\infty, 0, 1, \dots, s_0-1\}$ of a line to be the order of the associated sequence $n \mapsto F(n,jn+i)$.  We have the following bound on the density of lines of high order:

\begin{lemma}[Weak digit equidistribution implies high-order lines are rare]\label{rare}  Suppose that $F \colon \mathbb{B} \to \Sigma$ is a Sudoku solution with weak digit equidistribution.  Then, for non-negative order $0 \leq o \leq s_0-1$, and any slope $j$, the set $\{ i \in \Z: \ord_{F_{i,j}} = o \}$ has upper density at most $2^{-o+1}$ in $\Z$.
\end{lemma}

\begin{proof}  If $i$ is such that $\ord_{F_{i,j}} = o$, then there is an affine function $n \mapsto 2^o(an+b)$ with $a,b \in \Z/q\Z$ and $a$ odd, such that $F_{i,j}(n) = 2^o(an+b)$ whenever $2^o(an+b) \neq 0$.  In particular, $F(n,jn+i)$ attains the value $q/2 \Mod q$ at least $2^o N/q = 2^o q$.  On the other hand, by the weak digit equidistribution assumption and the triangle inequality the set of $(n,i) \in \{1,\dots,N\} \times \Z$ for which $F_{i,j}(n) = q/2 \Mod q$ has upper density at most $2/q$ in $\{1,\dots,N\} \times \Z$. The claim then follows from a standard double counting argument.
\end{proof}

Once one knows that high-order lines are rare, the function $F$ becomes mostly affine along horizontal lines, diagonals, and anti-diagonals.  One can then expect to ``concatenate'' this information together (in the spirit of \cite{tao-ziegler}) to conclude that $F$ is in fact mostly a two-dimensional affine function $F(n,m) = An+Bm+C$.  This is almost correct, but in our $2$-adic setting there is an additional technicality, in that a small amount of quadratic behavior is also permitted.  More precisely, define a \emph{pseudo-affine function} on $\Z^2$ to be a function $\Psi \colon \Z^2 \to\Z/q\Z$ that is of the form
\begin{equation}\label{psid}
\Psi(n,m) = An+Bm+C + D \frac{q}{4} m(m-n)
\end{equation}
for some coefficients $A,B,C,D \in \Z/q\Z$; see Figure \ref{fig:pseudo}.  
Observe that such functions are affine along infinite non-vertical lines $\overline{\ell}_{i,j} \coloneqq \{ (n,jn+i): n \in \Z\}$, since
$$ \Psi(n,jn+i) = An + Bjn + Bi + C + D \frac{q}{4} (2nij-in+i^2) + D \frac{q}{2} n \binom{j}{2}$$
thanks to the identity $$\frac{q}{2} n^2 -\frac{q}{2} n =q \binom{n}{2}= 0\Mod{q}.$$  The quadratic term $D \frac{q}{4} m(m-n)$ in the definition of a pseudo-affine function is unfortunately necessary, but plays only a minor technical role in the analysis (for $q$ large enough) and we recommend that the reader ignore these terms at a first reading.  The most important coefficient of a pseudo-affine function $\Psi$ will be the vertical coefficient $B$; in particular, the behavior is particularly tractable when $B$ is odd.

\begin{figure}
\centering
    \includegraphics[width = .8\textwidth]{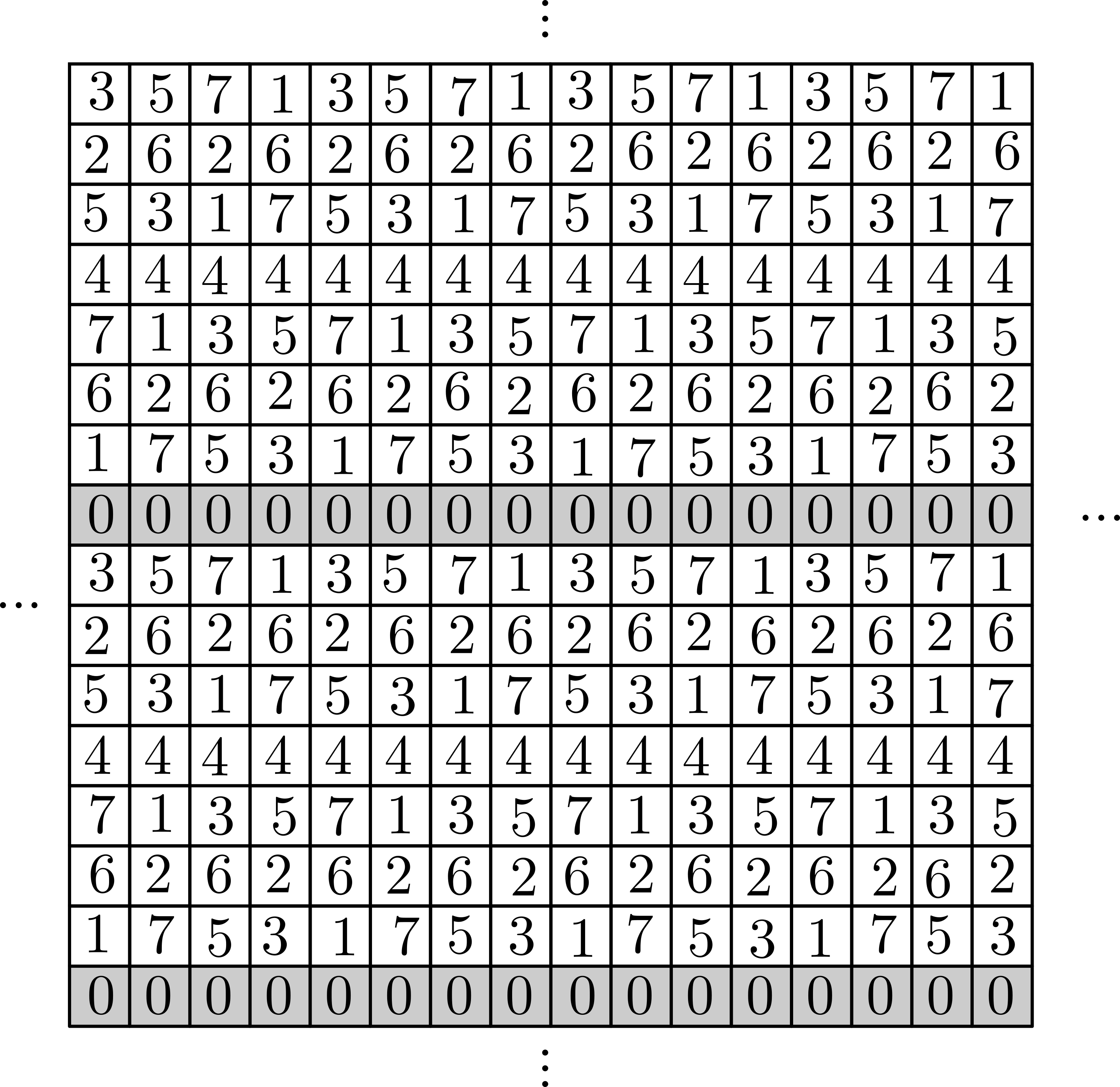}
    \caption{The pseudo-affine function $\Psi(n,m) = m + \frac{q}{4} m(m-n) \Mod{q}$ with $q=8$.  Observe that while $\Psi$ is not affine in a two-dimensional sense, it is affine along all non-vertical lines.  Also, the zero set of $\Psi$ (shaded in grey) remains well-behaved, being equal to $\Z \times q\Z$.  A Sudoku solution that agreed with this pseudo-affine function outside of the grey cells would be in normal form in the sense of Proposition \ref{norm} below, and suitable for applying a ``Tetris'' move for further analysis.}
  \label{fig:pseudo}
\end{figure} 

It is clear that the class of pseudo-affine functions forms an additive group. Note that this group is closed under translations in $\Z^2$, if $\Psi(x)$, $x\in \Z^2$ is a pseudo-affine function and $t\in \Z^2$, then $\Psi(x+t)$, $x\in \Z^2$ is a pseudo-affine function.  We have the following\footnote{See the recent preprint \cite{khare} for some further variations on the theme of this lemma.} concatenation result:  

\begin{lemma}[Concatenation lemma]\label{concat}  Let $F \colon Q \to \Z/q\Z$ be a function defined on an $8 \times 8$ square $Q$ such that for any (infinite) line $\overline{\ell}_{i,j} = \{ (n,jn+i): n \in \Z\}$ with $j = -1,0,1$ (i.e., an infinite anti-diagonal, horizontal line, or diagonal) intersecting $Q$, the function $n \mapsto F(n,jn+i)$ is affine on $\{ n: (n, jn+i) \in Q\}$.  Then there exists a pseudo-affine function $\Psi \colon \Z^2 \to \Z/q\Z$ which agrees with $F$ on $Q$.
\end{lemma}

\begin{proof}  By translation invariance we may normalize $Q = \{0,\dots,7\} \times \{0,\dots,7\}$.  The functions $n \mapsto F(n,0)$ and $n \mapsto F(n,n)$ are affine on $\{0,\dots,7\}$, thus there exist coefficients $A,B,C \in \Z/q\Z$ such that $F(n,m) = An+Bm+C$ for 
\begin{equation}\label{lam}
 (n,m) \in \{(n,0): 0 \leq n \leq 7 \} \cup \{(n,n):0 \leq n \leq 7 \}\}.
 \end{equation}
By subtracting the pseudo-affine function $An+Bm+C$ from $F(n,m)$ we may normalize $A=B=C=0$, thus $F$ now vanishes on the set \eqref{lam}.

The function $n \mapsto F(n,6-n)$ is affine on $\{0,\dots,6\}$ and vanishes at $n=3,6$, hence vanishes on all of $\{0,\dots,6\}$.  In particular $F$ now vanishes at both $(1,1)$ and $(5,1)$.  Since $n \mapsto F(n,1)$ is affine on $\{0,\dots,7\}$, we conclude that $F(n,1) = D \frac{q}{2} (1-n)$ for some $D \in \Z/q\Z$.  By subtracting the pseudo-affine function $D \frac{q}{4} m(m-n)$ from $F$ (which vanishes on \eqref{lam}) we may normalize $D=0$.  Thus $F$ now vanishes on $\{(n,1): 0 \leq n \leq 7 \}$.

For $i=1,\dots,7$, the function $n \mapsto F(n,i-n)$ is affine on $\{0,\dots,i\}$ and vanishes at $n=i-1,i$, hence vanishes on all of $\{0,\dots,i\}$.  In particular, $F(0,m) = F(1,m) = 0$ for all $m=0,\dots,6$.  As $n \mapsto F(n,m)$ is affine on $\{0,\dots,7\}$, we conclude that $F$ now vanishes on $\{0,\dots,7\} \times \{0,\dots,6\}$.  By inspecting $F$ on diagonal and anti-diagonal lines that meet the top row $\{0,\dots,7\} \times \{7\}$ of the square, one can then check that $F$ vanishes here also.  Thus $F$ is identically zero on $Q$, and the claim follows.
\end{proof}

We utilize this lemma as follows.

\begin{proposition}[Weak digit equidistribution implies pseudo-affine structure]\label{near-digit} Suppose that $F \colon \mathbb{B} \to \Sigma$ is a Sudoku solution with weak digit equidistribution.  Suppose that $q$ is sufficiently large.  Then there exists a pseudo-affine function $\Psi \colon \Z^2 \to \Z/q\Z$, which does not vanish on at least one square $\{n_0,\dots,n_0+7\} \times \{m_0,\dots,m_0+7\}$, such that $F(n,m) = \Psi(n,m)$ whenever $(n,m) \in \mathbb{B}$ is a cell with $\Psi(n,m) \neq 0$.
\end{proposition}

\begin{proof}
    Let $M > 100N$ be a sufficiently large parameter (that can depend on $q$) to be chosen later. The first step is to locate a square $Q = \{n_0,\dots,n_0+7\} \times \{m_0,\dots,m_0+7\}$ in $\{1,\dots,N\} \times \{1,\dots,M\}$ with good properties.  The number of possible such squares $Q$ is $(N-7)(M-7)$; we select one at random.
    
To each non-vertical line $\ell_{i,j} = \{ (n,jn+i): 1 \leq n \leq N\}$, one can form the bad set $\Gamma_{i,j} \coloneqq \{ (n,jn+i): n \in \Gamma_{F_{i,j}} \cap \{1,\dots,N\} \}$ associated to the bad coset $\Gamma_{F_{i,j}}$ of $F_{i,j}$.  If the $\ord_{F_{i,j}} = o$, this bad set has spacing $q/2^o$, and thus has cardinality $O(2^o N / q) = O(2^o q)$.  Thus, there are at most $O(2^o q)$ squares $Q$ with the property that $Q$ contains one of the elements of this bad set.  On the other hand, for $j=-1,0,1$ (i.e., horizontal lines, diagonals, and anti-diagonals), we see from Lemma \ref{rare} that the set of intercepts $i$ with $\ord_{F_{i,j}} = o$ have upper density $O(2^{-o})$.  Summing in $o$ and over the $O(M)$ possible lines $\ell_{i,j}$ of slope $j=-1,0,1$ intersecting $\{1,\dots,N\} \times \{1,\dots,M\}$, we conclude from double counting (for $M$ large enough) that the probability that $Q$ contains a bad point from a horizontal line, diagonal, or anti-diagonal intersecting $Q$ is $O(\frac{\sum_{o=0}^{s_0-1} 2^{-o} 2^o q \times M}{(N-7)(M-7)}) = O( \log q / q )$.  Thus, assuming $q$ is large enough, we can find a square 
$$ Q = \{n_0,\dots,n_0+7\} \times \{m_0,\dots,m_0+7\}$$
in $\{1,\dots,N\} \times \{1,\dots,M\}$ with the property that all horizontal lines, diagonals, and anti-diagonals $\ell_{i,j}$ passing through $Q$ are such that $Q \cap \Gamma_{i,j} = \emptyset$.  In particular, on every such line $\ell$, $F$ agrees on $Q \cap \ell$ with a (one-dimensional) affine function.  Applying Lemma \ref{concat}, we may find a pseudo-affine function $\Psi \colon \Z^2 \to \Z/q\Z$ such that $F$ agrees with $\Psi$ on $Q$. In particular, $\Psi$ is non-vanishing on $Q$.

Call a cell $(n,m)$ \emph{good} if either $\Psi(n,m)=0$, or if $\Psi(n,m)=F(n,m)$.  Then all elements of $Q$ are good.  Also, from Proposition \ref{rigid-out} we see that if a line $\ell_{i,j}$ contains eight good consecutive cells, then all the cells in the line are good.  Applying this fact to the eight horizontal lines $\ell_{m,0} = \{ (n,m): 1 \leq n \leq N \}$ for $m \in \{m_0,\dots,m_0+7\}$, we conclude that all the cells in the rectagular region $\{1,\dots,N\} \times \{m_0,\dots,m_0+7\}$ are good.  Applying the fact again to the diagonal lines $\ell_{m,1} = \{ (n,n+m): 1 \leq n \leq M\}$ for $m_0-1 \leq m \leq m_0+8-N$, we conclude that all the cells in the partial horizontal line $\{8,\dots,N\} \times \{m_0+8\}$ are good; applying the fact again to the horizontal line $\ell_{m_0+8,0} = \{ (n,m_0+8): 1 \leq n \leq N \}$, we conclude that in fact all the cells in $\ell_{m_0+8,0}$ are good.  A reflected version of the same argument shows that all the cells in $\ell_{m_0-1,0}$ are good.  Thus we have extended the rectangle of good cells by one row in both directions.  Iterating this argument to fill out the remaining rows of the Sudoku board, we conclude that all the cells in $\mathbb{B}$ are good, giving the claim.
\end{proof}

Assuming good columns, we can obtain an important control on a key coefficient $B$ of the pseudo-affine function $\Psi$.

\begin{proposition}[Odd vertical coefficient]\label{is-odd}  Let $F$ be a Sudoku solution with good columns.  Let $\Psi(n,m) = An + Bm + C + D \frac{q}{4} m(m-n)$ be the pseudo-affine function produced by Proposition \ref{near-digit}.  Then $B$ is odd.
\end{proposition}

\begin{proof}  By \lemref{digitdensity}, $F$ has weak digit equidistribution. Hence, from applying Proposition \ref{near-digit} we obtain that there exists $1 \leq n \leq N$ such that the function $m \mapsto \Psi(n,m)$ is not identically zero.  This function is affine on every coset of $4\Z$, and hence is non-vanishing on at least one coset $8\Z+c$ of $8\Z$.  
Suppose for contradiction that $B$ was even.  Then the function $m \mapsto \Psi(n,m)$ is $\langle \frac{q}{2} \rangle$-periodic on $8\Z+c$, thus $m \mapsto F(n,m)$ is also.  But as $F$ has good columns, we also have $F(n,m) = \sigma_n(m)$ whenever $\sigma_n(m \Mod{q}) \neq 0$, for some permutation $\sigma_n \colon \Z/q\Z \to \Z/q\Z$.  This implies that $\sigma_n$ has a zero in every coset $\{m \Mod{q}, m + \frac{q}{2} \Mod{q}\}$ of $\frac{q}{2}\Z/q\Z$ with $m \in 8\Z+c$, which is absurd.
\end{proof}

This gives us a normal form as follows.  Given a Sudoku solution $F$, define a \emph{shearing} of $F$ to be any map $F' \colon \mathbb{B} \to \Sigma$ of the form
$$ F'(n,m) \coloneqq B F(n, m+An+C)$$
for some integers $A,B,C$ with $B$ odd.  Note from Proposition \ref{inv} that $F'$ is also a Sudoku solution; furthermore, $F$ has good columns if and only if $F'$ does, and $F$ is periodic if and only if $F'$ is.  The property of one Sudoku solution being a shearing of another can also be easily verified to be an equivalence relation.  In view of Remark \ref{duality}, the shear-invariance of Sudoku solutions is closely related to the translation invariance of tiling equations $A \oplus F = G$.

\begin{proposition}[Normal form]\label{norm}  Let $F(n,m)$ be a Sudoku solution which agrees with a pseudo-affine function $\Psi(n,m) = An + Bm + C + D \frac{q}{4} m(m-n)$ with $B \in \Z/q\Z$ odd when $\Psi(n,m)$ is non-zero. Then there exists a shearing $F'$ of $F$ which is in \emph{normal form} in the sense that
\begin{equation}\label{normal}
F'(n,m) = m + D \frac{q}{4} m(m-n)
\end{equation}
for some $D \in \Z/q\Z$, all $n \in \{1,\dots,N\}$, and all $m \in \Z \backslash q\Z$.
\end{proposition}

\begin{proof}  
We claim that the zero set of $\Psi$ takes the form
\begin{equation}\label{0Psi}
    \{ (n,m) \in \Z^2: \Psi(n,m) = 0 \} = \{ (n,m) \in \Z^2: m = A' n + C' \Mod{q} \}
\end{equation} 
for some coefficients $A', C' \in \Z/q\Z$.  To see this, we temporarily divide out by the invertible element $B$ to normalize $B=1$.  If $\Psi(n,m)=0$, then $An+m+C=0 \Mod{4}$, hence
$$ 0 = \Psi(n,m) = An + m + C + D \frac{q}{4} (-An-C) (-An-C-n) \Mod{q}$$
which one can write (using $\frac{q}{2} n^2 = \frac{q}{2} n \Mod{q}$) as
$$ 0 = An + m + C + DC^2 \frac{q}{4} + (2A+1) D C \frac{q}{4} n + D \frac{q}{2} \binom{A+1}{2} n \Mod{q}$$
and thus
$$ m = A' n + C' \Mod{q}$$
where $A' \coloneqq - A -  (2A+1) D C \frac{q}{4} - D \frac{q}{2} \binom{A+1}{2}$ and $C' \coloneqq -C - DC^2 \frac{q}{4}$.  Conversely, if $m = A'n+C' \Mod{q}$ then $An+m+C=0 \Mod{4}$ and $\Psi(n,m)=0$.  This gives \eqref{0Psi}.

Let $F''$ denote the shearing
$$F''(n,m) \coloneqq F(n, m + A'n + C')$$
of $F$, and similarly define $\Psi''(n,m) \coloneqq \Psi(n, m + A'n+C')$.  From direct computation, $\Psi''$ is of the form $\Psi''(n,m) = A'' n + B'' m + C'' + D'' \frac{q}{4} m(m-n)$ for some $A'',B'',C'',D'' \in \Z/q\Z$ with $B''$ odd (and thus invertible), and $\Psi''(n,m)$ vanishes when $m = 0\Mod{q}$.  Substituting $m=0$ we conclude that $A''=C''=0$.  If we then set $F'(n,m) \coloneqq F''(n,m)/B''$ we obtain the desired shearing $F'$ in normal form.
\end{proof}

In view of Propositions \ref{is-odd}, \ref{norm}, we see that to conclude the proof of Theorem \ref{main-sudoku}, it suffices to show that all Sudoku solutions $F$ in normal form are non-periodic.  Suppose for contradiction that we had a periodic Sudoku solution $F$ in normal form, thus $F$ is $\langle (0, M)\rangle$-periodic for some period $M$ (i.e., $F(n,m) = F(n,m+M)$ for all $(n,m) \in \mathbb{B}$).  From the normal form condition \eqref{normal} we see that $M$ must be divisible by $q$.  The key proposition we will establish to conclude the argument is:

\begin{proposition}[Tetris iteration]\label{tetris-iterate}  Let $F$ be a Sudoku solution in normal form.  We consider the \emph{Tetris move} of replacing $F$ with the function
$$ F_*(n,m) \coloneqq F(n, qm)$$
which is also a Sudoku solution thanks to Proposition \ref{inv}.  Then there exists a shearing of $F^*$ that is in normal form.
\end{proposition}

Indeed, if $F$ is an $\langle (0,M)\rangle$-periodic Sudoku solution in normal form, the post-Tetris move solution $F_*$ will be a $\langle (0,M/q) \rangle$-periodic Sudoku solution, and its shearing will be a $\langle (0,M/q) \rangle$-periodic Sudoku solution in normal form.  Iterating this gives an infinite descent of periods $M$, which is absurd.

\begin{remark} In the computer game ``Tetris'', every time a row is completely filled with blocks, it is deleted.  Analogously to this, a Sudoku solution $F$ in normal form has its values completely specified on all rows $\ell_{m,0} = \{(n,m): 1 \leq n \leq N \}$ with $m \neq 0 \Mod{q}$; deleting all these rows yields the post-Tetris move solution $F_*$.  This may help explain our terminology of a ``Tetris move''.
\end{remark}

\subsection{Analyzing the Tetris move}\label{analysis-sec}

It remains to establish Proposition \ref{tetris-iterate}.  In order to deploy tools such as Proposition \ref{near-digit}, we will need to control upper digits of densities of the post-Tetris solution $F_*$.  To do this, we first analyze the  diagonal lines $F_{i,1}(n) = F(n, n+i)$ of the original solution $F$.  From \eqref{normal} we have
$$ F_{i,1}(n) = n+i + D \frac{q}{4} (n+i)i \Mod{q}$$
whenever $n+i \neq 0 \Mod{q}$.  We can simplify this to
$$ F_{i,1}(n) = a_{i,1} n + b_{i,1}$$
whenever $n+i \neq 0 \Mod{q}$, where the coefficients $a_{i,1}, b_{i,1} \in \Z/q\Z$ are given by the formulae
\begin{equation}\label{aform}
a_{i,1} \coloneqq 1 + D\frac{q}{4} i   \Mod{q}
\end{equation}
and
$$ b_{i,1} \coloneqq i + D \frac{q}{4} i^2 \Mod{q}.$$
Observe that $a_{i,1}$ is odd, and $F_{i,1}(n)$ is equal to $a_{i,1} n + b_{i,1}$ for $n \in \{1,\dots,N\}$ outside of the coset $\Gamma_{F_{i,1}}\coloneqq\{n \in \Z: n+i=0\Mod{q}\}$ of $q\Z$.  By Proposition \ref{n8} (applied to some interval $\{n_0,\dots,n_0+7\}$ in $\{1,\dots,N\}$ avoiding $\Gamma_{F_{i,1}}$), this forces the step $s_{F_{i,1}}$ of $F_{i,1}$ to equal $a_{i,1}$, and the order $\ord_{F_{i,1}}$ to equal $0$.

Thus, by \lemref{stats} (iv),  we may write
$$ F_{i,1}(n) = f_q( \tilde a_{i,1} n + \tilde b_{i,1} )$$
for some integers $\tilde a_{i,1}, \tilde b_{i,1}$ with $\tilde a_{i,1} = a_{i,1} \Mod{q}$ and $\tilde b_{i,1} = b_{i,1} \Mod{q}$.  If we now let $n_i \in \{1,\dots,q\}$ be such that $n_i+i = 0 \Mod{q}$, we conclude in particular that $\tilde a_{i,1} n_i + \tilde b_{i,1} = q \tilde c_{i,1}$ for some integer $\tilde c_{i,1}$, and
\begin{equation}\label{eqs}
\begin{split}
    F_*\left( n_i+qj, \frac{n_i+i}{q} + j \right) &= F_{i,1}( n_i + qj ) \\
    &= f_q( q \tilde c_{i,1} + \tilde a_{i,1} qj ) \\
    &= f_q( \tilde a_{i,1} j + \tilde c_{i,1} ).
\end{split}
\end{equation}
for $j=0,\dots,q-1$.  Since $\tilde a_{i,1}$ is odd, this implies that $F_*(n_i+q_j, \frac{n_i+i}{q}+j) = \tilde a_{i,1} j + \tilde c_{i,1}  \Mod q$ for all but one value of $j$.  In particular each digit $\gamma$ of $\Sigma$ is attained by $F_*(n_i+q_j, \frac{n_i+i}{q}+j)$ at most twice.
Averaging over all $i$ and double counting using $N=q^2$, we conclude that the upper density of $\{ (n,m) \in \mathbb{B}: F_*(n,m) = \gamma \}$ in $\mathbb{B}$ does not exceed the upper density of $E$ by more than $2/q$.  In other words, $F_*$ has weak digit equidistribution.

We may now invoke Proposition \ref{near-digit} and conclude that there exists a pseudo-affine function
$$ \Psi_*(n,m) = A_* n + B_* m + C_* + D_* \frac{q}{4} m(m-n)$$
which is not identically zero in $\mathbb{B}$, and such that $F_*(n,m) = \Psi_*(n,m)$ whenever $\Psi_*(n,m)$ is non-zero.

Since $\mathbb{B}$ can be covered by sets of the form $\{ (n_0 + qj, m_0+j): j=0,\dots,q-1\}$ for $n_0=1,\dots,q$ and $m_0 \in \Z$, we can find $n_0 = 1,\dots,q$ and $m_0 \in \Z$ such that $\Psi_*$ does not vanish identically on this set.  By repeating the calculation \eqref{eqs} (with $i = qm_0 - n_0$) we see that
$$ F_*( n_0 + qj, m_0 + j ) = f_q( \tilde a j + \tilde c )$$
for some integers $\tilde a, \tilde c$ (depending on $n_0,m_0$) with $\tilde a$ odd.  In particular,
\begin{equation}\label{fq}
 f_q( \tilde a j + \tilde c ) = \Psi_*( n_0 + qj, m_0+j)
 \end{equation}
whenever $j=0,\dots,q-1$ is such that the right-hand side is non-zero.  

As $\tilde a$ is odd, the left-hand side of \eqref{fq} attains the value $\frac{q}{2} \Mod q$ at most twice for $j=0,\dots,q-1$.  On the other hand, at the midpoint between consecutive values of $j$ in which the (not identically zero) affine function $\Psi_*( n_0 + qj, m_0+j)$ vanishes, this affine function will attain the value of $\frac{q}{2} \Mod q$.  We conclude that $\Psi_*( n_0 + qj, m_0+j)$ vanishes for at most three values of $j=0,\dots,q-1$; meanwhile $\tilde a j + \tilde c \Mod q$ vanishes for one value of $j$.  Hence by the pigeonhole principle, and the assumption that $q$ is large, the identity
$$ \tilde a l + \tilde c = \Psi_*( n_0 + ql, m_0+l) \Mod q;\quad l=j,j+1$$
holds for two consecutive values $l=j,j+1$ of $l$.  Subtracting these two identities and reducing modulo $2$, we conclude that $B_*$ has the same parity as $\tilde a$ and is thus odd.  Applying Proposition \ref{norm}, we conclude that there exists a shearing of $F_*$ that is in normal form.  This concludes the proof of Proposition \ref{tetris-iterate}, and hence of Theorem \ref{main-sudoku}, Theorem \ref{main}, and Corollary \ref{main-cor}.

   	\section{Open problems}

    We close by posing some problems left open by our work.

    \subsection{Explicit bound on dimension}

    The dimension $d$ produced by our proof of Corollary \ref{main-cor1} is explicit but extremely large and probably not optimal.  This is for a number of reasons, the most significant being that we need an enormous number of functional equations in order to encode the property $P_{\tilde \Omega}$ appearing in Section \ref{sec:sudoku}.  Thus, a natural question is
    
    	\begin{question}\label{ques:explicit-dim}
          What is the smallest value of $d$ for which Corollary \ref{main-cor1} (resp. Corollary \ref{main-cor}) is true?
        \end{question}

        The fact that our construction originates in the virtually two-dimensional space $\Z^2\times G_0$ hints that Conjectures \ref{ptc} and \ref{ptc-cts} might fail in a ``reasonably small'' dimension.
        
        On the other hand, there may be hope to extend the known positive results on the periodic tiling conjecture beyond the known cases.

        \begin{question}\label{ques:three-dim}
          Is the Conjecture \ref{ptc} true in $\Z^3$? Is the Conjecture \ref{ptc-cts} true in $\R^2$? 
        \end{question}

        \subsection{Connected tiles}

        An inspection of our proof of Corollary \ref{main-cor} reveals that the tile $\Sigma \subset \R^d$ constructed by the argument is a finite union of cubes; however, this union need not be connected.  Given the positive results available for connected tiles (and in particular for topological disks \cite{ken,err}), it is natural to ask\footnote{Note added in proof: the first author and Kolountzakis have answered this question in the affirmative in \cite{GK}.} 

        \begin{question}\label{ques:conn}
          Is it possible in Corollary \ref{main-cor} to choose $\Sigma$ to be an open connected set?
        \end{question}

     Of course the question can be trivially answered without the requirement that $\Sigma$ is open, simply by adding suitable measure zero line segments to the tile $\Sigma$ constructed by our arguments. Observe that an aperiodic tiling by translations, rotations and reflections of a convex domain in $\R^3$ was constructed by Schmitt--Conway--Danzer  \cite{schmidt} (aka SCD  biprism).

     \subsection{Cardinality of aperiodic tiles} 

     In view of the results in \cite{szegedy}, it might be interesting to compute the size of our tile $F$ in \corref{main-cor1}.

     \begin{question}\label{ques:size}
Suppose that a finite $F\subset \Z^d$ admits an aperiodic tiling. What is the fewest number of prime factors that the cardinality of $F$ can have? 
     \end{question}

    	\subsection{Decidability of tilings}
    	A famous application of the study of the periodicity of tiling is to the problem of determining whether tilings are \emph{decidable}. Namely, the question\footnote{One can also ask, for an individual tile $F$, whether the existence of a tiling $A \oplus F = G$ is logically decidable (i.e., provable or disprovable) in a first-order theory such as ZFC.  The two questions are closely related; see \cite{GT2} for further discussion.} whether there exists an algorithm that, upon any input of a finite set $F$ in a finitely generated abelian group $G$,  computes (in finite time) if this set is a tile of $G$ or not.
    	A well known argument of H. Wang \cite{wang} shows that if any tile admits a periodic tiling then any tiling problem is decidable.

    	In this work we prove that there are tiles of finitely generated abelian groups which tile aperiodically. However, the decidability of tilings by a single tile remains\footnote{Note added in proof: we have answered this question in the affirmative in \cite{undecidable}.} open.
    	\begin{question}\label{ques:decidability}
          Does there exist any undecidable tiling problem with a single tile?
    	\end{question}
    	In a previous paper \cite{GT2} we proved the  undeciability of tilings of \emph{periodic} sets by \emph{two} tiles. This implies, in particular, the existence of aperiodic tilings by two tiles. Our proof consists of encoding any Wang tiling as a tiling of a periodic set with two tiles; then, the undecidability of Wang tilings \cite{Ber,Ber-thesis} implies the existence of an undecidable tiling problem with only two tiles.

   \subsection{Weak periodicity}

   Let $d$ and $F\subset \Z^d$ be as in \corref{main-cor1}.
   Observe that by our construction, all the sets in $\Tile(F;\Z^d) \coloneqq \{ A \subset \Z^d \colon A \oplus F = \Z^d\}$ are $(d-2)$-periodic in the sense that for every $A\in\Tile(F;\Z^d)$ there exist $d-2$ linearly independent vectors $v_1,\dots,v_{d-2}$ in $\Z^d$ such that $A$ is invariant under translations by $v_j$ for every $j=1,\dots,d-2$.  
   \begin{definition}
       A set $S\subset \Z^d$ is called \emph{$k$-weakly periodic} if it can be partitioned into finitely many sets, each of which is $k$-periodic.
   \end{definition}
   It is not difficult to show that if a tile in $\Z^d$  admits a tiling of $\Z^d$ which is $(d-1)$-weakly periodic then it also admits a tiling which is periodic. Thus, our aperiodic construction contains the largest possible amount of periodicity.
   
   In \cite{GT} we showed that for every $F\subset \Z^2$ all the sets in $\Tile(F;\Z^2)$ are $1$-weakly periodic. This, in particular, implies \conjref{ptc} in $\Z^2$.
   
   The following question remains open.
    \begin{question}
        Let $d\geq 3$ and $F\subset \Z^3$ be finite. Are there any $A\in \Tile(F;\Z^d)$ which are \emph{not} $1$-weakly periodic? 
    \end{question}

      \subsection{The structure of our construction}
    We believe that with additional effort, our analysis should give a complete  classification of the space of Sudoku solutions with good columns, and hence also the set of tilings by the tile $F$ in Theorem \ref{main}, however the answer appears to be somewhat complicated\footnote{In particular, the $D$ coefficient in the pseudo-affine functions \eqref{psid} is somewhat difficult to control.} and we do not give it here.
   \begin{problem}
    Find a complete classification of the space $\Tile(F;\Z^2\times G_0)$, where $F$ and $G_0$ are as in \thmref{main}.  What is the dynamical structure of this space (viewed as a topological dynamical system with the translation action of $\Z^2 \times G_0$)?
   \end{problem}
   
   Following \cite{labbe},  it would be of interest to study the tilings in $\Tile(F;\Z^2\times G_0)$ that have a substitution structure.  
   \begin{question}
    Can any of the tilings by our aperiodic tile  be interpreted as a substitution tiling?
   \end{question}
   The $2$-adic nature of the Sudoku solutions suggests a positive answer.

    \end{document}